\documentclass[12pt]{article}
\usepackage{amsmath}
\usepackage{amssymb}
\usepackage{harpoon}
\usepackage{latexsym,color}
\usepackage{ntheorem}
\usepackage{mathrsfs}
\usepackage[numbers,sort&compress]{natbib}

  \allowdisplaybreaks[4]
\title{ \Large The Spatially Homogeneous
Boltzmann Equation for Bose-Einstein Particles: Rate of
Strong Convergence to Equilibrium }
\author{ Shuzhe Cai\footnote{ Department of Mathematical Sciences, Tsinghua
University, Beijing 100084, P.R.China; \,\, e-mail address: csz16@mails.tsinghua.edu.cn }\,\,
and  Xuguang Lu\footnote{Department of Mathematical Sciences,
Tsinghua University, Beijing 100084, P.R.China;  e-mail address: xglu@math.tsinghua.edu.cn
} }

\date{}  

\newcommand{\ld}{\lambda}

\newcommand{\vp}{\varphi}
\newcommand{\vep}{\varepsilon}
\newcommand{\og}{\omega}
\newcommand{\Og}{\Omega}
\newcommand{\sg}{\sigma}

\newcommand{\gm}{\gamma}
\newcommand{\Gm}{\Gamma}
\newcommand{\dt}{\delta}
\newcommand{\Dt}{\Delta}
\newcommand{\fr}{\frac}
\newcommand{\wt}{\widetilde}
\newcommand{\wh}{\widehat}

\newcommand{\bR}{{\mathbb R}^3 }
\newcommand{\bS}{{\mathbb S}^2 }
\newcommand{\bSS}{{\bS}\times{\bS}}

\newcommand{\bRR}{{\bR}\times{\bR}}

\newcommand{\bRS}{{\bR}\times {\mathbb S}^2 }
\newcommand{\bRRS}{{\bRR}\times{\mathbb S}^2 }

\newcommand{\bRd}{{\mathbb R}^d}
\newcommand{\la}{\langle}
\newcommand{\ra}{\rangle}
\newcommand{\mR}{{\mathbb R}}
\newcommand{\mN}{{\mathbb N}}

\newcommand{\mS}{{\mathbb S}}
\newcommand{\be}{\begin{myequation}}
\newcommand{\ee}{\end{myequation}}
\newcommand{\bes}{\begin{myeqnarray}}
\newcommand{\ees}{\end{myeqnarray}}
\newcommand{\beas}{\begin{eqnarray*}}
\newcommand{\eeas}{\end{eqnarray*}}
\newcommand{\lb}{\label}
\newcounter{thm}
\theoremseparator{.}
\newtheorem{theorem}{Theorem}[section]
\newtheorem{proposition}[theorem]{Proposition}
\newtheorem{definition}[theorem]{Definition}
\newtheorem{lemma}[theorem]{Lemma}
\newtheorem{remark}[theorem]{Remark}

\newtheorem{assumption}[theorem]{Assumption}

\newcounter{myequation}[section]

\newenvironment{proof}{{\bf Proof.}}{$\hfill\Box$}
\newenvironment{myequation}{\stepcounter{myequation}\begin{equation}}{\end{equation}}
\newenvironment{myeqnarray}{\stepcounter{myequation}\begin{eqnarray}}{\end{eqnarray}}
\newcommand{\dnumber}{\stepcounter{myequation}}
\begin{document}
\maketitle
\vskip 0.1in \baselineskip 18.2pt
\begin{abstract}
The paper is a continuation of our previous work on the spatially homogeneous
Boltzmann equation for Bose-Einstein particles with quantum collision kernel
that includes the hard sphere model. Solutions $F_t$ under consideration that
conserve the mass, momentum, and energy and converge at least weakly to
equilibrium $F_{{\rm be}}$ as $t\to\infty$ have been proven to exist at least
for isotropic initial data that have positive entropy, and $F_t$ have to be  Borel measures
for the case of low temperature. The new progress is as
follows: we prove that the long time convergence of $F_t(\{0\})$ to the Bose-Einstein
condensation $F_{{\rm be}}(\{0\})$ holds for all isotropic initial data $F_0$ satisfying the low temperature condition.  This immediately implies the long time strong
convergence to equilibrium. We also obtain an algebraic rate of the strong convergence
for arbitrary temperature. Our proofs are based on entropy control, positive lower bound of entropy,  Villani's inequality for entropy dissipation, a suitable time-dependent
convex combination between
the solution and a fixed positive function (in order to deal with
logarithm terms), the convex-positivity of the cubic collision integral, and an iteration technique for obtaining a positive lower bound of condensation.

{\bf Key words}: Bose-Einstein particles, entropy, strong convergence, equilibrium,
low temperature, condensation.

\end{abstract}

\begin{center}\section { Introduction}\end{center}

The quantum Boltzmann equations for Bose-Einstein particles and for Fermi-Dirac particles (which are also called Boltzmann-Nordheim equation, Uehling-Uhlenbeck equation, etc.) were first derived by Nordheim \cite{Nordheim} and Uehling $\&$ Uhlenbeck \cite{Uehling and Uhlenbeck}
 and then taken attention and developed by \cite{weak-coupling},\cite{Chapman and Cowling},\cite{ESY},\cite{LS}. For the case of Bose-Einstein particles and for the spatially homogeneous solutions, the equation under consideration is written
 \be\fr{\partial}{\partial t}f({\bf v},t)=\int_{{\bRS}}B({\bf
{\bf v-v}_*},\og)\big(f'f_*'(1+f)(1+f_*)-ff_*(1+f')(1+f_*')\big) {\rm d}\omega{\rm
d}{\bf v_*}\label{Equation1}\ee
with $({\bf v}, t)\in{\mathbb R}^3\times(0,\infty)$, where the solution $f=f({\bf v},t)\ge 0$ is the  number density of particles at time $t$ with the velocity ${\bf v}$,
and as usual we denote briefly
$f_*=f({\bf v_*},t), f'=f({\bf v'},t),f_*'=f({\bf v_*'},t)$  where ${\bf v},{\bf v_*}$ and ${\bf v'},{\bf v_*'}$ are velocities of two particles before and after
their collision:
\be{\bf v}'={\bf v}- (({\bf v}-{\bf v}_*)\cdot\omega)\omega,\quad {\bf
v}_*'={\bf v}_*+ (({\bf v}-{\bf v}_*)\cdot\omega)\omega, \qquad \omega\in{\mathbb S}^2
\lb{colli}\ee
which conserves the momentum and kinetic energy
\be {\bf v}'+{\bf v}_*'={\bf v}+{\bf v}_*,\quad |{\bf v}'|^2+|{\bf v}_*'|^2=|{\bf v}|^2+|{\bf v}_*|^2.\lb{conser}\ee
The function $B({\bf {\bf v-v}_*},\omega)$ is the collision kernel and is assumed to take the
following general form in order to include possible models:
 \be B({\bf {\bf v-v}_*},\omega)= \fr{1}{(4\pi)^2}|({\bf v-v}_*)\cdot\omega|
\Phi(|{\bf v}-{\bf v}'|, |{\bf {\bf v}-{\bf v}_*'}|)\lb{kernel}\ee where
\be 0\le \Phi\in C_b({\mR}_{\ge 0}^2),\quad  \Phi(r,\rho)=\Phi(\rho, r)\quad \forall\, (r,\rho)\in{\mR}_{\ge 0}^2 \lb{Phi}.\ee
According to \cite{weak-coupling} and \cite{ESY} in the weak-coupling regime, the function $\Phi$  takes the following
form (after normalizing physical parameters)
\be \Phi(r,\rho)=\big(\widehat{\phi}(r)+\widehat{\phi}(\rho)\big)^2,\quad r,\rho\ge 0\label{kernel2}\ee
where $\widehat{\phi}$ is the Fourier transform (in terms of theory of generalized functions ) of a radially symmetric particle interaction potential
$ \phi(|{\bf x}|)\in {\mR}$:
$$\widehat{\phi}(r):=\widehat{\phi}(\xi)|_{|\xi|=r}=
\int_{{\bR}}\phi(|{\bf x}|) e^{-{\rm i}\xi\cdot{\bf x}}{\rm d}{\bf x} \Big|_{|\xi|=r}.$$
In particular if $\phi(|{\bf x}|)=\frac{1}{2}\delta({\bf x})$, where $\delta({\bf x})$ is the three dimensional Dirac delta function concentrating at
${\bf x}=0$, then  $\widehat{\phi}\equiv \fr{1}{2}$ hence $\Phi\equiv 1$ and (\ref{kernel}) becomes the hard sphere model:
\be B({\bf
{\bf v-v}_*},\omega)=\frac{1}{(4\pi)^2}|({\bf v-v}_*)\cdot\omega|\lb{Hard}\ee
which is the only model that has the same form as in the classical Boltzmann equation, and has been mainly concerned in many papers about  Eq.(\ref{Equation1}). In view of physics, the hard sphere model
(\ref{Hard})
can be extended by allowing the interaction potential $\phi({\bf x})$ contains an attractive long-range term $\fr{-1}{2}U(|{\bf x}|)$, i.e.
$\phi(|{\bf x}|)=\fr{1}{2}(\delta({\bf x})-U(|{\bf x}|))$ where $U(|{\bf x}|)\ge 0$ and,
for a technical reason in proving the occurrence of and convergence to the
condensation, we consider such a case that the Fourier transform of $\wh{U}(\xi)$
of $U(|{\bf x}|)$ behaves as
$\widehat{U}(\xi)|_{|\xi|=r}= \fr{1}{1+r^{\eta}}$ with $0<\eta<1$  so that
\be \widehat{\phi}(r)=\fr{1}{2}\big(1-\widehat{U}(\xi)|_{|\xi|=r}\big)= \fr{1}{2}\cdot\fr{r^{\eta}}{1+r^{\eta}}
,\quad r\ge 0\lb{phi}\ee
see Appendix for the existence and positivity of such potentials $U$.

In order to include the hard sphere model and the case (\ref{kernel2}) with (\ref{phi}) we introduce the following assumption:
 \vskip2mm

\begin{assumption}\lb{assp}    $B({\bf {\bf v-v}_*},\omega)$ is given by
(\ref{kernel}),(\ref{Phi}) where $\Phi$ also satisfies
\bes&& r\mapsto \Phi(r,\rho)\,\,\,{\rm is\,\,non-decreasing\,\,on}\,\,{\mR}_{\ge 0}\qquad \forall\, \rho\in{\mR}_{\ge 0}, \lb{Phi*}\\
&& b_0\min\{1,\,(r^2+\rho^2)^{\eta}\}\le \Phi(r,\rho)\le 1\qquad \forall\,
(r,\rho)\in{\mR}^2_{\ge 0} \dnumber \lb {1.6}\ees
for some constants $0<b_0<1, \eta\ge 0.$
\end{assumption}

The monotone assumption (\ref{Phi*}) together with a further restriction $0\le \eta<1/4$ is mainly used to prove the convex positivity of the cubic collision integral (see Proposition \ref{convex-positivity}) and the convergence of
condensation $F_t(\{0\})\to F_{{\rm be}}(\{0\})\,(t\to\infty)$.  Except these,
Assumption \ref{assp} says that the collision kernel $B({\bf {\bf v-v}_*},\omega)$ maybe vanishes at ${\bf {\bf v-v}_*}=0$ but still behaves like the hard sphere model for
large $|{\bf {\bf v-v}_*}|$:
for all ${\bf v},{\bf v}_*\in{\bR}, \og\in {\bS}$
\be  \fr{b_0}{(4\pi)^2}
|({\bf v-v}_*)\cdot\og|\le B({\bf {\bf v-v}_*},\omega)\le \fr{1}{(4\pi)^2}|({\bf v-v}_*)\cdot\omega|\
\quad {\rm with}\,\,\, |v-v_*|\ge 1.\label{kerneB}\ee

Due to the strong nonlinear structure  of the collision integrals and the effect of condensation,
existence and uniqueness of solutions of Eq.(\ref{Equation1}) for anisotropic initial data
have been so far only proven for finite time interval without smallness assumption on the initial data
\cite{Briant-Einav}
and  for global time interval with a smallness assumption on the initial data \cite{LiLu}.
For global in time solutions with general initial data, in particular for the case of low temperature, one has to consider weak solutions $f$ which are solutions of the
following equation
\bes\fr{{\rm d}}{{\rm d}t}\int_{{\bR}}\psi({\bf v})f({\bf v},t){\rm d}{\bf v}&=&
\fr{1}{2}\int_{{\bRRS}}(\psi+\psi_*-\psi'-\psi'_*)B({\bf {\bf v-v}_*},\omega)f'f_*'{\rm d}{\bf v}{\rm d}{\bf v}_*
{\rm d}\og  \qquad \nonumber\\
&
+&\int_{{\bRRS}}(\psi+\psi_*-\psi'-\psi'_*)B({\bf {\bf v-v}_*},\omega)ff'f_*'{\rm d}{\bf v}{\rm d}{\bf v}_*
{\rm d}\og  \qquad \label{weak}\ees
for all test functions $\psi$ and all $t\in [0,\infty)$. Here we have used the fact that the common
quartic terms $f'f_*'ff_*,\, ff_*f'f_*'$ cancel each other:
$$f'f_*'(1+f)(1+f_*)-ff_*(1+f')(1+f_*')=f'f_*'(1+f+f_*)-ff_*(1+f'+f_*').$$
In general however the cubic integral is divergent:
$$\sup_{f}
\int_{{\bRRS}}\Psi({\bf v},{\bf v}',{\bf v}_*')B({\bf v-v}_*,\og)f({\bf v})f({\bf v}')
f({\bf v}_*') {\rm d}{\bf v}{\rm d}{\bf v}_*{\rm d}\og=\infty$$
where $\Psi\in C_b({\mR}^9)$ is an arbitrary positive function and
$f$ under the $\sup$ is taken all nonnegative functions in $C_c^{\infty}({\bR})$
satisfying $\int_{{\bR}}f({\bf v}){\rm d}{\bf v}\le 1$ (see e.g. \cite{Lu2004}).
A subclass of $f$ that has no such divergence is the isotropic (i.e. radially symmetric)
functions: $f({\bf v})=f(|{\bf v}|^2/2)$. By changing variables
$x=|{\bf v}|^2/2, y=|{\bf v}'|^2/2, z=|{\bf v}_*'|^2/2$, one has
$$B({\bf v-v}_*,\og)f(|{\bf v}|^2/2)f(|{\bf v}'|^2/2)
f(|{\bf v}_*'|^2/2) {\rm d}{\bf v}{\rm d}{\bf v}_*{\rm d}\og
=4\pi\sqrt{2}W(x,y,z){\rm d}F(x){\rm d}F(y){\rm d}F(z)$$
where ${\rm d}F(x)=f(x)\sqrt{x}{\rm d}x$, etc., and
$x,y,z\in {\mR}_{\ge 0}$ in the right side are independent variables,
see below for details. This is the main reason that all results obtained so far (except the ones mentioned above) are concerned with isotropic
initial data hence isotropic solutions,
see
e.g.\cite{EMV0}-\cite{JPR},\cite{Lu2000}-\cite{Lu2018},\cite{Nouri},\cite{Semikov and Tkachev1},\cite{Semikov and Tkachev2},\cite{SH}.
Despite this shortage, results obtained so far have
shown that the Eq.(\ref{Equation1}) can be used to describe the formation, transition,
and propagation of the Bose-Einstein condensation of dilute Bose gases at low temperature, see Theorem \ref{theorem1.2} and Proposition \ref{proposition4.6} below, see also (for instance)
 \cite{JPR},\cite{MP2005},\cite{Semikov and Tkachev1},\cite{Semikov and Tkachev2},\cite{SH}  for self-similar structure and deterministic numerical methods;
 \cite{BV},\cite{EMV2},\cite{EV1},\cite{EV2},\cite{Lu2013},
for singular solutions and the formation of blow-up and condensation in finite time;
\cite{Lu2005}, \cite{Lu2016}, \cite{Lu2018} for long time strong and weak convergence to the Bose-Einstein distribution;  \cite{ET} for a linearized model of Eq.(\ref{Equation1}) and rate of convergence to equilibrium; and
\cite{A},\cite{AN},\cite{Nouri} for general discussions and basic results for similar models on low temperature evolution of condensation.

Before stating the main result of the paper we introduce
some notations and the definition of measure-valued isotropic solution of Eq.(\ref{Equation1}).
Let
$L^1_s({\bR})$ with $s\ge 0$ be the linear space of the weighted Lebesgue integrable functions
defined by $L^1_0({\bR})=L^1({\bR})$ and
$$ L^1_s({\bR})=\Big\{f\in L^1({\bR})\,\,\Big|\,\,
\|f\|_{L^1_s}:=\int_{{\bR}}\la {\bf v}\ra^s|f({\bf v})|{\rm d}{\bf v}<\infty\Big\},\quad \la {\bf v}\ra:=(1+|{\bf v}|^2)^{1/2}.$$
Let ${\cal B}_k(X)$ ($k\ge 0$) be the linear space
of signed real Borel measures $F$ on a Borel set $X\subset {\bRd}$ satisfying
$\int_{X}(1+|x|)^k{\rm d}|F|(x)<\infty$,  where
$|F|$ is the total variation of $F$.  Let
$${\cal B}_k^{+}(X)=\{F\in {\cal B}_k(X)\,|\, F\ge 0\}.$$ For the case $k=0$ we also denote
${\cal B}(X)={\cal B}_{0}(X), {\cal B}^{+}(X)={\cal B}_{0}^{+}(X)$.
In this paper we only consider two cases $X={\bR}$ and $X={\mR}_{\ge 0}$,
and in many cases we consider isotropic measures $\bar{F}\in {\cal B}_{2k}({\bR})$, which define and can be defined by
measures $F\in {\cal B}_{k}({\mR}_{\ge 0})$ in terms of the following relations:
\be F(A)=\fr{1}{4\pi\sqrt{2}}\int_{{\bR}}{\bf 1}_{
A}(|{\bf v}|^2/2){\rm d}\bar{F}({\bf v}),\quad A\subset {\mR}_{\ge 0}\lb{m1}\ee
\be \bar{F}(B)=4\pi\sqrt{2}\int_{{\mR}_{\ge 0}}\Big(\fr{1}{4\pi}\int_{{\bS}}{\bf 1}_{
B}(\sqrt{2x}\,\og){\rm d}\og\Big){\rm d}F(x),\quad B\subset {\bR} \lb{m2}\ee
for all Borel measurable sets $A, B$. They are equivalent to functional forms:
\be \int_{{\mR}_{\ge 0}}\vp(x){\rm d}F(x)=
\fr{1}{4\pi\sqrt{2}}\int_{{\bR}}\vp(|{\bf v}|^2/2){\rm d}\bar{F}({\bf v}),\lb{m3}\ee
\be\int_{{\bR}}\psi({\bf v}){\rm d}\bar{F}({\bf v})=
4\pi\sqrt{2}\int_{{\mR}_{\ge 0}}\Big(\fr{1}{4\pi}\int_{{\bS}}\psi(\sqrt{2x}\,\og){\rm d}\og\Big){\rm d}F(x)\lb{m4}\ee
for all bounded Borel measurable functions  $\vp, \psi$.
For any $k\ge 0$ let
\beas\|F\|_k=\int_{{\mathbb R}_{\ge 0}}(1+x)^k{\rm d}|F|(x),\quad F\in {\cal B}_k({\mathbb R}_{\ge 0}).\eeas
The moment $M_p(F)$ of order $p\in [0, k]$ of $F\in {\cal B}_k^{+}({\mathbb R}_{\ge 0})$ is defined by
$$M_p(F)=\int_{{\mathbb R}_{\ge 0}}x^p{\rm d}F(x).$$
Moments of orders $0, 1$ correspond to the mass and energy and are particularly denoted as
$N(F)=M_0(F), E(F)=M_1(F)$, i.e.
$$ N(F)=\int_{{\mathbb R}_{\ge 0}}{\rm d}F(x)
, \quad E(F)=\int_{{\mathbb R}_{\ge 0}}x{\rm d}F(x).$$
Let $C^k_b({\mathbb R}_{\ge 0})$ with $k\in {\mathbb N}$ be the class of bounded continuous functions on
${\mathbb R}_{\ge 0}$ having bounded continuous derivatives on ${\mathbb R}_{\ge 0}$ up to the order $k$.
For isotropic functions $f=f(|{\bf v}|^2/2)\ge 0$,
$\vp=\vp(|{\bf v}|^2/2)$ with $f(|\cdot|^2/2)\in L^1_2({\bR}),
\vp\in  C^2_b({\mathbb R}_{\ge 0})$, and for the measure $F$ defined by ${\rm d}F(x)=f(x)\sqrt{x}{\rm d}x$, the collision integrals in (\ref{weak}) can be rewritten
(see Appendix)
\beas&&
\fr{1}{2}\int_{{\bRRS}}(\vp+\vp_*-\vp'-\vp'_*)Bf'f_*'
{\rm d}{\bf v}{\rm d}{\bf v}_*
{\rm d}\og=4\pi\sqrt{2}\int_{{\mathbb R}_{\ge 0}^2}{\cal J}[\varphi]{\rm d}^2F,\\ \\
&&\int_{{\bRRS}}(\vp+\vp_*-\vp'-\vp'_*)B ff'f_*'
{\rm d}{\bf v}{\rm d}{\bf v}_*
{\rm d}\og
=4\pi\sqrt{2}\int_{{\mathbb R}_{\ge 0}^3}{\cal K}[\varphi]{\rm d}^3F
\eeas
where $B=B({\bf {\bf v-v}_*},\omega)$ is given by (\ref{kernel}) with
(\ref{Phi}), ${\rm d}^2F={\rm d}F(y){\rm d}F(z), {\rm d}^3F={\rm d}F(x){\rm d}F(y){\rm d}F(z)$, and
${\cal J}[\varphi],{\cal K}[\varphi]$ are linear operators
(often used in this paper) defined as follows:
\be{\cal J}[\varphi](y,z)=\frac{1}{2}
\int_{0}^{y+z}{\cal K}[\varphi](x,y,z)
\sqrt{x}{\rm d}x,\quad {\cal K}[\varphi](x, y,z)=W(x,y,z)\Delta\varphi(x,y,z)
,\lb{JKW}\ee
\be \Delta\varphi(x,y,z)=\varphi(x)+\varphi(x_*)-\varphi(y)-\varphi(z)=
(x-y)(x-z)
\int_{0}^1\!\!\!\int_{0}^{1}\vp''(\xi) {\rm d}s {\rm
d}t
\lb{diff}\ee $\xi=y+z-x+t(x-y)+s(x-z)$,
$x,y,z\ge 0,\, x_*=(y+z-x)_{+}$,
\be
W(x,y,z)=\fr{1}{4\pi\sqrt{xyz}}
\int_{|\sqrt{x}-\sqrt{y}|\vee |\sqrt{x_*}-\sqrt{z}|}
^{(\sqrt{x}+\sqrt{y})\wedge(\sqrt{x_*}+\sqrt{z})}{\rm d}s
\int_{0}^{2\pi}\Phi(\sqrt{2}s, \sqrt{2} Y_*){\rm d}\theta
\qquad {\rm if}\quad x_*xyz>0,\lb{W1}\ee
\be
 W(x,y,z)=\left\{\begin{array}
{ll}
 \displaystyle
\fr{1}{\sqrt{yz}}
\Phi(\sqrt{2y}, \sqrt{2z}\,)
\qquad\,\, \qquad\qquad{\rm if}\quad  x=0,\,y>0,\, z>0 \\ \\ \displaystyle
\fr{1}{\sqrt{xz}}
\Phi(\sqrt{2x}, \sqrt{2(z-x)}\,)
\qquad \qquad {\rm if}\quad  y=0,\, z>x>0
\\ \\ \displaystyle
\fr{1}{\sqrt{xy}}
\Phi(\sqrt{2(y-x)}, \sqrt{2x}\,)
\qquad \quad \quad{\rm if}\quad  z=0,\, y>x>0
\\ \\ \displaystyle
0\qquad \quad \qquad {\rm others}\end{array}\right.\lb{W2}\ee

\be Y_*=Y_*(x,y,z,s,\theta)=\left\{\begin{array}
{ll}\displaystyle
\Bigg|\sqrt{\Big(z-\fr{(x-y+s^2)^2}{4s^2}\Big)_{+}
}
+e^{{\rm i}\theta}\sqrt{\Big(x-\fr{(x-y+s^2)^2}{4s^2}
\Big)_{+}}\,\Bigg|\quad {\rm if}\quad s>0\\  \displaystyle
\\
0\qquad\qquad  {\rm if}\quad s=0
\end{array}\right.\lb{Y}\ee
where $\Phi(r,\rho)$ is given in  (\ref{Phi}),$(u)_{+}=\max\{u, 0\}$,
$a\vee b =\max\{a,b\},\, a\wedge b =\min\{a,b\},$
${\rm i}=\sqrt{-1}$.

Based on the existence results,
we introduce directly the concept of measure-valued isotropic solutions of Eq.(\ref{Equation1}) in the weak form:
\vskip2mm

\begin{definition}\label{definition1.1} Let $B({\bf {\bf v-v}_*},\omega)$ be given by (\ref{kernel}), (\ref{Phi}).
Let $F_0\in {\cal B}_{1}^{+}({\mathbb R}_{\ge 0})$. We say that a
family $\{F_t\}_{t\ge 0}\subset {\cal B}_{1}^{+}({\mathbb R}_{\ge 0})$, or simply $F_t$, is a conservative  measure-valued isotropic solution of Eq.(\ref{Equation1}) on the time-interval $[0, \infty)$ with the initial datum $F_t|_{t=0}=F_0$ if

{\rm(i)} $N(F_t)=N(F_0),\,\, E(F_t)=E(F_0)$ for all $t\in [0, \infty)$,

{\rm(ii)} for every $\varphi\in C^2_b({\mathbb R}_{\ge 0})$,  $t\mapsto \int_{{\mathbb R}_{\ge 0}}\varphi(x){\rm d}F_t(x)$ belongs to
$C^1([0, \infty))$,

{\rm(iii)} for every $\varphi\in C^2_b({\mathbb R}_{\ge 0})$
\be\frac{{\rm d}}{{\rm d}t}\int_{{\mathbb R}_{\ge 0}}\varphi{\rm
d}F_t= \int_{{\mathbb R}_{\ge 0}^2}{\cal J}[\varphi]{\rm d}^2F_t+ \int_{{\mathbb R}_{\ge 0}^3}{\cal K}[\varphi]{\rm d}^3F_t\qquad \forall\,t\in[0, \infty). \label{Equation3}\ee
\end{definition}

\begin{remark}\lb{remark}{\rm  (1) The transition from
(\ref{W1}) to (\ref{W2}) in defining $W$ is due to
 the identity
\be (\sqrt{x}+\sqrt{y})\wedge (\sqrt{x_*}+\sqrt{z})-|\sqrt{x}-\sqrt{y}|\vee |\sqrt{x_*}-\sqrt{z}|
=2\min\{\sqrt{x},\sqrt{x_*},\sqrt{y},\sqrt{z}\}\lb{1.difference}\ee
(in case $x_*>0$) from which one sees also that if $\Phi(r,\rho)\equiv 1$,  then $W(x,y,z)$ becomes the function corresponding to the hard sphere model.

(2) As has been proven in \cite{Lu2013} that the test function
space $C_b^2({\mR}_{\ge 0})$ in Definition \ref{definition1.1}  can be weaken to
$$
C^{1,1}_b({\mathbb R}_{\ge 0})=\Big\{ \varphi\in C^1_b({\mathbb R}_{\ge 0})\, \Big|\,
\,
\frac{{\rm d}}{{\rm d}x}\varphi\in {\rm Lip}({\mathbb R}_{\ge 0})\Big\}$$
where ${\rm Lip}({\mathbb R}_{\ge 0})$ is the class of functions
satisfying Lipschitz condition on ${\mathbb R}_{\ge 0}$.  In fact, for the function $\Phi$
satisfying (\ref{Phi}), it is
not difficult to prove that for any $\vp\in C^{1,1}_b({\mathbb R}_{\ge 0})$
the functions $(y,z)\mapsto (1+\sqrt{y}+\sqrt{z})^{-1}{\cal J}[\vp](y,z), (x,y,z)\mapsto {\cal K}[\vp](x,y,z)$ belong to $C_b({\mR}_{\ge 0}^2)$ and $C_b({\mR}_{\ge 0}^3)$
respectively.
Thus there is no problem of integrability in the right hand side of Eq.(\ref{Equation3}).
 A typical example of $\vp\in C^{1,1}_b({\mathbb R}_{\ge 0})$ is $\vp_{\vep}(x)=[(1-x/\vep)_{+}]^2$ (with $\vep>0$) which is often used in Section 5.

(3) It is proved in Appendix that Definition \ref{definition1.1} is
equivalent to Definition \ref{definition5.1}, the latter is established in \cite{Lu2004}.
Thus we conclude from Theorem 1 (Weak Stability), Theorem 2 (Existence) and Theorem 3 in \cite{Lu2004} that for any $F_0\in {\cal B}_{1}^{+}({\mathbb R}_{\ge 0})$,
the Eq.(\ref{Equation1}) has always a conservative measure-valued isotropic solution
 $F_t$ on the time-interval $[0, \infty)$ with the initial datum $F_t|_{t=0}=F_0$.
The reason that we use Definition \ref{definition1.1}
is just because the collision integrals in Definition \ref{definition1.1}
are simpler in structure than those in Definition \ref{definition5.1}.
Note that if we write
$B({\bf {\bf v-v}_*},\omega)$ as
$$B({\bf {\bf v-v}_*},\omega)=B(V,\cos(\theta))= \fr{1}{(4\pi)^2}V\cos(\theta)\Phi(V\cos(\theta),V\sin(\theta)) $$
where $V=|{\bf {\bf v-v}_*}|,  \theta=\arccos(|({\bf v-v}_*)\cdot\omega|/|{\bf {\bf v-v}_*}|)\in [0,\pi/2]$, then it is also required in \cite{Lu2004} that
$\int_{0}^{\pi/2}\sin(\theta)B(V,\cos(\theta)){\rm d}\theta>0$ for all $V>0$. But
in \cite{Lu2004} this strict positivity is just used to guarantee the uniqueness of the equilibrium and
to prove (with a further assumption on $B$) the moment production; it is apparently not used in the proofs of the three theorems in \cite{Lu2004} mentioned above.}
\end{remark}
\vskip2mm

{\bf Kinetic Temperature.}   Let $F\in {\mathcal B}_{1}^{+}({\mathbb R}_{\ge 0})$, $N=N(F), E=E(F)$ and suppose $N>0$. If $m$ is the mass of one particle,
then  $m4\pi\sqrt{2}N$, $m 4\pi\sqrt{2}E$ are total
mass and kinetic energy of the particle system per unite
space volume.
The kinetic temperature $\overline{T}$ and the kinetic
critical temperature $\overline{T}_c$ are defined by (see e.g.\cite{Lu2004} and references therein)
$$\overline{T}=\frac{2m}{3k_{\rm B}}
\frac{E}{N}
,\qquad\overline {T}_c=\frac{\zeta(5/2)}{(2\pi)^{1/3}[\zeta(3/2)]^{5/3}}
\frac{2m}{k_{\rm B}}N^{2/3} $$
where $k_{\rm B}$ is the Boltzmann constant, $\zeta(s)=\sum_{n=1}^{\infty} n^{-s}, s>1$. Some properties involving temperature effect, for instance the Bose-Einstein
condensation at low temperature, are often expressed
in terms of the ratio
$$ \frac{\overline{T}}{\overline {T}_c}= \frac{(2\pi)^{1/3}[\zeta(3/2)]^{5/3}}{3\zeta(5/2)}
\frac{E}{N^{5/3}}= 2.2720\frac{E}{N^{5/3}}. $$
Keeping in mind the constant $m4\pi\sqrt{2}$,
there will be no confusion if we also call $N$ and $E$ the
mass and energy of a particle system.
\vskip2mm

{\bf Regular-Singular Decomposition.}  According to measure theory (see e.g.\cite{Rudin}), every
finite positive Borel measure can be uniquely decomposed into
regular part and singular part with respect to the Lebesgue measure. For instance
if $F\in {\cal B}_1^{+}({\mathbb R}_{\ge 0})$, then
there exist unique $0\le f\in L^1({\mathbb R}_{\ge 0},(1+x)\sqrt{x}{\rm d}x)$, $\nu\in{\cal B}_1^{+}({\mathbb R}_{\ge 0})$ and a Borel set $Z\subset {\mathbb R}_{\ge 0}$
such that
$${\rm d}F(x)=f(x)\sqrt{x}{\rm d}x+{\rm d}\nu(x),\quad mes(Z)=0,\quad \nu({\mR}_{\ge 0}\setminus Z)=0.$$
We call $f$ and $\nu$ the regular part and the singular part of $F$
respectively\footnote{Strictly speaking the product $f(x)\sqrt{x}$
is the regular part of $F$. The reason that we only mention $f$ is because $f(x)\sqrt{x}$ comes from the 3D-isotropic function $f=f(|{\bf v}|^2/2)$. }.  If the regular part $f$ is non-zero, i.e., if  $\int_{{\mathbb R}_{\ge 0}}f(x)\sqrt{x}\,{\rm d}x>0$, then we say that $F$ is non-singular.
\vskip2mm

{\bf Bose-Einstein Distribution.}  According to Theorem 5 of \cite{Lu2004} and
its equivalent version proved in the Appendix of \cite{Lu2013} we know that
for any  $N>0$, $E>0$ the Bose-Einstein distribution $F_{{\rm be}}\in {\mathcal B}_1^{+}({\mathbb R}_{\ge 0})$ given by
\beas {\rm d}F_{\rm be}(x)=f_{\rm be}(x)\sqrt{x}{\rm d}x+\big(1-(\overline{T}/\overline{T}_c)^{3/5}\big)_{+}
N\dt(x){\rm d}x\eeas
is the unique equilibrium solution of Eq.(\ref{Equation3})
satisfying $N(F_{\rm be})=N, E(F_{\rm be})=E$,  where
\be f_{\rm be}(x)=
\left\{\begin{array}{ll}
\displaystyle \frac{1}{Ae^{x/\kappa}-1},\quad A>1,\,\,\qquad {\rm if}
\quad \overline{T}/\overline{T}_c>1,\\ \\
\displaystyle
\frac{1}{e^{x/\kappa}-1},\qquad A=1\,\,\,\qquad  {\rm if}\quad
\overline{T}/ \overline{T}_c\le 1
\end{array}\right.\label{1-E}\ee
$\dt(x)$ is the Dirac
delta function concentrated at $x=0$, i.e.
$\dt(x){\rm d}x={\rm d}\nu_0(x)$  where  $\nu_0$ is the Dirac measure
concentrated at $x=0$, and
functional relations of the coefficients $A=A(N,E)\ge 1, \kappa=\kappa(N,E)>0$
can be found in for instance Proposition 1 in \cite{Lu2005}. From (\ref{1-E}) one sees that
$\overline{T}/\overline{T}_c\ge 1 \, \Longrightarrow\,  {\rm d}F_{\rm be}(x)=f_{\rm be}(x)\sqrt{x}{\rm d}x$, and
$$\overline{T}/\overline{T}_c<1\, \Longleftrightarrow\, F_{{\rm be}}(\{0\})=\big(1-(\overline{T}/\overline{T}_c)^{3/5}\big)N>0.$$
The positive number $(1-(\overline{T}/\overline{T}_c)^{3/5} )N$
is called the Bose-Einstein Condensation (BEC) of the equilibrium state of Bose-Einstein particles at low temperature $\overline{T}<\overline{T}_c$.
\vskip2mm

{\bf Entropy.}  The entropy functional for Eq.(\ref{Equation1}) is
\be S(f)=\int_{{\bR}}\big((1+f({\bf v}))\log(1+f({\bf v}))
-f({\bf v})\log f({\bf v})\big){\rm d}{\bf v},\quad 0\le f\in L^1_2({\bR}).\lb{ent}\ee
$S(f)$ is always finite since for any $f=f({\bf v})\ge 0$ we have (\cite{Lu2005})
\beas 2\fr{f}{1+f}\le  (1+f)\log(1+f)
-f\log f
\le 2\min\{\sqrt{f},\, (1+|{\bf v}|)f+e^{-|{\bf v}|/2}\}.\eeas
Moreover since the function
\be y\mapsto s(y)=(1+y)\log(1+y)
-y\log y, \quad y\in[0,\infty)\lb{entropyfunction}\ee is concave and  non-decreasing  with $s(0)=0$,  it follows that
for all $0\le f,g \in L^1_2({\bR})$
\bes &&
0=S(0)\le S(f)\le S(g)<\infty\quad {\rm if}\quad f\le g,\lb{S1}\\
&&
S((1-\alpha)f+\alpha g)\ge (1-\alpha)S(f)+\alpha S(g)\quad
{\rm if}\quad 0<\alpha<1,\dnumber\lb{S2}\\
&&S(f+g)\le S(f)+ S(g).\dnumber\lb{S3}\ees
To study measure-valued solutions we define the entropy $S(F)$ of a measure
$F\in {\cal B}_2^{+}({\bR})$ by
\be S(F):=\sup_{\{f_n\}_{n=1}^{\infty}}\limsup_{n\to\infty}S(f_n)\lb{ent1}\ee
where $\{f_n\}_{n=1}^{\infty}$ under the sup is taken all sequences in $L^1_2({\bR})$ satisfying
\bes&& f_n\ge 0,\quad \sup_{n\ge 1}\|f_n\|_{L^1_2}<\infty;  \lb{1.20}\\
&&\lim_{n\to\infty}\int_{{\bR}}\psi({\bf v})f_n({\bf v}){\rm d}{\bf v}
=\int_{{\bR}}\psi({\bf v}){\rm d}F({\bf v})\qquad \forall\,\psi\in C_b({\bR}).\dnumber\lb{1.21}\ees
Let $0\le f\in L^1_2({\bR})$ be the regular part of $F$, i.e.
${\rm d}F({\bf v})=f({\bf v}){\rm d}{\bf v}+{\rm d}\nu({\bf v})$ with $\nu\ge 0$
 the singular part of $F$. By Lemma \ref{lemma.entropy} (see Section 3)  we have
\be S(F)=S(f)\lb{ent2}\ee
which shows that the singular part of
$F$ has no contribution
to the entropy $S(F)$ and that $F$ is non-singular if and only if $S(F) > 0.$
A referee of the paper conveyed us
that the convex functions of measures had been defined and studied
in \cite{DT}, and the equality (\ref{ent2}) coincides with the definition in \cite{DT}:
$S(F)=\int_{{\bR}}{\rm d}s(F)({\bf v})$ where
$y\mapsto s(y)$ is the function
 (\ref{entropyfunction}) and, according to \cite{DT}, the transformed measure $s(F)$ is
 defined by ${\rm d}s(F)({\bf v})=s(f({\bf v})){\rm d}{\bf v}+s_{\infty}{\rm d}\nu({\bf v})$ with
 $s_{\infty}:=\lim\limits_{t\to+\infty}\fr{s(t)}{t}=0$, here the zero limit is obvious by definition of $s(y)$. Thus we obtain (\ref{ent2}) again.

For any
$0\le f\in L^1({\mathbb R}_{\ge 0},(1+x)\sqrt{x}\,{\rm d}x)$, the entropy
$S(f)$ is defined by $S(f)=S(\bar{f})$
with $\bar{f}({\bf v}):=f(|{\bf v}|^2/2)$, so that  (using (\ref{ent}) and change of variable)
\be S(f)=S(\bar{f})=4\pi\sqrt{2}\int_{{\mathbb R}_{\ge 0}}\big((1+f(x))\log(1+f(x))-f(x)\log f(x)\big)\sqrt{x}\,{\rm d}x.\lb{ent3}\ee
In general, the entropy $S(F)$  for a measure
$F\in {\cal B}_1^{+}({\mathbb R}_{\ge 0})$ is defined by
$S(F)=S(\bar{F})$ where
$\bar{F}\in {\cal B}_2^{+}({\bR})$ is defined by $F$ through (\ref{m2}) or (\ref{m4}) and
$S(\bar{F})$ is defined by (\ref{ent1}). Accordingly for the regular-singular decomposition ${\rm d}F(x)=f(x)\sqrt{x}\,{\rm d}x+{\rm d}\nu(x)$  with the singular part $\nu$, we have
the regular-singular decomposition ${\rm d}\bar{F}({\bf v})=\bar{f}({\bf v})
{\rm d}{\bf v}+{\rm d}\bar{\nu}({\bf v})$
with $\bar{f}({\bf v})=f(|{\bf v}|^2/2)$ and the singular part $\bar{\nu}$
is expressed by $\nu$ through (\ref{m2}). Thus from
(\ref{ent2}),(\ref{ent3}) we have
$S(F)=S(\bar{F})=S(\bar{f})=S(f)$.

Although the entropy $S(F_t)$ defined above does not provide any information about the singular part of $F_t$
so that one has to consider other methods for proving the convergence of $F_t(\{0\})$ to the condensation $F_{\rm be}(\{0\})$,  the entropy difference $S(F_{\rm be})-S(F_t)$ can still describe and control the convergence to equilibrium in a semi-strong norm
(see (\ref{3.10})),  and thus in an indirect way it gives a rate of strong convergence to equilibrium (see proofs of Theorem \ref{theorem4.8} and Theorem \ref{theorem1.2}).
 \vskip2mm

{\bf Main Result.} The main result of the paper is as follows:
\begin{theorem}\lb{theorem1.2}
Suppose $B({\bf {\bf v-v}_*},\omega)$ satisfy Assumption \ref{assp} with $0\le\eta<1/4$.
Let $F_0\in {\mathcal B}_1^{+}({\mathbb R}_{\ge 0})$ satisfy
$N(F_0)>0, E(F_0)>0$, let $F_{\rm be}$ be
the unique Bose-Einstein distribution with the same mass $N=N(F_0)$ and energy $E=E(F_0)$. Let 
 $\fr{1}{20}<\ld<\fr{1}{19}$. Then
there exists a conservative  measure-valued isotropic solution $F_t$ of
Eq.(\ref{Equation1}) on $[0,\infty)$ with the initial datum $F_0$ such that
$S(F_t)\ge S(F_0)$ and $S(F_t)>0 $ for all $t>0$, and it holds
\beas 0\le S(F_{\rm be})-S(F_t)\le C(1+t)^{-\ld},\quad\|F_t-F_{{\rm be}}\|_1\le C(1+t)^{-\fr{(1-\eta)\ld}{2(4-\eta)}}\qquad \forall\, t\ge 0.\eeas
In particular if\, $\overline{T}/\overline{T}_c<1$  then
\be\big|F_t(\{0\})-(1-(\overline{T}/\overline{T}_c)^{3/5})N\big|\le  C(1+t)^{-\fr{(1-\eta)\ld}{2(4-\eta)}}\qquad \forall\, t\ge 0.\label{rate2}\ee
Here the constant $C>0$ depends only on $N,E, b_0, \eta$ and $\ld.$
\end{theorem}
\vskip2mm

\begin{remark}\lb{remark1.5}{\rm (1)
Let
${\rm d}F(x)=f(x){\rm d}x+{\rm d}\mu(x), {\rm d}G(x)=g(x){\rm d}x+{\rm d}\nu(x)$ be the
regular-singular decompositions of $F, G\in {\cal B}^{+}(X)$  respectively.
Recall that this means that
$0\le f,g \in L^1(X), \mu, \nu\in {\cal B}^{+}(X)$ and there are Borel sets $Z_F,Z_G\subset X$  such that
$mes(Z_F)=mes(Z_G)=0,
\mu(X\setminus Z_F)=0,\nu(X\setminus Z_G)=0.$
According to measure theory,
the regular-singular decomposition of $|F-G|$ is
\be{\rm d}|F-G|(x)=|f(x)-g(x)|{\rm d}x+{\rm d}|\mu-\nu|(x)\lb{decomp2}\ee
i.e.
\beas |F-G|(A)=\int_{A}|f(x)-g(x)|{\rm d}x+|\mu-\nu|(A)\qquad \forall\,{\rm Borel\,\,set}\,\,
A\subset X\lb{decomp3}\eeas
which is easily proved by using
 the fact that (see e.g.\cite{Rudin}) there is a real Borel function $h$
on $X$ satisfying $h(x)^2\equiv 1$ on $X$ such that
\be {\rm d}(F-G)(x)=h(x){\rm d}|F-G|(x),\quad {\rm i.e.}\quad
{\rm d}|F-G|(x)=h(x){\rm d}(F-G)(x).\lb{1.29}\ee
The same holds also for $\mu-\nu$.
Using (\ref{decomp2}) to the regular-singular decompositions ${\rm d}F_t(x)
=f(x,t)\sqrt{x}{\rm d}x+{\rm d}\nu_t(x)$ and ${\rm d}F_{\rm be}(x)
=f_{\rm be}(x)\sqrt{x}{\rm d}x+{\rm d}\nu_{\rm be}(x)$ we have
\beas&& \|F_t-F_{{\rm be}}\|_1=\int_{{\mR}_{\ge 0}}(1+x)|f(x,t)-f_{\rm be}(x)|\sqrt{x}{\rm d}x+\|\nu_t-\nu_{\rm be}\|_1,\\
&&\|\nu_t-\nu_{\rm be}\|_1=
\int_{{\mR}_{>0}}(1+x){\rm d}\nu_t(x)+|F_t(\{0\})-F_{\rm be}(\{0\})|\eeas
for all $t\ge 0.$
From these and Theorem \ref{theorem1.2} one sees that, as $t\to\infty$, the regular part and the singular part of $F_t$ converge strongly to the regular part and the singular part of $F_{\rm be}$
 respectively.

(2) In comparison with the exponential convergence to equilibrium
for the spatially homogeneous classical Boltzmann equation for hard potentials,
the rate of convergence to equilibrium obtained in Theorem \ref{theorem1.2} is very slow even for the
hard sphere model (i.e. the case $\eta=0$).
This is perhaps not only because of the condensation effect (it is hard to obtain a fast decay
rate for $|F_t(\{0\})-F_{\rm be}(\{0\})|$ for low temperature), but also because of the complicated structure of Eq.(\ref{Equation1}) which leads to complicated structures of entropy and entropy dissipation. In fact from \cite{ET} one sees that, even for a linearized model of Eq.(\ref{Equation1}) for low temperature, it is difficult to obtain a high order algebraic rate of
convergence to equilibrium.

(3) An easy case which is not mentioned in Theorem \ref{theorem1.2}  is that $N>0, E=0$.
In this case, any conservative  measure-valued isotropic solution $F_t$ of
Eq.(\ref{Equation1}) on $[0,\infty)$ with $N(F_t)=N, E(F_t)=0$ is equal to the same
Dirac measure: $F_t=F_{\rm be}$ for all $t\ge 0$, where ${\rm d}F_{\rm be}(x)=N\dt(x){\rm d}x$, and
the Dirac measure $F_{\rm be}$ is also the unique equilibrium solution of Eq.(\ref{Equation3})
satisfying $N(F)=N, E(F)=0$.

(4) Theorem \ref{theorem1.2} shows that the strong convergence to equilibrium
is grossly determined,
which means it depends only on
the mass, energy, and some constants coming from the collision kernel.
 In particular, for the case of low temperature
$\overline{T}/\overline{T}_c<1$, the convergence to BEC does not depend on any local information of initial data. Thus Theorem \ref{theorem1.2} also gives an essential improvement to our previous work \cite{Lu2016},\cite{Lu2018}.

}

\end{remark}
\vskip2mm

 The rest of the paper is organized as follows:
In Section 2 we prove an inequality of entropy and entropy dissipation  for general functions.
In Section 3 we introduce approximate solutions of Eq.(\ref{Equation1}) and prove positive lower bounds of entropy for isotropic approximate solutions and hence for isotropic measure-valued solutions.
In Section 4, using the results of Sections 2 and 3 we obtain an algebraic decay rate of
the entropy difference $S(F_{\rm be})-S(F_t)$.  In Section 5 we
prove the long time convergence of the condensation $F_t(\{0\})$ to $F_{\rm be}(\{0\})$, and at the end of that section we finish the
proof of our main result Theorem \ref{theorem1.2}. Section 6 is an appendix where we
prove some general properties of collision integrals and the equivalence of two definitions of
measure-valued isotropic solutions of Eq.(\ref{Equation1}), we also prove existence and positivity of some interaction potentials mentioned just above
Assumption \ref{assp}.

\begin{center}\section{Inequality of Entropy and Entropy Dissipation}\end{center}

Entropy and entropy dissipation are powerful tools
for investigating long time behavior of solutions of classical and quantum Boltzmann equations.
For the latter case, see for instance \cite{EM}, \cite{Lu2005}, and the derivation below.

 We begin with the following lemma which provides some connections between strong convergence to equilibrium, entropy convergence, and convergence of condensation.

\begin{lemma}\lb{lemma3.1} Given $N>0, E>0$. Let $
 F \in{\cal B}_1^{+}({\mR}_{\ge 0})$ satisfy
 $N(F)=N, E(F)=E$. Then there are finite constants $C>0$ (which may be different in different lines)
depending only on $N, E$  such that

\noindent {\rm(I)}
\be \fr{1}{C}\big(\|F-F_{\rm be}\|_1^{\circ}\big)^2
\le S(F_{\rm be})-S(F)\le C\big(\|F-F_{\rm be}\|_1^{\circ}\big)^{1/2}.\lb{3.10}\ee
(II)
If\, $\overline{T}/\overline{T}_c<1$, then
\be |F(\{0\})-F_{\rm be}(\{0\})|\le \|F-F_{\rm be}\|_{1}\le 2|F(\{0\})-F_{\rm be}(\{0\})|
+C(\|F-F_{\rm be}\|_{1}^{\circ})^{1/3}.\lb{3.11}\ee
If\, $\overline{T}/\overline{T}_c\ge 1$, then
\be \|F-F_{\rm be}\|_{1}\le
C\big(\|F-F_{\rm be}\|_{1}^{\circ}\big)^{1/3}.\lb{3.12}\ee
Here
$$\|\mu\|_{1}^{\circ}=\int_{{\mR}_{\ge 0}}x{\rm d}|\mu|(x),\quad\|\mu\|_{1}=\int_{{\mR}_{\ge 0}}(1+x){\rm d}|\mu|(x),\quad   \mu\in {\cal B}_1({\mR}_{\ge 0}).$$
\end{lemma}

\begin{proof} {\rm(I)}: Let $0\le f\in L^1({\mathbb R}_{+},(1+x)\sqrt{x}{\rm d}x)$, $\nu\in{\cal B}_1^{+}({\mathbb R}_{\ge 0})$
be the regular part and the singular part of $F$.
Let $\bar{f}({\bf v})=f(|{\bf v}|^2/2)$ and recall that $\Og({\bf v})=f_{\rm be}(|{\bf v}|^2/2)$
is given in
(\ref{REE}). We have proved in Lemma 4 of \cite{Lu2005} that
\be \fr{1}{C}\big(\|\bar{f}-\Og\|^{\circ}_{L^1_2}\big)^2\le
S(\Og)-S(\bar{f})\le {C}\big(\|\bar{f}-\Og\|^{\circ}_{L^1_2}\big)^{1/2}\lb{3.13}\ee
where
$$\|\psi\|^{\circ}_{L^1_2}=\int_{{\bR}}|{\bf v}|^2 |\psi({\bf v})|{\rm d}{\bf v},\quad \psi\in L^1_2({\bR}).$$
Let $(F_{\rm be}-F)^{+}=\fr{1}{2}(|F_{\rm be}-F|+ F_{\rm be}-F)$ be the positive part of $ F_{\rm be}-F$.
Then
\bes&& {\rm d}|F_{\rm be}-F|(x)={\rm d}F(x)-{\rm d}F_{\rm be}(x)+2{\rm d}(F_{\rm be}-F)^{+}(x)\lb{3.14}\\
&&
{\rm d}(F_{\rm be}-F)^{+}(x)
=h_{+}(x){\rm d}(F_{\rm be}-F)(x),\quad h_{+}(x)\in[0,1]\dnumber \lb{3.15}\ees
where $h_{+}(x)=\fr{1}{2}(1+h(x))$ (see (\ref{1.29})).
Using (\ref{3.14}) and $E(F)=E(F_{\rm be})$ we have
$$\|F-F_{\rm be}\|_1^{\circ}=2\int_{{\mR}_{\ge 0}}
x{\rm d}(F_{\rm be}-F)^{+}(x)\le 2\int_{{\mR}_{\ge 0}}
x|f_{\rm be}(x)-f(x)|\sqrt{x}{\rm d}x.$$
On the other hand from (\ref{decomp2})we have
$${\rm d}|F-F_{\rm be}|(x)=|f(x)-f_{\rm be}(x)|\sqrt{x}{\rm d}x+
{\rm d}|\nu-\nu_{\rm be}|(x)\ge |f(x)-f_{\rm be}(x)|\sqrt{x}{\rm d}x.$$
Thus
\be \int_{{\mR}_{\ge 0}}
x|f(x)-f_{\rm be}(x)|\sqrt{x}{\rm d}x\le \|F-F_{\rm be}\|_1^{\circ}\le 2\int_{{\mR}_{\ge 0}}
x|f(x)-f_{\rm be}(x)|\sqrt{x}{\rm d}x.\lb{3.16}\ee
Since
$$\int_{{\mR}_{\ge 0}}
x|f(x)-f_{\rm be}(x)|\sqrt{x}{\rm d}x=\fr{1}{8\pi\sqrt{2}}\|\bar{f}-\Og\|^{\circ}_{L^1_2}
,\quad S(F_{\rm be})-S(F)=S(\Og)-S(\bar{f})$$
(\ref{3.10}) follows from (\ref{3.13}) and (\ref{3.16}).

(II): For the case $\overline{T}/\overline{T}_c<1$, the inequality (\ref{3.11})
is a result of Lemma 2.8 in \cite{Lu2016}.
Suppose $\overline{T}/\overline{T}_c\ge 1$. Then
${\rm d}F_{\rm be}(x)=f_{\rm be}(x)\sqrt{x}{\rm d}x=\fr{1}{Ae^{x/\kappa}-1}\sqrt{x}{\rm d}x$ and so
using $N(F)=N(F_{\rm be})$ and (\ref{3.14}),(\ref{3.15}) we have for any $\vep>0$
\beas&&
\|F-F_{\rm be}\|_0=
2\int_{[0,\vep]}{\rm d}(F_{\rm be}-F)^{+}(x)
+2\int_{(\vep,\infty)}{\rm d}(F_{\rm be}-F)^{+}(x),\\
&&\le 2\int_{[0,\vep]}
\fr{1}{e^{x/\kappa}-1}\sqrt{x}{\rm d}x+\fr{1}{\vep}\|F-F_{\rm be}\|_{1}^{\circ}
\le 4\kappa\sqrt{\vep}+\fr{1}{\vep}\|F-F_{\rm be}\|_{1}^{\circ}.
\eeas
Minimizing w.r.t $\vep>0$ gives
$\|F-F_{\rm be}\|_0\le C(\|F-F_{\rm be}\|_{1}^{\circ})^{1/3}$
and thus
$\|F-F_{\rm be}\|_{1}= \|F-F_{\rm be}\|_0
+\|F-F_{\rm be}\|_{1}^{\circ}\le
 C(\|F-F_{\rm be}\|_{1}^{\circ})^{1/3}$.
\end{proof}
\vskip3mm

Since the entropy dissipation
$D(f)$ is monotone non-decreasing with respect to the collision kernel $B$, it suffices to
establish the relevant estimates for a ``minimal" kernel $B_{\rm min}$:
\be B_{\rm min}({\bf v}-{\bf v}_*,\omega)=\fr{b_0}{(4\pi)^2}\Psi_0(|{\bf v-v}_*|)\cos^3(\theta)\sin^3(\theta)
\lb{cutoff}\ee
where $0<b_0<1$ is a constant, $\Psi_0: {\mR}_{\ge 0}\to [0,1]$ is a non-decreasing function
satisfying $0<\Psi_0(r)\le 1$ for all $r>0$, and $\theta=\arccos(|({\bf v-v}_*)\cdot\omega|/|{\bf {\bf v-v}_*}|)\in [0,\pi/2]$
and we define $\theta=0$ for the case ${\bf v}={\bf v}_*$.
The entropy dissipation
corresponding to $B_{\rm min}$ is
\beas&& D_{\rm min}(f)=\fr{1}{4}
\int_{{\bRRS}}B_{\rm min}({\bf
{\bf v-v}_*},\og)\\
&&\times \big(f'f_*'(1+f)(1+f_*)-ff_*(1+f')(1+f_*')\big)
\log\Big(\fr{f'f_*'(1+f)(1+f_*)}{ff_*(1+f')(1+f_*')}\Big)
{\rm d}\omega{\rm
d}{\bf v_*}{\rm
d}{\bf v}. \eeas
If we define
\be \Gm(a,b)=
\left\{\begin{array}{ll}
\displaystyle (a-b)\log\big(\fr{a}{b}\big)\,\qquad {\rm if}
\quad a>0, b>0\\
\displaystyle
\,\infty\qquad \quad\qquad \qquad \,\,  {\rm if}\quad
a>0=b\,\,\,{\rm or}\,\,\, a=0<b\\
\displaystyle
\,0\qquad \qquad \qquad \quad\,\,\,\,\,  {\rm if}\quad a=b=0
\end{array}\right.\lb{2.Gamma}\ee
\be\Pi(f)=(1+f)(1+f_*)(1+f')(1+f_*'),\quad g=\fr{f}{1+f},\quad 0\le f\in L^1_2({\bR})\lb{2.Dmin}\ee
then
$D_{\rm min}(f)$ can be written as a shorter and clear version:
\be D_{\rm min}(f)=\fr{1}{4}
\int_{{\bRRS}}B_{\rm min}({\bf
{\bf v-v}_*},\og)\Pi(f)\Gm(g'g_*',gg_*){\rm d}\omega{\rm
d}{\bf v_*}{\rm
d}{\bf v}. \label{diss}\ee
We will also use
Villani's  inequality (\cite{Villani})
\be H(f|M)\le \fr{1}{c_{H_0}}{\cal D}_2(f)\qquad \forall\,f\in {\cal H}_0 \lb{Villani}\ee
where  ${\cal D}_2(f)$ is
the entropy dissipation  for a ``super hard potential" model
\be {\cal D}_2(f)=\fr{1}{4}\int_{{\bRRS}}\fr{2|({\bf v-v}_*)\cdot\omega|}{|{\bf v-v}_*|}
(1+|{\bf v-v}_*|^2)\Gm(f'f_*',ff_*)
{\rm d}\omega{\rm
d}{\bf v_*}{\rm
d}{\bf v},\lb{vdiss}\ee
$$H(f|M)=H(f)-H(M),\quad H(f)=\int_{{\bR}}f({\bf v})\log f({\bf v}){\rm d}{\bf v},\quad
M({\bf v})=\fr{e^{-|{\bf v}|^2/2}}{(2\pi)^{3/2}},$$
$${\cal H}_0=\Big\{ 0\le f\in L^1_2({\bR})\,\Big|\,
\int_{{\bR}}(1,{\bf v}, |{\bf v}|^2/3)f({\bf v}){\rm d}{\bf v}=(1,0,1),\, H(f|M)\le H_0\Big\},$$
$0\le H_0<\infty$ is a constant, and the constant $c_{H_0}>0$ depends only $H_0$.
See \cite{Villani}, see also pp.724-725 of \cite{CCL} for some estimate of $c_{H_0}$.
In particular for isotropic functions $f\in {\cal H}_0$,
the inequality (\ref{Villani}) holds with $c_{H_0}=\fr{2\pi}{7}$  (\cite{Villani}).
Note that for all $0\le f\in
L^1_2({\bR})$ it holds $\int_{f({\bf v})\le 1}f({\bf v})|\log f({\bf v})|{\rm d}{\bf v}
 <\infty$ and thus the $H$-function $H(f)$ makes sense and $H(f)<\infty$ if and only if
$\int_{{\bR}}f({\bf v})|\log f({\bf v})|{\rm d}{\bf v}<\infty$.
Note also that in the original Villani's inequality (\ref{Villani}),
the entropy dissipation ${\cal D}_2(f)$ is equivalently expressed with the $\sg$-representation
(\ref{colli2}) (see the formula (\ref{3.9})).
\vskip2mm

The aim of this section is to prove the following
inequality between entropy and entropy dissipation:

\begin{proposition}\lb{prop2.1} Given $N>0, E>0$. Let $f_{\rm be}(x)=\fr{1}{Ae^{x/\kappa}-1}$ be the regular part of the unique equilibrium $F_{\rm be}\in {\cal B}_{1}^{+}({\mathbb R}_{\ge 0})$ which has the mass and energy $N, E$. Let
\be\Og({\bf v})=f_{\rm be}(|{\bf v}|^2/2)=\fr{a e^{-b|{\bf v}|^2}}{1-ae^{-b|{\bf v}|^2}},
\quad 0<a=\fr{1}{A}\le 1,\,\, b=\fr{1}{2\kappa}>0\lb{REE}\ee
and let $0\le f\in L^1_2({\bR})$ satisfy
\beas \int_{{\bR}}(1, {\bf v}, |{\bf v}|^2/2)f({\bf v}){\rm d}{\bf v}= 4\pi\sqrt{2}(N,0, E).\eeas
Then $S(f)\le S(\Og)$.  If in addition $f\in L^1_4({\bR})$ and for some $0<S_0, C_0<\infty$,
\beas
S(f)\ge S_0, \quad \|f\|_{L^1_4}\le C_0\eeas
then
\be S(\Og)-S(f)\le C\big( ({\cal D}_2(g))^{1/10}+(D_{\rm min}(f))^{1/4}\,\big)\lb{e-e}\ee
where
$g=\fr{f}{1+f},$  ${\cal D}_2(\cdot), D_{\rm min}(\cdot)$ are entropy dissipations given in (\ref{vdiss}), (\ref{cutoff})-(\ref{diss}), and
the constant $C>0$ depends only on $N, E, S_0, C_0, b_0 $ and $\Psi_0$.
 \end{proposition}

\begin{proof} From the assumption of the proposition we have
$\int_{{\bR}}|{\bf v}|^2f({\bf v}){\rm d}{\bf v}
=\int_{{\bR}}|{\bf v}|^2\Og({\bf v}){\rm d}{\bf v}$ and for the case
$a=1/A<1$ we also have $\int_{{\bR}}f({\bf v}){\rm d}{\bf v}
=\int_{{\bR}}\Og({\bf v}){\rm d}{\bf v}$. Thus it follows from Lemma 4 in \cite{Lu2005}
that $S(f)\le S(\Og)$. Now suppose $f\in L^1_4({\bR})$.
The proof of (\ref{e-e}) is divided into six steps.

{\bf Step1.}
Let
$$M_g({\bf v})=\alpha e^{-\beta|{\bf v-v}_0|^2},\quad M_2(g)=\int_{{\bR}}|{\bf v-v}_0|^2g({\bf v}){\rm d}{\bf v}$$
with $\alpha>0, \beta>0, {\bf v}_0\in {\bR}$ given by $g$ through the moment equations
\be\int_{{\bR}}M_g({\bf v})(1, {\bf v}, |{\bf v}|^2){\rm d}{\bf v}=
\int_{{\bR}}g({\bf v})(1, {\bf v}, |{\bf v}|^2){\rm d}{\bf v}.\lb{moment2}\ee
In this step we prove the lower and upper bounds:
\be C_1\le \|g\|_{L^1},  M_2(g)\le C_2,\quad
C_3\le\alpha, \beta\le C_4,\quad  |{\bf v}_0|\le C_5\lb{2.12}\ee
where here and below the constants $0<C_i<\infty\,(i=1,2,...,16)$ depend only
on $N, E, S_0, C_0, b_0 $ and $\Psi_0$.

By definition of $S(f)$ and using Cauchy-Schwarz inequality we have
\be
S(f)^2
\le \|g\|_{L^1}\int_{f({\bf v})>0}\big(1+\fr{1}{f({\bf v})}\big)
s(f({\bf v}))^2{\rm d}{\bf v}.\lb{2.13}\ee
Then we use the inequality $\log(1+y)\le \sqrt{y}\wedge y$ to get for all
$y>0$
$$
s(y)^2=\big(\log(1+y)+y\log(1+ \fr{1}{y})\big)^2
\le \min\big\{4y,\,2y^2+2y^2\big(\log(1+ \fr{1}{y})\big)^2\big\}$$
from which we deduce
\be
\int_{f({\bf v})>0}\big(1+\fr{1}{f({\bf v})}\big)
s(f({\bf v}))^2{\rm d}{\bf v}
\le 6\|f\|_{L^1}+2
\int_{f({\bf v})>0}f({\bf v})\big(\log\big(1+ \fr{1}{f({\bf v})}\big)\big)^2
{\rm d}{\bf v}.\lb{2.14}\ee
Furthermore using the inequality
$y(\log(1+\fr{1}{y}))^2\le 4 \sqrt{y}\,\,(y>0)$ we have
\beas
f({\bf v})\big(\log\big(1+\fr{1}{f({\bf v})}\big)\big)^2
&\le&4\sqrt{f({\bf v})}1_{\{f({\bf v})\le e^{-|{\bf v}|}\}}+
f({\bf v})\big(\log(1+e^{|{\bf v}|})\big)^2 1_{\{f({\bf v})> e^{-|{\bf v}|}\}}\\
&\le& 4e^{-|{\bf v}|/2}+2
f({\bf v})\la{\bf v}\ra^2.\eeas
Thus
$$
\int_{f({\bf v})>0}
f({\bf v})\big(\log\big(1+\fr{1}{f({\bf v})}\big)\big)^2
{\rm d}{\bf v}\le 2^8\pi+2\|f\|_{L^1_2} $$
which together with (\ref{2.13}), (\ref{2.14}) gives
\be\|f\|_{L^1}\ge \|g\|_{L^1}\ge \fr{S(f)^2}{2^{9}\pi+10\|f\|_{L^1_2}}.\lb{2.15}\ee
Next using Cauchy-Schwarz inequality and $0\le g\le 1$ we have
$$(\|g\|_{L^1})^2\le \int_{{\bR}}|{\bf v-v}_0|^2g({\bf v}) {\rm d}{\bf v}
\int_{{\bR}}\fr{1}{|{\bf v-v}_0|^2}g({\bf v}) {\rm d}{\bf v}\\
\le M_2(g)2(4\pi)^{2/3}(\|g\|_{L^1})^{1/3}$$
$\Longrightarrow$
\be M_2(g)\ge \fr{1}{2(4\pi)^{2/3}} (\|g\|_{L^1})^{5/3}
\ge
\fr{1}{2(4\pi)^{2/3}}\Big(\fr{S(f)^2}{2^{9}\pi+10\|f\|_{L^1_2}}\Big)^{5/3}.\lb{2.16}\ee
Finally from ${\bf v}_0=\fr{1}{\|g\|_{L^1}}\int_{{\bR}}{\bf v} g({\bf v}){\rm d}{\bf v}$ and
Cauchy-Schwarz inequality we have
\be |{\bf v}_0|^2\le \fr{1}{\|g\|_{L^1}}\int_{{\bR}}|{\bf v}|^2 g({\bf v}){\rm d}{\bf v}
\quad{\rm  hence}\quad
 M_2(g)
 \le  4\int_{{\bR}}|{\bf v}|^2g({\bf v}){\rm d}{\bf v}\le 4\int_{{\bR}}|{\bf v}|^2 f({\bf v}){\rm d}{\bf v}.\lb{2.gg}\ee
Since, from the moment equation (\ref{moment2}),
\be\alpha=\pi^{-3/2}\Big(\fr{3}{2}\Big)^{3/2}
\fr{(\|g\|_{L^1})^{5/2}}{(M_2(g))^{3/2}},\quad \beta=\fr{3}{2}\fr{\|g\|_{L^1}}{M_2(g) },\lb{2.19}\ee
the inequalities in (\ref{2.12}) follow from
(\ref{2.15})-(\ref{2.19}),
$ \|f\|_{L^1}=4\pi\sqrt{2}N, \int_{{\bR}}|{\bf v}|^2 f({\bf v}){\rm d}{\bf v}=8\pi\sqrt{2}E$, and
$S(f)\ge S_0$.

{\bf Step2.} Let
$$M_g^*({\bf v})=(\alpha\wedge 1) e^{-\beta |{\bf v-v}_0|^2},\quad  \Og_g^*({\bf v})=\fr{M_g^*({\bf v})}{1-M_g^*({\bf v})}.$$
According to {\bf Proposition 2} in \cite{Lu2005}
(taking $\underline{b}(\cos(\theta))=\fr{b_0}{(4\pi)^2}, \Psi(r)=\Psi_0(r)$ in that proposition) we have
\be \|(1-M_g^*)(f-\Og_g^*)\|_{L^1}
\le C_{\alpha,\beta}(\|f\|_{L^1}+\|M_g^*\|_{L^1})\big( \|g-M_g^*\|_{L^1}
+\sqrt{D_{\rm min}(f)}\,\big)\lb{2005}\ee where
$$C_{\alpha,\beta}=C_{\alpha,\beta, b_0,\eta}=C_{b_0}
\Big(\int_{{\bR}}(\alpha\wedge 1) e^{-\beta|{\bf z}|^2}\Psi_0(|{\bf z}|){\rm  d}{\bf z}
\Big)^{-1}$$
$C_{b_0}>0$ depends only on $b_0$.

{\bf Step3.} Here we prove that
\be S(\Og)-S(\Og_g^*)\le C_6\big(\|g-M_g^*\|_{L^1}+
\sqrt{D_{\rm min}(f)}\,\big)^{1/2}.\lb{2.17}\ee
First from the assumption of the proposition we have
\be \int_{{\bR}}\Og({\bf v}) {\rm d}{\bf v}\le \int_{{\bR}}f({\bf v}){\rm d}{\bf v},\,\,
\int_{{\bR}}({\bf v}, |{\bf v}|^2/2)\Og({\bf v}){\rm d}{\bf v}=\int_{{\bR}}
({\bf v}, |{\bf v}|^2/2)f({\bf v}){\rm d}{\bf v}.\lb{2.new0}\ee
Since the function $y\mapsto s(y)$ given in \ref{entropyfunction} is concave,
it holds the inequality $s(y)-s(y_0)\le (\log(1+\fr{1}{y_0}))(y-y_0), y\ge 0,y_0>0,$ from which
and (\ref{2.new0}) we obtain
\bes
S(\Og)-S(\Og_g^*) &\le & \int_{{\bR}}\Big(\log\big(\fr{1}{\alpha\wedge 1}\big)+\beta|{\bf v-v}_0|^2\Big)(\Og({\bf v})-\Og_g^*({\bf v})){\rm d}{\bf v}\nonumber\\
&\le &\int_{{\bR}}\Big(\log\big(\fr{1}{\alpha\wedge 1}\big)+\beta|{\bf v-v}_0|^2\Big)(f({\bf v})-\Og_g^*({\bf v})){\rm d}{\bf v}.\lb{2.new1}\ees
Next using the inequalities
\be
\log\big(\fr{1}{\alpha\wedge 1}\big)\le \fr{1-\alpha\wedge 1}{\alpha\wedge 1},\quad
\fr{1-\alpha \wedge 1 +\beta|{\bf v-v}_0|^2}{1+\beta|{\bf v-v}_0|^2}\le 1-M_g^*({\bf v})
\lb{2.new3} \ee
it is easily seen that
\beas \log\big(\fr{1}{\alpha\wedge 1}\big)+\beta|{\bf v-v}_0|^2
\le \fr{1}{\alpha\wedge 1}\big(
1-M_g^*({\bf v})\big)\Big(1+\beta|{\bf v-v}_0|^2\Big).\eeas
From this, Cauchy-Schwarz inequality, and $0\le 1-M_g^*\le 1$ we obtain
\bes&&
\int_{{\bR}}\Big(\log\big(\fr{1}{\alpha\wedge 1}\big)+\beta|{\bf v-v}_0|^2\Big)|f({\bf v})-\Og_g^*({\bf v})|{\rm d}{\bf v}\nonumber \\
&&\le
\int_{{\bR}}\fr{1}{\alpha\wedge 1}\big(
1-M_g^*({\bf v})\big)(1+\beta |{\bf v-v}_0|^2)
|f({\bf v})-\Og_g^*({\bf v})|{\rm d}{\bf v}\nonumber\\
&&=\fr{\sqrt{2(\beta^2\vee 1)}}{\alpha\wedge 1}\big(\|(1-M_g^*)(f-\Og_g^*)\|_{L^1}\big)^{1/2}
 \left(\int_{{\bR}}\big(1+|{\bf v-v}_0|^4\big)
|f({\bf v})-\Og_g^*({\bf v})|{\rm d}{\bf v}\right)^{1/2}\nonumber
\\
&&\le C_7\big( \|g-M_g^*\|_{L^1}
+\sqrt{D_{\rm min}(f)}\,\big)^{1/2}\lb{2.new2}\ees
where the constant $C_7$ comes from (\ref{2005}),
$\|f\|_{L^1_4}\le C_0$, $|{\bf v}_0|\le C_5$, and
\beas
\int_{{\bR}}
(1+|{\bf v-v}_0|^4)\Og_g^*({\bf v})
{\rm d}{\bf v}
\le 4\pi
\int_{0}^{\infty}r^2
(1+r^4)\fr{1}{e^{\beta r^2}-1}
{\rm d}r<\infty.\eeas
Combining (\ref{2.new1}) and (\ref{2.new2})  we obtain
(\ref{2.17}).

{\bf Step4.} We next prove that
\be S(\Og_g^*)-S(f)\le C_8\big(\|g-M_g^*\|_{L^1}+
\sqrt{D_{\rm min}(f)}\,\big)^{1/2}.\lb{2.29}\ee
It is  easily seen from $s'(y)=\log(1+1/y)\le 3 y^{-1/3}$ that for all $y> y_0\ge 0$
\be s(y)-s(y_0)=(y-y_0)\int_{0}^{1}\log\Big(1+\fr{1}{y_0+\theta(y-y_0)}\Big){\rm d}\theta\le 5(y-y_0)^{2/3}.\lb{2.31}\ee
From these and Cauchy-Schwarz inequality we have
\bes&& S(\Og_g^*)-S(f)\nonumber\\
&&\le
\int_{\Og_g^*({\bf v})>f({\bf v})}[s(\Og_g^*({\bf v}))-s(f({\bf v}))]{\rm d}{\bf v}
\le 5
\int_{\Og_g^*({\bf v})>f({\bf v})}\big(\Og_g^*({\bf v})-f({\bf v})\big)^{2/3}{\rm d}{\bf v}\nonumber\\
&&\le 5
\Big(\|(1-M_g^*)(\Og_g^*-f)\|_{L^1}\Big)^{1/2}
\Big(\int_{{\bR}}
\fr{1}{1-M_g^*({\bf v})}(\Og_g^*({\bf v}))^{1/3}
{\rm d}{\bf v}\Big)^{1/2}.\lb{2.31}\ees
From the second inequality in (\ref{2.new3}) we have
$\fr{1}{ 1-M_g^*({\bf v})}\le \fr{1}{\beta |{\bf v-v}_0|^2}(1+\beta|{\bf v-v}_0|^2)$, so
the last integral in (\ref{2.31}) is finite
(using change of variable)
\beas&&\int_{{\bR}}
\fr{1}{1-M_g^*({\bf v})}(\Og_g^*({\bf v}))^{1/3}
{\rm d}{\bf v}=\int_{{\bR}}
\fr{\big(M_g^*({\bf v})\big)^{1/3}}{(1-M_g^*({\bf v}))^{4/3}}
{\rm d}{\bf v}\\
&&\le 4\pi\big(1/\beta\big)^{4/3}
\int_{0}^{\infty} \fr{1}{r^{2/3}}(1+\beta r^2)^{4/3}
((\alpha\wedge 1) e^{-\beta r^2})^{1/3}{\rm d}r\le C_{9}.\eeas
The inequality (\ref{2.29}) then follows from this, (\ref{2.31}) and
(\ref{2005}).

{\bf Step5.} Now we prove that
\be S(\Og)-S(f)\le C_{10}\big( (\|g-M_g\|_{L^1})^{1/5}+
(D_{\rm min}(f))^{1/4}\big).\lb{2.34}\ee
In fact first we have from (\ref{2.17}) and (\ref{2.29}) that
\be S(\Og)-S(f)\le
C_{11}\big(\|g-M_g^*\|_{L^1}+
\sqrt{D_{\rm min}(f)}\,\big)^{1/2}.\lb{2.35}\ee
If $\alpha\le 1$, then $M_g^*=M_g$ and so from (\ref{2.35}) we see that (\ref{2.34}) holds.

Suppose $\alpha>1$. From
$\|g\|_{L^1}=\|M_g\|_{L^1}$ and
$0\le g({\bf v})\le 1$ we have
$$
\|g-M_g\|_{L^1}= 2\int_{{\bR}}(M_g({\bf v})-g({\bf v}))_{+}{\rm d}{\bf v}
\ge
8\pi \int_{0}^{\infty}r^2(\alpha e^{-\beta r^2}-1)_{+}{\rm d}r$$
and using the inequality  $
 \alpha e^{-\beta r^2}-1\ge \alpha \beta
\big(\fr{\alpha-1}{\alpha\beta}-r^2\big)$ to the last term we have
$$
8\pi \int_{0}^{\infty}r^2(\alpha e^{-\beta r^2}-1)_{+}{\rm d}r
\ge
8\pi \int_{0}^{\infty}r^2\alpha \beta \Big(\fr{\alpha-1}{\alpha\beta}-r^2\Big)_{+}{\rm d}r
=\fr{16\pi}{15} \fr{(\alpha-1)^{5/2}}{(\alpha\beta)^{3/2}}$$
so that we deduce
\beas&&\alpha-1\le \Big(\fr{15}{16\pi} (\alpha\beta)^{3/2}\Big)^{2/5}
\big(\|g-M_g\|_{L^1}\big)^{2/5},\\
&&\|M_g-M_g^*\|_{L^1}=
(\alpha-1) \int_{{\bR}}e^{-\beta|{\bf v-v}_0|^2}
{\rm d}{\bf v}\le C_{12}\big(\|g-M_g\|_{L^1}\big)^{2/5},\\
&&
\|g-M_g^*\|_{L^1}\le
\|g-M_g\|_{L^1}+
\|M_g-M_g^*\|_{L^1}
\le C_{13}\big(\|g-M_g\|_{L^1}\big)^{2/5}.\eeas
Inserting $\|g-M_g^*\|_{L^1}\le
C_{13}\big(\|g-M_g\|_{L^1}\big)^{2/5}$ into the right hand side of (\ref{2.35}) gives
(\ref{2.34}).

{\bf Step6.}  We will use
 Csiszar-Kullback inequality
\be \|g-M_{g}\|_{L^1}\le \sqrt{2\|g\|_{L^1}}\sqrt{H(g\,|\,M_{g})}\lb{CK}\ee
and Villani's inequality (\ref{Villani}) to finish the proof of (\ref{e-e}). To do this we need to normalize $g, M_g$. Let
$$
\wt{g}({\bf v})=\mu g(\ld {\bf v+v}_0),\quad \wt{M_g}({\bf v})=
\mu M_g(\ld {\bf v+v}_0)=\mu \alpha e^{-\ld\beta |{\bf v}|^2}$$
where
\be \ld=3^{-1/2}
\Big(\fr{M_2(g)}{\|g\|_{L^1}}\Big)^{1/2},\quad \mu= 3^{-3/2} \fr{(M_2(g))^{3/2}}{(\|g\|_{L^1})^{5/2}}.
\lb{LM}\ee
Then $\wt{M_g}({\bf v})=M({\bf v})=(2\pi)^{-3/2}e^{-|{\bf v}|^2/2}$ and
$\wt{g}\in {\cal H}_0$ with the constant
$$H_0= \log\Big(\fr{(C_2)^{3/2}}{3^{3/2}(C_1)^{5/2}}\Big) +\fr{3}{2}\log(2\pi e)$$
where $C_1, C_2$ are the constants in (\ref{2.12}).
In fact from (\ref{LM}) and
$0\le g\le 1$  we have
$\ld^{-3}\mu \|g\|_{L^1}=1$ and $H(g)\le 0$, and so
$0\le H(\wt{g}\,|\wt{M_g}) \le \log(\mu)-H(\wt{M_g})\le H_0.$
Applying Villani's inequality (\ref{Villani}) to
$\wt{g}, \wt{M_g}$ we have
\be \fr{1}{\|g\|_{L^1}}H(g|M_g)=H(\wt{g}\,| \wt{M_g})\le \fr{1}{c_{H_0}}{\cal D}_2(\wt{g}).\lb{Villani2}\ee
Since
\beas
{\cal D}_2(\wt{g})&=&\fr{1}{4}\int_{{\bRRS}}\fr{2|({\bf v-v}_*)\cdot\og|}{|{\bf v-v}_*|}
(1+|{\bf v-v}_*|^2)
\Gm(\wt{g}'\wt{g}_*',\wt{g}\wt{g}_*)
{\rm d}{\bf v}{\rm d}{\bf v}_*{\rm d}\og\\
&=&\fr{\mu^2}{\ld^8}\fr{1}{4}\int_{{\bRRS}}\fr{2|({\bf v-v}_*)\cdot\og|}{|{\bf v-v}_*|}(\ld^2+|{\bf v-v}_*|^2)
\Gm(g'g_*',gg_*){\rm d}{\bf v}{\rm d}{\bf v}_*{\rm d}\og
\\
&\le& \fr{\mu^2}{\ld^8}\max\{\ld^2, 1\}{\cal D}_2(g)\le C_{14}{\cal D}_2(g),\eeas
it follows from (\ref{CK}, (\ref{Villani2}) that
$$
 \|g-M_{g}\|_{L^1}\le
\sqrt{2}\|g\|_{L^1}\sqrt{H(\wt{g}\,|\,\wt{M}_{g})}
\le C_{15}\sqrt{{\cal D}_2(g)}.$$
Inserting this into (\ref{2.34}) gives (\ref{e-e}) with $C=C_{10}\max\{(C_{15})^{1/5}, 1\}$ and
completes the proof.
\end{proof}

\begin{center}\section{Positive Lower Bound of Entropy}\end{center}

In this section we prove that if an initial datum $F_0\in {\cal B}_{1}^{+}({\mR}_{\ge 0})$
has positive energy, $E(F_0)>0,$ then
there is a conservative measure-valued solution $F_t$
of Eq. with $F_{t=0}=F_0$, such that $S(F_t)>0$ for all $t>0$.  This is equivalent to saying that $F_t$ is non-singular for all $t>0$ even if $F_0$ is singular. To do this we first prove that
the entropies of isotropic approximate solutions have a uniform positive lower bound.
For convenience of stating approximate solutions under consideration, we introduce
a definition of a class of approximate  solutions:

\begin{definition}\label{definition3.1} Let $B({\bf {\bf v-v}_*},\omega)$ be given by
(\ref{kernel}), (\ref{Phi}).
We say that $\{B_K({\bf {\bf v-v}_*},\omega)\}_{K\in {\mN}}$ is a sequence of approximation of $B({\bf {\bf v-v}_*},\omega)$
if
$B_K({\bf {\bf v-v}_*},\omega)$ are such Borel measurable functions on ${\bRS}$ that they are functions of
$(| {\bf v}-{\bf v}'|, |{\bf v}-{\bf v}_*'|)$ only and satisfy
$$B_K({\bf v-v}_*,\omega)\ge 0, \quad \lim\limits_{K\to\infty}
B_K({\bf v-v}_*,\omega)=
B({\bf v-v}_*,\omega)$$ for all $({\bf v-v}_*,\omega)
\in {\bR}\times {\bS}$.
Let $Q_K(f)$ be the collision integral operators corresponding to the approximate kernels $B_K$, i.e.
\beas Q_K(f)({\bf v})=
\int_{{\bRS}}B_K({\bf
{\bf v-v}_*},\og)\big(f'f_*'(1+f+f_*)-ff_*(1+f'+f_*')\big) {\rm d}\omega{\rm
d}{\bf v_*}.\eeas
Given any $K\in {\mN}$ and $0\le f^K_0\in L^1_2({\bR})$. We
 say that $f^K=f^K({\bf v},t)$ is a conservative approximate solution of Eq.(\ref{Equation1}) on ${\bR}\times [0,\infty)$ corresponding to the approximate kernel $B_K$  with the initial datum $f^K_0$ if
$({\bf v},t)\mapsto f^K({\bf v},t)$ is a nonnegative Lebesgue measurable function on ${\bR}\times [0,\infty)$ satisfying

{\rm(i)}
$\sup_{t\ge 0}\|f^K(t)\|_{L^1_2}<\infty$ (here and below $f^K(t):=f^K(\cdot,t)$) and
\be\int_{0}^{T}{\rm d}t
\int_{{\bRRS}}B_K({\bf
{\bf v-v}_*},\og)(f^K)'(f^K)_*'\big(1+f^K+f^K_*\big)\sqrt{1+|{\bf v}|^2+|{\bf v}_*|^2}\,{\rm d}\omega{\rm
d}{\bf v}{\rm d}{\bf v_*}<\infty\lb{3.new1}\ee
for all $0<T<\infty$.

{\rm(ii)}  There is a null set
$Z\subset {\bR}$ which is independent of\, $t$ such that
\be
 f^K({\bf v},t)=f^K_0({\bf v})+\int_{0}^{t}Q_K(f^K)({\bf v},\tau){\rm d}\tau\quad \forall\, t\in[0,\infty),\,\,\forall\,{\bf v}\in {\bR}\setminus Z.\lb{3.new2}\ee

{\rm(iii)}  $f^K$ conserves the mass, momentum, and energy, and satisfies the entropy
 equality, i.e.
\bes &&
\int_{{\bR}}(1,{\bf v}, |{\bf v}|^2/2) f^K({\bf v},t){\rm d}{\bf v}=
\int_{{\bR}}(1,{\bf v}, |{\bf v}|^2/2) f^K_0({\bf v}){\rm d}{\bf v}\qquad \forall\, t\ge 0
\label{energyconser} \\
&& S(f^K(t))=S(f^K_0)+\int_{0}^{t}D_K(f^K(\tau)){\rm d}\tau\qquad\forall\, t\ge 0.
\dnumber\label{eelity}\ees
Here $Q_K(f^K)({\bf v},t)=Q_K(f^K(\cdot,t))({\bf v})$, $D_K(f)$ is the entropy dissipation corresponding to the approximate kernel $B_K({\bf
{\bf v-v}_*},\og)$ defined as in (\ref{2.Gamma})-(\ref{diss}), i.e.
\be D_K(f)=\fr{1}{4}
\int_{{\bRRS}}B_K({\bf
{\bf v-v}_*},\og)\Pi(f)\Gm(g'g_*', gg_*) {\rm d}\omega{\rm
d}{\bf v_*}{\rm
d}{\bf v}\label{3.2}\ee
with $g=f/(1+f)$.

\noindent If a conservative approximate  solution $f^K$ is isotropic, i.e. if
$f^K({\bf v},t)\equiv f^K(|{\bf v}|^2/2,t)$,
then $f^K$ is called a conservative isotropic approximate
 solution of Eq.(\ref{Equation1}).
\end{definition}

By using change of variables one sees that the integral in the left hand side of (\ref{3.new1})
is equal to that where
$(f^K)'(f^K)_*'\big(1+f^K+f^K_*\big)$ is replaced by
$f^Kf^K_*\big(1+(f^K)'+(f^K)_*'\big)$. Thus (\ref{3.new1}) not only implies
$Q_K(f^K)\in L^1({\bR}\times [0,T])$ for all $0<T<\infty$ so that the integral in
right hand said of \ref{3.new2}) is absolutely convergent for all ${\bf v}\in {\bR}\setminus Z$ and all
$t\ge 0$,  but also enables us to prove some important relations between
entropy and entropy dissipation for approximate solutions as we will do
in the proof of Proposition \ref{prop4.2} below.

Also as one sees from (\ref{3.new1}) that a main role of an approximation $B_K$ of $B$ is to ensure the absolute convergence of the
cubic collision integrals.
A suitable class of such $B_K$ that had been used before
is
\be B_K({\bf {\bf v-v}_*},\omega)=\min\big\{B({\bf {\bf v-v}_*},\omega),\, K|{\bf v}-{\bf v}'|^2
 |{\bf v}-{\bf v}_*'|\big\},\quad K\ge 1 \label{3.1}\ee
 which works well at least for isotropic approximate
 solutions (see e.g.\cite{Lu2000}).  In fact
using $\sqrt{1+|{\bf v}|^2+|{\bf v}_*|^2}
\le \la {\bf v}\ra \la{\bf v}_*\ra$ one has,
as proved in \cite{Lu2000}, that
for all isotropic functions $f, g\in L^1_1({\bR}), h\in L^1({\bR})$
(recall that $f$ is isotropic means that $f({\bf v})$ depends only on
$|{\bf v}|$ )
\beas&& \int_{{\bRRS}}B_K(
{\bf v-v}_*,\og)|f({\bf v}')g({\bf v}_*')h({\bf v})|\sqrt{1+|{\bf v}|^2+|{\bf v}_*|^2}\,{\rm d}\omega{\rm
d}{\bf v_*}{\rm
d}{\bf v}\\
&&=
 \int_{{\bRRS}}B_K({\bf v-v}_*,\og)|f({\bf v})g({\bf v}_*)h({\bf v}')| \sqrt{1+|{\bf v}|^2+|{\bf v}_*|^2}\,{\rm d}\omega{\rm
d}{\bf v_*}{\rm
d}{\bf v}\\
&&\le 2K\|f\|_{L^1_1}\|g\|_{L^1_1}\|h\|_{L^1}.
\eeas

To prove the main result of this section, we need two  lemmas:

\begin{lemma}\lb{lemma.entropy} Let $F\in {\cal B}_2^{+}({\bR})$ and suppose
\be\rho=\int_{{\bR}}{\rm d}F({\bf v})>0,\quad
{\bf u}=\fr{1}{\rho}\int_{{\bR}}{\bf v}{\rm d}F({\bf v}),\quad
T=\fr{1}{3\rho}\int_{{\bR}}|{\bf v}-{\bf u}|^2{\rm d}F({\bf v})
>0.\lb{5.FFF}\ee
Let $S(F)$ be defined by (\ref{ent1}),(\ref{1.20}),(\ref{1.21}), and let
$0\le f\in L^1_2({\bR})$ be the regular part of $F$. Then
$S(F)=S(f)$.
\, Moreover there is a sequence $0\le f_n\in  L^1_2({\bR})\,(n\in{\mN})$
satisfying
$$\int_{{\bR}}(1, {\bf v}, |{\bf v}|^2/2)f_n({\bf v}){\rm d}{\bf v}=
\int_{{\bR}}(1, {\bf v}, |{\bf v}|^2/2){\rm d}F({\bf v})\quad \forall\,n\in{\mN}$$
such that the weak convergence (\ref{1.21}) holds and
\beas S(f_n)\ge \Big(1-\fr{1}{2n}\Big)S(F)\quad \forall\,n\in {\mN}. \eeas
Consequently  we have $\lim\limits_{n\to\infty}S(f_n)=S(F)$.
Furthermore, if $F$ is also isotropic, i.e. if $F$ is
defined by a measure $\wt{F}\in {\cal B}_1^{+}({\mR}_{\ge 0})$ through
(\ref{m2}) or (\ref{m4}), then all $f_n$ can be also isotropic: $f_n({\bf v})=\wt{f}_n(|{\bf v}|^2/2)$, and thus $\wt{f}_n$ converges weakly to $\wt{F}$, i.e.
\be\lim_{n\to\infty}\int_{{\mR}_{\ge 0}}\vp(x)\wt{f}_n(x)\sqrt{x}{\rm d}x=
\int_{{\mR}_{\ge 0}}\vp(x){\rm d}\wt{F}(x)
\quad \forall\, \vp\in C_b({\mR}_{\ge 0}).\lb{weak-isotropic}\ee
\end{lemma}

\begin{proof} Write
${\rm d}F({\bf v})=f({\bf v}){\rm d}{\bf v}+{\rm d}\nu({\bf v})$ with
$0\le f\in L^1_2({\bR})$ and
$\nu\in {\cal B}_2^{+}({\bR})$ is the singular part of $F$.
We have proved in Lemma 3 of \cite{Lu2005} that $S(F)\le S(f)$.
The proof of the converse inequality $S(F)\ge S(f)$ is included in the proof of the
second part of this lemma.

For every $n\in{\mN}$, let $\mu_n\in {\cal B}_2^{+}({\bR})$ be defined by
${\rm d}\mu_n({\bf v})=\fr{1}{2n}f({\bf v})+{\rm d}\nu({\bf v})$.
From the assumption (\ref{5.FFF}) it is easily seen that
$\rho_n:=\int_{{\bR}}{\rm d}\mu_n({\bf v})>0,
T_n=\fr{1}{3\rho_n} \int_{{\bR}}|{\bf v}-{\bf u}_n|^2{\rm d}\mu_n({\bf v})
>0$, where ${\bf u}_n=\fr{1}{\rho_n} \int_{{\bR}}{\bf v}{\rm d}\mu_n({\bf v}).$
Also by considering  $\int_{{\bR}}{\rm d}\nu({\bf v})>0$ and
$\int_{{\bR}}{\rm d}\nu({\bf v})=0$ respectively it is easily seen that
$\sup\limits_{n\ge 1}\fr{1}{\rho_n}\int_{{\bR}}|{\bf v}|^2{\rm d}\mu_n({\bf v})<\infty$ and thus
\be \sup_{n\ge 1}\{\rho_n, |{\bf u}_n|, T_n\}<\infty.\lb{RUT}\ee
Let $0\le h_n\in L^1_2({\bR})$ be the Mehler transform of $\mu_n$, i.e.
\beas h_n({\bf v})=
e^{3n}\int_{{\bR}}M_{1,0,T_n}\Big(e^{n}\big({\bf v}-{\bf u}_n-\sqrt{1-e^{-2n}}
({\bf v}_*-{\bf u}_n)\big)\Big)
{\rm d}\mu_n({\bf v}_*),\quad {\bf v}\in {\bR}\eeas
where
$M_{1,0,T_n}({\bf v})=(2\pi T_n)^{-3/2}\exp(-\fr{|{\bf v}|^2}{2T_n}).$
As is well-known,
for any Borel measurable function $\psi$ on ${\bR}$ satisfying
$\sup_{{\bf v}\in {\bR}}|\psi({\bf v})|(1+|{\bf v}|^2)^{-1}<\infty$, we have (see e.g. \cite{Lu2004})
\beas&&
\int_{{\bR}}\psi({\bf v})h_n({\bf v})
{\rm d}{\bf v}
=\int_{{\bR}}M_{1,0,1}({\bf v})\left(\int_{{\bR}}
\psi\big(e^{-n}T_n^{1/2}{\bf v}+{\bf u}_n+\sqrt{1-e^{-2n}}({\bf v}_*-{\bf u}_n)\big){\rm d}\mu_n({\bf v}_*)\right){\rm d}{\bf v},\\
&&
\int_{{\bR}}(1, {\bf v}, |{\bf v}|^2/2)h_n({\bf v})
{\rm d}{\bf v}=
\int_{{\bR}}(1, {\bf v}, |{\bf v}|^2/2)
{\rm d}\mu_n({\bf v})\quad \forall\, n\in{\mN}.\eeas
Then together with (\ref{RUT}) it is easily proved the weak convergence:
\beas\lim_{n\to\infty}
\int_{{\bR}}\psi({\bf v})h_n({\bf v})
{\rm d}{\bf v}=\lim_{n\to\infty}
\int_{{\bR}}\psi({\bf v})
{\rm d}\mu_n({\bf v})=\int_{{\bR}}\psi({\bf v})
{\rm d}\nu({\bf v})\quad \forall\, \psi\in C_b({\bR}).\eeas
Let
$$f_n({\bf v})=\Big(1-\fr{1}{2n}\Big)f({\bf v})+h_n({\bf v}).$$
Then $0\le f_n\in L^1_2({\bR})$ have the same mass, momentum,
and energy as $F$, and $f_n$ converge weakly
 to $F$ in the sense of (\ref{1.21}). Using the entropy properties (\ref{S2}), (\ref{S1})
 we have
$S(f_n)\ge (1-\fr{1}{2n})S(f)+\fr{1}{2n}S(2n h_n)\ge (1-\fr{1}{2n})S(f)$ for all $n\ge 1$.
 This together with $\limsup\limits_{n\to\infty}
S(f_n)\le S(F)$ implies $S(f)\le S(F)$ and thus we conclude $S(F)=S(f)$ and
 $\lim\limits_{n\to\infty}S(f_n)= S(F)$.

Finally suppose further that $F$ is defined by a measure $\wt{F}\in {\cal B}_1^{+}({\mR}_{\ge 0})$
through (\ref{m2}) or (\ref{m4}). Write ${\rm d}\wt{F}(x)=\wt{f}(x)\sqrt{x}\,{\rm d}x+{\rm d}\wt{\nu}(x)$
where $0\le \wt{f}\in L^1({\mR}_{\ge 0}, (1+x)\sqrt{x}\,{\rm d}x)$ and $\wt{\nu}$
is the singular part of $\wt{F}$. Then, by uniqueness of the regular-singular decomposition,
we have $f({\bf v})=\wt{f}(|{\bf v}|^2/2)$ (up to a null set) and $\nu$ is equal to
the measure defined by $\wt{\nu}$ through (\ref{m2}) or equivalently (\ref{m4}). In this case we have
${\bf u}_n={\bf 0}$ and it is easily seen that
$h_n({\bf v})=: \wt{h}_n(|{\bf v}|^2/2)$ are  isotropic and thus
$f_n({\bf v})= (1-\fr{1}{2n})\wt{f}(|{\bf v}|^2/2)+\wt{
h}_n(|{\bf v}|^2/2)=: \wt{f}_n(|{\bf v}|^2/2)$ are  isotropic. The weak
convergence (\ref{weak-isotropic}) follows from the weak convergence  (\ref{1.21})
of $f_n$ to $F$.
\end{proof}
\vskip2mm

\begin{lemma}\lb{prop3.1}Let $f,g, h:{\mR}_{\ge 0}\to {\mR}_{\ge 0}$ be Lebesgue measurable and
$0<a\le b\wedge  c$.
Then
\bes&&\int_{{\bRRS}}|({\bf v}-{\bf v}_*)\cdot\omega|
1_{\{|{\bf v}|\le a\}}1_{\{|{\bf v}'|\ge b\}}1_{\{|{\bf v}_*'|\ge c\}}f(|{\bf v}|^2/2)
g(|{\bf v}'|^2/2)h(|{\bf v}_*'|^2/2){\rm d}\og{\rm d}{\bf v}_*{\rm d}{\bf v}\nonumber\quad\\
&&
=\left(\int_{|{\bf v}|\le a}f(|{\bf v}|^2/2){\rm d}{\bf v}\right)
\left(\int_{|{\bf v}|\ge b}\fr{1}{|{\bf v}|}g(|{\bf v}|^2/2){\rm d}{\bf v}\right)
\left(\int_{|{\bf v}|\ge c }\fr{1}{|{\bf v}|}h(|{\bf v}|^2/2){\rm d}{\bf v}\right).\lb{entropy.1}\ees
\end{lemma}

\begin{proof} Using the first equality in (\ref{5.4new}) (see Appendix) to $\Phi(r,\rho)\equiv (4\pi)^2,
\Psi(x,y,z)=1_{\{\sqrt{2x}\le a\}}1_{\{\sqrt{2y}\ge b\}}1_{\{\sqrt{2z}\ge c\}}f(x)g(y)h(z)$ and
using (\ref{1.difference}) and
noting that $\min\{\sqrt{x},\sqrt{x_*},\sqrt{y},\sqrt{z}\}=\sqrt{x}$ for
$0\le x\le y\wedge z$, we have
\beas&&{\rm the\, \,l.h.s.\,\,of\,(\ref{entropy.1})}
=\sqrt{2}(4\pi)^3 \int_{{\mR}_{\ge 0}^3}\sqrt{x}
1_{\{\sqrt{2x}\le a\}}1_{\{\sqrt{2y}\ge b\}}1_{\{\sqrt{2z}\ge c\}}f(x)g(y)h(z)
{\rm d}x{\rm d}y{\rm d}z\\
&&={\rm the\, \,r.h.s.\,\,of\,(\ref{entropy.1})}.
\eeas
\end{proof}
\vskip1mm

To prove the main result Proposition \ref{prop3.2},  we will also use the following inequalities:
if $0\le f\in L^1_2({\bR})$, then
\be mes(\{{\bf v}\in{\bR}\,\,|\,\, f({\bf v})> 1/3\})\le 2S(f),\quad
\int_{f({\bf v})< 9}f({\bf v}){\rm d}{\bf v}
\le 5S(f).\lb{entropy.2}\ee
In fact for the function $s(y)=(1+y)\log(1+y)
-y\log y$ we have $s(y)\ge 2\fr{y}{1+y}, y\ge0.$
So
$1_{\{f({\bf v})> 1/3\}}s(f({\bf v})) \ge \fr{1}{2}1_{\{f({\bf v})>1/3\}}$,
 $ 1_{\{f({\bf v})< 9\}}s(f({\bf v}))\ge \fr{1}{5}1_{\{f({\bf v})< 9\}} f({\bf v}),$
hence (\ref{entropy.2}) follows.
\vskip2mm

\begin{proposition}\lb{prop3.2} Let $B({\bf {\bf v-v}_*},\omega)$ satisfy Assumption \ref{assp}. Given any $N>0, E>0$.

{\rm(I)} Let $B_K({\bf {\bf v-v}_*},\omega)$
be given by (\ref{3.1}) (or equivalently (\ref{3.kernel1}),(\ref{3.kernel2})),
let $\{f_0^K=f^K_0(|{\bf v}|^2/2)\}_{K\in{\mN}} $ be any sequence of nonnegative
isotropic functions in $L^1_2({\bR})$  satisfying
\be \int_{{\bR}}(1, |{\bf v}|^2/2)f^K_0(|{\bf v}|^2/2){\rm d}{\bf v}= 4\pi\sqrt{2}(N, E)\qquad
\forall\, K\in{\mN}.\lb{entropy.3}\ee
Then for every $K\in{\mN}$, there exists a unique conservative isotropic approximate  solution
$f^K=f^K(|{\bf v}|^2/2,t)$ of
Eq.(\ref{Equation1}) on ${\bR}\times [0,\infty)$  corresponding to the approximate kernel
 $B_K$ such that $f^K|_{t=0}=f^K_0$, and
it holds the moment production:
\be \sup_{K\in{\mN}}\|f^K(t)\|_{L^1_s}\le C_s(1+1/t)^{s-2}\quad \forall\, t>0,\,\,\forall\, s>2 \label{mprod}\ee
where the constant $0<C_s<\infty$ depends only on $N, E,b_0,\eta$ and $s$.

Moreover for any $t_0>0$, let
\be S_{*}(t_0)=\min\Bigg\{\fr{7\pi a^3}{24},\,\fr{4\pi^2 E^2}{5C(1+2/t_0)^2},\,
b_0\min\{1, (2a^2)^{1\vee \eta}\}
\fr{7\pi^4 \sqrt{2}a^3E^5 t_0}{96C^3(1+2/t_0)^6}
\Bigg\}\lb{entropy.4}\ee
where $a=\fr{1}{2}\sqrt{E/N},$ $0<C=C_4<\infty$ is the constant in (\ref{mprod}) for $s=4$ so that
$C$ depends only on $N, E,b_0,\eta$.
Then
\be S(f^K(t))\ge S(f^K(t_0))\ge S_{*}(t_0)\qquad \forall\, t\ge t_0,\,\,\forall\, K\in {\mN}.
\lb{entropy3.5}\ee

{\rm(II)}
Let $F_0\in {\mathcal B}_1^{+}({\mathbb R}_{\ge 0})$ satisfy
$N(F_0)=N, E(F_0)=E$.  Then there exists a conservative  measure-valued isotropic solution $F_t$ of Eq.(\ref{Equation1}) on $[0,\infty)$ with the initial datum $F_0$, such that
\be S(F_t)\ge S_{*}(t_0)\qquad \forall\, t\ge t_0\lb{entropy3.6}\ee
for all $t_0>0$, and
\be M_{p}(F_t)\le C_p(1+1/t)^{2p-2}\quad \forall\,t>0,\,\,\forall\, p>1\lb{3.moment}\ee
where the constant $0<C_p<\infty$ depends only on
$N, E,b_0,\eta$ and $p$.

Moreover there exists a sequence $\{f^K\}_{K\in{\mN}}$
of a conservative isotropic approximate solutions of  Eq.(\ref{Equation1}) on ${\bR}\times [0,\infty)$ presented in part {\rm(I)}  such that their initial data sequence $\{f^K_0\}_{K\in{\mN}}$ satisfies
\be \lim_{K\to\infty}S(f^K_0)=S(F_0)\lb{entropy3.7}\ee
and there is a subsequence $\{f^{K_n}\}_{n\in{\mN}}$ such that $F_t$ and the measure
$\bar{F}_t$  defined by $F_t$ through (\ref{m2}) or (\ref{m4}) are weak limits of $f^{K_n}(x,t)$ and $f^{K_n}(|{\bf v}|^2/2,t)$ respectively,  i.e.
\be\int_{{\mR}_{\ge 0}}\vp(x){\rm d}F_t(x)= \lim_{n\to\infty}\int_{{\mR}_{\ge 0}}
\vp(x)f^{K_n}(x,t)\sqrt{x}{\rm d}x
\quad \forall\, \vp\in C_b({\mR}_{\ge 0}),\,\,\forall\, t\ge 0\lb{3.weak1}\ee
\be\int_{{\bR}}\psi({\bf v}){\rm d}\bar{F}_t({\bf v})= \lim_{n\to\infty}\int_{{\bR}}\psi({\bf v})f^{K_n}(|{\bf v}|^2/2,t){\rm d}{\bf v}
\quad \forall\, \psi\in C_b({\bR}),\,\,\forall\, t\ge 0.\lb{3.weak2}\ee
\end{proposition}

\begin{proof} (I): Let us rewrite (\ref{3.1}) as
\bes&& B_K({\bf v},-{\bf v}_*,\og)=
\fr{|({\bf v-v}_*)\cdot\omega|}{(4\pi)^2}\Phi_K(|{\bf v}-{\bf v}'|, |{\bf {\bf v}-{\bf v}_*'}|),\lb{3.kernel1}\\
&&{\rm with}\quad \Phi_K(r,\rho)=\min\big\{
\Phi(r,\rho),\,
(4\pi)^2K r\rho\big\}.\dnumber\lb{3.kernel2}\ees
The existence, uniqueness and
the uniform moment production (\ref{mprod}) of the
conservative isotropic approximate  solutions
$f^K$ have been essentially proven in the proofs of
Theorem 2, Theorem 3 in \cite{Lu2000} and
Theorem 4 in \cite{Lu2004}.  The only difference
from the present case is that \cite{Lu2000}, \cite{Lu2004}
consider those approximate kernels $B_K({\bf v-v}_*,\og)$ in (\ref{3.1}) where
$B({\bf v-v}_*,\og)=b(\cos(\theta))|{\bf v-v}_*|^{\gm}$
(with $\theta =\arccos(|({\bf v-v}_*)\cdot\omega|/|{\bf {\bf v-v}_*}|)$
with
$0<\gm\le 1 $ and $0\le b(\cdot)\in C([0,1]), b(0)=0, \int_{0}^{1}b(\tau){\rm d}\tau>0$.
However, for the present case, one sees from (\ref{kerneB}) that for the case $|{\bf v-v}_*|\ge 1$, the
the lower bound $\fr{b_0}{(4\pi)^2}\cos(\theta)
|{\bf v-v}_*|$ and the  upper bound
$\fr{1}{(4\pi)^2}\cos(\theta)|{\bf v-v}_*|$ of
$B({\bf {\bf v-v}_*},\omega)$ have the same form as used in \cite{Lu2004}
with $\gm=1$, and thus
the proof of  (\ref{mprod})  is completely the same as the proof of
Theorem 4 in \cite{Lu2004} (and thus the proof of the uniqueness of $f^K$ is also
valid for the present kernel $B_K$.)
In fact to see this is really the case, one needs only to prove
the following inequality
\be \int_{{\bRRS}}B_K({\bf v-v}_*,\og) ff_*[\kappa(\theta)]^s
|{\bf v}|^s {\rm d}\og{\rm d}{\bf v}{\rm d}{\bf v}_*\ge
 \fr{A_s}{2}\Big( \|f\|_{L^1}M_{s+1}(f)-(\|f\|_{L^1})^2\Big)\lb{MM}\ee
where
$$A_s=4\pi\int_{0}^{\pi/2}\min\big\{\fr{b_0}{(4\pi)^2},\, \cos(\theta)\sin(\theta)\big\}\sin(\theta)\cos(\theta) [\kappa(\theta)]^s{\rm d}\theta,$$
$\kappa(\theta)=\min\{\cos(\theta), 1-\cos(\theta)\}$, and
$M_s(f)=\int_{{\bR}}|{\bf v}|^s f({\bf v}){\rm d}{\bf v}$.
To do this we first note that from (\ref{3.1}) and Assumption \ref{assp} we have
$$B_K({\bf v-v}_*,\og)\ge |{\bf v}-{\bf v}_*|\cos(\theta)\min\big\{\fr{b_0}{(4\pi)^2},\,\cos(\theta)\sin(\theta)\big\}
\min\{1, |{\bf v-v}_*|^{2\vee 2\eta}\}.$$
Then
for all $s>0$, $0\le f({\bf v})=f(|{\bf v}|^2/2)\in L^1_{s+1}({\bR})$, we have
\bes&& \int_{{\bRRS}}B_K({\bf v-v}_*,\og) ff_*[\kappa(\theta)]^s
|{\bf v}|^s {\rm d}\og{\rm d}{\bf v}{\rm d}{\bf v}_*\nonumber\\
&&\ge
A_s\int_{{\bRR}}ff_*|{\bf v}|^s|{\bf v-v}_*|\min\{1, |{\bf v-v}_*|^{2\vee 2\eta}\}
{\rm d}{\bf v}{\rm d}{\bf v}_*\lb{AP}\ees
To estimate the right-hand side of (\ref{AP}), we will use the following inequality (the proof is esiy): if $\Psi$ is nonnegative Lebesgue
measurable on ${\mR}_{\ge 0}^3$ satisfying that  $\rho\mapsto \Psi(r,r_*, \rho)$ is non-decreasing on $[0,\infty)$
for all $r,r_*\ge 0$,  then (using spherical coordinates transform)
\beas \int_{{\bRR}}\Psi(|{\bf v}|,|{\bf v}_*|, |{\bf v- v}_*|){\rm d}{\bf v}{\rm d}{\bf v}_*
\ge \fr{1}{2}\int_{{\bRR}}\Psi(|{\bf v}|,|{\bf v}_*|,|{\bf v}|){\rm d}{\bf v}{\rm d}{\bf v}_*.\eeas
Using this inequality and noting that $\min\{1,|{\bf v}|^{2\vee 2\eta}\}=1-(1-|{\bf v}|^{2\vee 2\eta})_{+}$ we obtain
(\ref{MM}):
\beas&&
\int_{{\bRR}}ff_*|{\bf v}|^s|{\bf v-v}_*|\min\{1, |{\bf v-v}_*|^{2\vee 2\eta}\}
{\rm d}{\bf v}{\rm d}{\bf v}_*
\\
&&\ge \fr{1}{2}\int_{{\bRR}}ff_*|{\bf v}|^{s+1}|\min\{1, |{\bf v}|^{2\vee 2\eta}\}
{\rm d}{\bf v}{\rm d}{\bf v}_*\ge
 \fr{1}{2}\Big( \|f\|_{L^1}M_{s+1}(f)-(\|f\|_{L^1})^2\Big).
\eeas
\vskip2mm

Now we are going to prove the positive lower bound (\ref{entropy3.5}) of entropy, which is the new thing of
the proposition. Given any $t_0>0, K\in {\mN}$.
For notation convenience we denote
$$f(|{\bf v}|^2/2,t):=f^K(|{\bf v}|^2/2,t).$$
Since $t\mapsto S(f(t))$ is non-decreasing, we need only
prove that $S(f(t_0))\ge S_{*}(t_0)$.  To do this we may assume that (recall
$S_{*}(t_0)$ in (\ref{entropy.4}))
\be S(f(t_0))\le \min\Big\{\fr{7\pi a^3}{24},\,\fr{4\pi^2 E^2}{5C(1+2/t_0)^2}\Big\}.\lb{entropy3.8} \ee
Recall $a= \fr{1}{2}\sqrt{E/N}$ and let
$$b=\Big(\fr{C (1+2/t_0)^2}{2\pi\sqrt{2} E}\Big)^{1/2}.$$
 By conservation of mass and energy and moment production (\ref{mprod}) we have
\beas&&
\int_{|{\bf v}|\le 2a}\fr{1}{2}|{\bf v}|^2 f(|{\bf v}|^2/2,t){\rm d}{\bf v}
\le \fr{4a^2}{2}4\pi\sqrt{2}N=2\pi\sqrt{2}E,\\
&&
\int_{|{\bf v}|> 2a}\fr{1}{2}|{\bf v}|^2 f(|{\bf v}|^2/2,t){\rm d}{\bf v}
\ge
2\pi\sqrt{2}E,\\
&&
\int_{{\bR}}|{\bf v}|^4 f(|{\bf v}|^2/2,t){\rm d}{\bf v}\le C \big(1+2/t_0\big)^2\quad
\forall\, t\ge t_0/2.\eeas
Since, by Cauchy-Schwarz inequality,  $E^2/N\le \fr{1}{16\pi\sqrt{2}}
\int_{{\bR}}|{\bf v}|^4 f(|{\bf v}|^2/2,t){\rm d}{\bf v}$, it follows that
$b\ge 4\sqrt{2}a>2a$. Also since for all $t\ge t_0/2$
\beas
 \int_{|{\bf v}|\ge b}\fr{1}{2}|{\bf v}|^2 f{\rm d}{\bf v}
\le\fr{1}{2b^2}\int_{{\bR}}|{\bf v}|^4 f{\rm d}{\bf v}
\le \fr{1}{2b^2} C \big(1+2/t_0\big)^2
=\pi\sqrt{2}E\eeas
it follows that
\bes && \int_{2a\le |{\bf v}|\le b}f(|{\bf v}|^2/2,t){\rm d}{\bf v}
\ge   \fr{2}{b^2}\Big(\int_{ |{\bf v}|\ge 2a}\fr{1}{2}|{\bf v}|^2f{\rm d}{\bf v}
-\int_{ |{\bf v}|\ge b}\fr{1}{2}|{\bf v}|^2f{\rm d}{\bf v}
\Big)\nonumber
\\
&&\ge \fr{2}{b^2}\big(2\pi\sqrt{2}E-\pi\sqrt{2}E\big)
=\fr{8\pi^2 E^2}{C(1+2/t_0)^2}.\lb{entropy3.9}\ees
Let
\beas&& {\cal V}_t
=\Big\{({\bf v},{\bf v}_*, \og)\in {\bRRS}\,\,\Big|\,\, a/2\le |{\bf v}|\le a, 2a\le
|{\bf v}'|\le b, 2a\le |{\bf v}_*'|\le b,\\
&& f(|{\bf v}|^2/2,t)\le 1/3,
f(|{\bf v}'|^2/2,t)\ge 9,  f(|{\bf v}_*'|^2/2,t)\ge 9
\Big\},\quad t\ge t_0/2.
\eeas
Then for all $({\bf v},{\bf v}_*, \og)\in {\cal V}_t$ we have:
$|{\bf v}-{\bf v}'|\ge a,|{\bf v}-{\bf v}_*'|\ge a$
 and so
\beas&& B_K({\bf {\bf v-v}_*},\og)
\ge \fr{|({\bf v}-{\bf v}_*)\cdot\og|}{(4\pi)^2}\min\Big\{
b_0\min\{1, (2a^2)^{\eta}\},\, (4\pi)^2 a^2\Big\}\\
&&\ge
\fr{|({\bf v}-{\bf v}_*)\cdot\og|}{(4\pi)^2}b_0\min\{1, (2a^2)^{1\vee \eta}\},\\
&&
\Pi(f)\ge f(|{\bf v}'|^2/2,t)f(|{\bf v}_*'|^2/2,t)\ge \fr{1}{b^2}
 |{\bf v}'||{\bf v}_*'|f(|{\bf v}'|^2/2,t)f(|{\bf v}_*'|^2/2,t),\eeas
and,  for $g(\cdot,t):=f(\cdot,t)/(1+f(\cdot,t))\,(\le 1 )$,
\beas g=g(|{\bf v}|^2/2,t)\le 1/4,\quad
g'=g(|{\bf v}'|^2/2,t)
\ge 9/10,\quad g_*'=g(|{\bf v}_*'|^2/2,t)
\ge 9/10,\eeas
\beas\Gm\big({g}'{g}_*',\, g{g}_*\big)=
\big(g'g_*'
-gg_*\big)\log\Big(
\fr{g'g_*'}{gg_*}
\Big)\ge \big((9/10)^2
-1/4\big)\log\big((9/10)^2 4\big)>\fr{1}{2}.\eeas
Thus
$$\Pi(f)
\Gm\big({g}'{g}_*',\, g{g}_*\big)
\ge \fr{1}{2b^2}|{\bf v}'|
f(|{\bf v}'|^2/2,t)|{\bf v}_*'|f(|{\bf v}_*'|^2/2,t)$$
and so for all $t\ge t_0/2$
\beas&&
D_K(f(t))\ge \fr{1}{4}
\int_{{\cal V}_t}B_K({\bf
{\bf v-v}_*},\og)\Pi(f)
\Gm\big({g}'{g}_*',\, g{g}_*\big) {\rm d}\og{\rm
d}{\bf v_*}{\rm
d}{\bf v}
\\
&&\ge\fr{b_0\min\{1, (2a^2)^{1\vee \eta}\}}{8(4\pi)^2b^2}
\int_{{\cal V}_t}|({\bf v}-{\bf v}_*)\cdot\og||{\bf v}'|
f(|{\bf v}'|^2/2,t)|{\bf v}_*'|f(|{\bf v}_*'|^2/2,t) {\rm d}\og{\rm
d}{\bf v_*}{\rm
d}{\bf v}
.\eeas
We then compute using the formula (\ref{entropy.1}) that
\beas&&
\int_{{\cal V}_t}|({\bf v}-{\bf v}_*)\cdot\og||{\bf v}'|f(|{\bf v}'|^2/2,t)|{\bf v}_*'|f(|{\bf v}_*'|^2/2,t)
{\rm d}\og{\rm
d}{\bf v_*}{\rm
d}{\bf v}
\\
&&=\int_{{\bRRS}}|({\bf v}-{\bf v}_*)\cdot\og|1_{\{a/2\le |{\bf v}|\le a\}}1_{\{f(|{\bf v}|^2/2,t)\le 1/3\}}
1_{\{2a\le |{\bf v}'|\le b\}}
1_{\{2a\le |{\bf v}_*'|\le b\}}\\
&&\times
1_{\{f(|{\bf v}'|^2/2,t)\ge 9\}}1_{\{f(|{\bf v}_*'|^2/2,t)\ge 9\}}
|{\bf v}'|f(|{\bf v}'|^2/2,t)|{\bf v}_*'|f(|{\bf v}_*'|^2/2,t)  {\rm d}\og{\rm
d}{\bf v_*}{\rm
d}{\bf v}
\\
&&=
mes\big(\{{\bf v}\in {\bR}\,|\,a/2\le |{\bf v}|\le a, f(|{\bf v}|^2/2,t)\le 1/3\}\big)
\Big(\int_{
2a\le |{\bf v}|\le b,\,f(|{\bf v}|^2/2,t)\ge 9}
f(|{\bf v}|^2/2,t){\rm d}{\bf v}\Big)^2.\quad
\eeas
Also using (\ref{entropy.2}) and the non-decrease of the entropy  we have for all $t\in [t_0/2, t_0]$
\beas&& mes(\{{\bf v}\in{\bR}\,\,|\,\, f(|{\bf v}|^2/2, t)> 1/3\})\le 2S(f(t))
\le 2S(f(t_0))
,\\
&&
\int_{f(|{\bf v}|^2/2, t)< 9}f(|{\bf v}|^2/2,t){\rm d}{\bf v}
\le 5S(f(t))\le 5S(f(t_0))\eeas
and so, using (\ref{entropy3.8}),(\ref{entropy3.9}),
\beas&&mes\big(\{{\bf v}\in {\bR}\,|\, a/2
\le |{\bf v}|\le a, f(|{\bf v}|^2/2,t)\le 1/3\}\big)
\\
&& \ge mes\big(\{{\bf v}\in {\bR}\,|\,a/2
\le |{\bf v}|\le a\}\big)
-mes\big(\{{\bf v}\in {\bR}\,|\, f(|{\bf v}|^2/2,t)> 1/3\}\big)
\\
&&\ge \fr{7\pi}{6} a^3-2S(f(t_0))\ge  \fr{7\pi}{12} a^3,\\ \\
&&
\int_{
2a\le |{\bf v}|\le b,\,f(|{\bf v}|^2/2,t)\ge 9}
f(|{\bf v}|^2/2,t){\rm d}{\bf v}
\ge \int_{
2a\le |{\bf v}|\le b}
f{\rm d}{\bf v}-
\int_{f(|{\bf v}|^2/2,t)< 9}
f{\rm d}{\bf v}
\\
&&\ge \fr{8\pi^2 E^2}{C_4(1+2/t_0)^2}-5S(f(t_0))
\ge\fr{4\pi^2 E^2}{C_4(1+2/t_0)^2}.\eeas
Thus for all $t\in [t_0/2, t_0]$ we have
\beas&&
\int_{{\cal V}_t}|({\bf v}-{\bf v}_*)\cdot\og||{\bf v}'|f(|{\bf v}|^2/2,t)|{\bf v}_*'|
f(|{\bf v}_*|^2/2,t)
{\rm d}\og{\rm
d}{\bf v_*}{\rm
d}{\bf v}
\ge
 \fr{7\pi a^3}{12}\Big(\fr{4\pi^2 E^2}{C_4(1+2/t_0)^2}
\Big)^2\eeas
and so
\beas&&
D_K(f(t))
\ge \fr{b_0\min\{1, (2a^2)^{1\vee \eta}\}}{8(4\pi)^2b^2}\fr{7\pi a^3}{12}\Big(\fr{4\pi^2 E^2}{C(1+2/t_0)^2}
\Big)^2\\
&&=b_0\min\{1, (2a^2)^{1\vee \eta}\}
\fr{7\pi^4 \sqrt{2}a^3E^5}{48 C^3(1+2/t_0)^6},\\ \\
&&
S(f(t_0))=
S(f(t_0/2))
+\int_{t_0/2}^{t_0}
D_K(f(t)){\rm d}t\ge \int_{t_0/2}^{t_0}
D_K(f(t)){\rm d}t\\
&&\ge b_0\min\{1, (2a^2)^{1\vee \eta}\}
\fr{7\pi^4 \sqrt{2}a^3E^5}{48 C^3(1+2/t_0)^6}\fr{t_0}{2}
\ge S_{*}(t_0).\eeas
This proves (\ref{entropy3.5}).

(II): Let $\bar{F}_0\in B_{2}^{+}({\bR})$ be the isotropic measure defined by $F_0$ through (\ref{m2}) or
(\ref{m4}). For each $K\in {\mN}$, let $0\le f^K_0=f^K_0(|{\bf v}|^2/2)\in L^1_2({\bR})$
be obtained in Lemma \ref{lemma.entropy} for $\bar{F}_0$, so that $f^K_0$ satisfy
(\ref{entropy.3}), (\ref{entropy3.5}) (because $S(\bar{F_0})=S(F_0)$) and (\ref
{weak-isotropic}) with $\wt{f}_K(x)=f^K_0(x), \wt{F}=F$.
Let $f^K=f^K(|{\bf v}|^2/2,t)$ be the conservative isotropic approximate  solution
of
Eq.(\ref{Equation1}) on ${\bR}\times [0,\infty)$  corresponding to the approximate kernel
 $B_K$ obtained in part (I) of the proposition satisfying $f^K|_{t=0}=f^K_0$.
 Let $F^K_t, F^K_0\in {\cal B}_1^{+}({\mR}_{\ge 0})$
 be defined by ${\rm d}F^K_t(x)=f^K(x,t)\sqrt{x}{\rm d}x,
 {\rm d}F^K_0(x)=f^K_0(x)\sqrt{x}{\rm d}x$,
 then the measure $F^K_t$ with the initial datum $F_0^K$
 is a conservative measure-valued isotropic solution of
Eq.(1.1) on $[0,\infty)$ in the sense of Definition \ref{definition1.1} corresponding to the collision kernel $B_K({\bf v-v}_*,\og)$ given by
(\ref{3.kernel1}),(\ref{3.kernel2}). Since Definition \ref{definition1.1} is
equivalent to Definition \ref{definition5.1} (see Appendix), it follows from Theorem 1 (Weak Stability)
in \cite{Lu2004} that
 there exist a conservative  measure-valued isotropic solution $F_t$ of Eq.(\ref{Equation1}) on $[0,\infty)$ (corresponding to the kernel $B$) with the initial datum $F_0$,
and a subsequence $\{f^{K_n}\}_{n\in{\mN}}$,  such that
$F_t$ is the weak limit of $f^{K_n}(\cdot,t)$, i.e.
(\ref{3.weak1}) holds true. This implies that the weak convergence
(\ref{3.weak2}) also holds true. Thus by definition of entropy of measures and (\ref{entropy3.5}) we obtain
$$S(F_t)=S(\bar{F}_t)\ge \limsup_{n\to\infty}S(f^{K_n}(t))\ge S_{*}(t_0)\qquad \forall\, t\ge t_0$$
for all $t_0>0$. Finally since $M_p(F_t)=
(4\pi\sqrt{2})^{-1}\int_{{\bR}}(|{\bf v}|^2/2)^p{\rm d}\bar{F}_t({\bf v})$, the moment production (\ref{3.moment}) follows easily from the weak convergence (\ref{3.weak2}) and (\ref{mprod}).
\end{proof}

\begin{center}{}\end{center}
\begin{center}\section{Rate of Entropy Convergence}\end{center}

In this section we prove the first part of Theorem \ref{theorem1.2}: algebraic rate of entropy convergence for measure-valued solutions. As usual,
we start to work  with approximate solutions then take weak limit to passing the result to a true solution.

For convenience of proof and in order to connect some known results as mentioned in Section 2 for ${\cal D}_2(f)$, we will also use the $\sg$-representation of $({\bf v}', {\bf v}_*'):$
\be {\bf v}'=\fr{{\bf v+v}_*}{2}+\fr{|{\bf v-v}_*|}{2}\sg,\quad
{\bf v}'_*=\fr{{\bf v+v}_*}{2}-\fr{|{\bf v-v}_*|}{2}\sg,\quad \sg\in{\bS},\quad {\bf v},{\bf v}_*\in {\bR}.\lb{colli2}\ee
It is not difficult to deduce the following relation between the
$\og$-representation (\ref{colli}) and the $\sg$-representation (\ref{colli2}) (see e.g. Section 4 of Chapter 1 in \cite{Villani2}):
\be
\int_{{\bS}}\Psi({\bf v}', {\bf v}_*')
\big|_{\og-{\rm rep.}}{\rm d}\og=
\int_{{\bS}}\fr{|{\bf v-v}_*|}{2|{\bf v}-{\bf v}'|}\Psi({\bf v}', {\bf v}_*')\big|_{\sg-{\rm rep.}}{\rm d}\sg. \lb{3.9}
\ee In particular we have
\be
\int_{{\bS}}B_K({\bf {\bf v-v}_*},\omega)\Psi({\bf v}', {\bf v}_*')\big|_{\og-{\rm rep.}}{\rm d}\og=
\int_{{\bS}}\bar{B}_K({\bf {\bf v-v}_*},\sg)\Psi({\bf v}', {\bf v}_*')\big|_{\sg-{\rm rep.}}{\rm d}\sg
\lb{3.new3}\ee
where $\bar{B}_K$ is defined through $B_K$ as follows (recall that $B_K$ is
given by (\ref{3.kernel1}),(\ref{3.kernel2})):
\bes&&
B_K({\bf {\bf v-v}_*},\omega)=
\fr{|{\bf v}-{\bf v}'|}{(4\pi)^2}\Phi_K(|{\bf v}-{\bf v}'|, |{\bf {\bf v}-{\bf v}_*'}|)\big|_{\og-{\rm rep.}},
\lb{3.new4}\\
&&
\bar{B}_K({\bf {\bf v-v}_*},\sg)=\fr{|{\bf v}-{\bf v}_*|}{2(4\pi)^2}\Phi_K(|{\bf v}-{\bf v}'|, |{\bf v}-{\bf v}_*'|)\big|_{\sg-{\rm rep.}}.\dnumber\lb{3.new5}\ees

The $\sg$-representation (\ref{colli2}) in many cases is convenient than the $\og$-representation (\ref{colli}), but since it is non-linear and non-smooth in $({\bf v},{\bf v}_*)$, one needs to go back to $\og$-representation (\ref{colli}) when proving some integral identities.
For instance this can be seen in the proof of the following formula of change of variables for $\sg$-representation:
\bes&&\int_{{\bRS}}\Psi(|{\bf v-v}_*| , {\bf n}\cdot\sg , {\bf v}', {\bf v}_*' )
{\rm d}\sg{\rm d}{\bf v}_*
\qquad \nonumber\\
&&=
\int_{0}^{\pi}\fr{\sin(\theta)}{\cos^3(\theta/2)}{\rm d}\theta
\int_{{\bR}}{\rm d}{\bf v}_*\int_{{\mS}^1({\bf
n})}\Psi\Big(\fr{|{\bf v-v}_*|}{\cos(\theta/2)},\,\cos(\theta),\,
{\bf v}-|{\bf v-v}_*|\fr{\sin(\theta/2)}{\cos(\theta/2)}{\tilde \sg}, {\bf v}_*\Big){\rm d}{\tilde \sg}
\qquad  \quad \lb{PSI}\ees
where ${\bf n}=\fr{{\bf v-v}_*}{|{\bf v-v}_*|}$, $\Psi$ is any nonnegative Lebesgue measurable functions
on ${\mR}_{\ge 0}\times [-1,1]\times {\bRR}$,
${\mathbb S}^1({\bf n})=\{\tilde{\sg}\in {\bS}\,\,|\,\,\tilde{\sg}\,\bot\, {\bf n}\,\}$
and ${\rm d}{\tilde \sg}$ is the circle measure element on ${\mathbb S}^1({\bf n})$.

The proof of (\ref{PSI}) is just several elementary changes of variables:

\noindent(1)  take reflection
$\sg\to -\sg$, then use
$|{\bf v-v}_*|=|{\bf v}'-{\bf v}_*'|$, ${\bf v}_*={\bf v}'+{\bf v}_*'-{\bf v}$ to
 write the integrand as a function
of $({\bf v}', {\bf v}_*') $
(with ${\bf v}$ fixed): $\Psi=\Psi\big(|{\bf v}'-{\bf v}_*'| , \fr{({\bf v}'-{\bf v}_*')\cdot({\bf v}'+{\bf v}_*'-2{\bf v})}{|{\bf v}'-{\bf v}_*'|^2} , {\bf v}'_*, {\bf v}'\big)$;

\noindent (2) use the formula
(\ref{3.9}),  then change variables ${\bf v}_*={\bf v}-{\bf z},
{\bf z}=r\sg$;

\noindent (3) change variables $r=\fr{\rho}{\sg\cdot\og}$ for $\sg\cdot\og>0$,
$\rho\og={\bf z}={\bf v}-{\bf v}_*$, then denote again ${\bf n}=\fr{{\bf v-v}_*}{|{\bf v-v}_*|}$;

\noindent (4)  change variables $\sg={\bf n}\cos(\theta)+\sin(\theta)\wt{\sg},\,\, \wt{\sg}
\in {\mS}^1({\bf n})$,
etc.

\noindent Accordingly we compute
\beas&&\int_{{\bRS}}\Psi\big(|{\bf v-v}_*| , {\bf n}\cdot\sg , {\bf v}', {\bf v}_*'\big)
{\rm d}\sg{\rm d}{\bf v}_*
=\int_{{\bRS}}2|{\bf n}\cdot\omega|\Psi\big(|{\bf v}-{\bf v}_*| ,
2({\bf n}\cdot\og)^2-1, {\bf v}'_*, {\bf v}'\big)
{\rm d}\og{\rm d}{\bf v}_*\\
&&=4
\int_{{\bSS}}{\bf 1}_{\{\sg\cdot\og>0\}} |\sg\cdot\omega|{\rm d}\og{\rm d}\sg \int_{0}^{\infty}r^2\Psi\big(r,
2(\sg\cdot\og)^2-1, {\bf v}-r\sg +r(\sg\cdot\og)\og,
{\bf v}-r (\sg\cdot \og)\og\big){\rm d}r\\
&&=4
\int_{{\bR}}{\rm d}{\bf v}_*\int_{{\bS}}{\bf 1}_{\{\sg\cdot {\bf n}>0\}}
\fr{1}{(\sg\cdot {\bf n})^2}
\Psi\Big(\fr{|{\bf v-v}_*|}{\sg\cdot{\bf n}},
2(\sg\cdot{\bf n})^2-1, {\bf v}-\fr{|{\bf v-v}_*|}{\sg\cdot{\bf n}}\sg +{\bf v}-{\bf v}_*,
{\bf v}_*\Big){\rm d}\sg
\\
&&=4
\int_{{\bR}} {\rm d}{\bf v}_*\int_{0}^{\pi/2}\fr{\sin(\theta)}{\cos^2(\theta)}
{\rm d}\theta\int_{{\mS}^1({\bf n})}
\Psi\Big(\fr{|{\bf v-v}_*|}{\cos(\theta)},
2\cos^2(\theta)-1, {\bf v}-|{\bf v-v}_*|\fr{\sin(\theta)}{\cos(\theta)}\wt{\sg},
{\bf v}_*\Big){\rm d}\wt{\sg}
\eeas
which is equal to the right hand side of (\ref{PSI}).

As an application of the formula (\ref{PSI}) we prove the following lemma which is also used in the proof of
Proposition \ref{prop4.2}.

\begin{lemma}\label{lemma3.5} Let $\alpha>0$ be a constant and $0\le f\in L^1({\bR})$.
Then
\be \int_{{\bRS}}|{\bf v-v}_*||{\bf v}-{\bf v}'|f({\bf v}')e^{-\alpha|{\bf v}_*'|}{\rm d}\sg {\rm d}{\bf v}_*
\le 8\pi\fr{1+2\alpha}{\alpha^2}\la {\bf v}\ra\|f\|_{L^1}\lb{3.21}\ee
\end{lemma}

\begin{proof} Using the formula (\ref{PSI}) and
$|{\bf v}-{\bf u}|\ge ||{\bf v}|-|{\bf u}||$ we have
\beas&&
\int_{{\bRS}}|{\bf v-v}_*||{\bf v}-{\bf v}'|f({\bf v}')e^{-\alpha|{\bf v}_*'|}{\rm d}\sg {\rm d}{\bf v}_*
\\
&&=
\int_{0}^{\pi}\fr{\sin(\theta)}{\cos^3(\theta/2)}{\rm d}\theta
\int_{{\bR}}{\rm d}{\bf v}_*\int_{{\mS}^1({\bf
n})}\fr{|{\bf v-v}_*|^2}{\cos(\theta/2)}f({\bf v}_*)
\exp\Big(-\alpha\Big|{\bf v}-|{\bf v-v}_*|\fr{\sin(\theta/2)}{\cos(\theta/2)}{\tilde \sg}\Big|\Big){\rm d}{\tilde \sg}\\
&&\le 8\pi\int_{{\bR}}|{\bf v-v}_*|^2f({\bf v}_*){\rm d}{\bf v}_*
\int_{0}^{\pi/2}\fr{\sin(\theta)}{\cos^3(\theta)}
\exp\Big(-\alpha\Big||{\bf v}|-|{\bf v-v}_*|\fr{\sin(\theta)}{\cos(\theta)}
\Big|\Big){\rm d}\theta,\eeas
and for the inner integral we compute (changing variable $\theta=\arctan(t)$, etc.)
$$
\int_{0}^{\pi/2}\{\cdots\}{\rm d}\theta=\fr{1}{|{\bf v-v}_*|^2}\int_{-|{\bf v}|}^{\infty}(t+|{\bf v}|)
e^{-\alpha|t|}{\rm d}t
\le \fr{1}{|{\bf v-v}_*|^2}\Big(\fr{1}{\alpha^2}
+2|{\bf v}|\fr{1}{\alpha}\Big).$$
Inserting this into the above inequality gives (\ref{3.21}).
\end{proof}
\vskip3mm

The following proposition is established for approximate solutions of Eq.(\ref{Equation1}) for
arbitrary initial data, in particular
it includes anisotropic initial data.

\begin{proposition}\label{prop4.2} Let $B({\bf {\bf v-v}_*},\omega)$ satisfy Assumption \ref{assp} and let
$B_K({\bf {\bf v-v}_*},\omega)$ $(K\in {\mN})$ be the approximations of $B({\bf {\bf v-v}_*},\omega)$
defined by (\ref{3.kernel1}),(\ref{3.kernel2}) so that  
\be  \min\big\{B({\bf {\bf v-v}_*},\omega),\, K|{\bf v}-{\bf v}'|^2
 |{\bf v}-{\bf v}_*'|\big\}
=B_K({\bf {\bf v-v}_*},\omega)\le B({\bf {\bf v-v}_*},\omega)\lb{3.new6}\ee
for all $K\in{\mN}, ({\bf v},{\bf v}_*,\og)\times {\bRRS}$.
Given any $N>0, E>0$, let $\Og({\bf v})=f_{\rm be}(|{\bf v}|^2/2)$ be
given in Proposition \ref{prop2.1} with $f_{\rm be}(x)$
the regular part of the equilibrium $F_{\rm be}$ that has the mass and energy $N,E$.
Let $\{f^K_0\}_{K\in{\mN}} \subset L^1_2({\bR})$ be a nonnegative sequence satisfying
\be \int_{{\bR}}(1,{\bf v}, |{\bf v}|^2/2) f_0^K({\bf v}){\rm d}{\bf v}= 4\sqrt{2}(N,0,E)\quad
\forall\, K\in{\mN},\lb{3.new7}\ee
let $\{f^K\}_{K\in {\mN}}\subset L^{\infty}([0,\infty); L^1_2({\bR}))\cap
L^{\infty}([1,\infty); L^1_4({\bR}))$ be a sequence of approximate solutions
(in terms of Definition \ref{definition3.1})
of Eq.(\ref{Equation1}) on ${\bR}\times [0,\infty)$ corresponding to the approximate kernels $\{B_K\}_{K\in{\mN}}$
satisfying $f^K|_{t=0}=f^K_0$ for all $K\in{\mN}$, and suppose for some constants $S_0>0, 0<C_0<\infty$,
\be \inf_{K\in{\mN}, t\ge 1}S(f^K(t))\ge S_0,\quad \sup_{K\in{\mN}, t\ge 1}\|f^K(t)\|_{L^1_{4}}\le C_0.\lb{3.new8}\ee
Given any $\ld$ satisfying $\fr{1}{10\max\{2, \fr{4+2\eta}{3}\}} <\ld<\fr{1}{10\max\{2, \fr{4+2\eta}{3}\}-1}$.
Let $p=p_{\ld}, \dt=\dt_{\ld}$ satisfy
\be \max\big\{2, \fr{4+2\eta}{3}\big\}<p<\fr{1}{10}\big(1+\fr{1}{\ld}\big),\quad \dt=1-\ld (10p -1).
\lb{abd}\ee
Let ${\cal E}({\bf v})=\alpha_0 e^{-\beta_0|{\bf v}|}$ where
$\alpha_0=\alpha_0(N,E)>0, \beta_0=\beta_0(N,E)>0$ are given through the
moment equations
$$
 \int_{{\bR}}(1, |{\bf v}|^2/2){\cal E}({\bf v}){\rm d}{\bf v}=
4\pi\sqrt{2}(N, E).$$
Let
\be F^K({\bf v},t)=(1-e^{-t^{\dt}})f^K({\bf v},t)+
e^{-t^{\dt}}{\cal E}({\bf v}),\quad ({\bf v},t)\in{\bR}\times [0,\infty).\lb{3.22}\ee
Then there are constants $0<C_i<\infty\,(i=1,2,3,4)$ that depend only on $N,E, S_0, C_0,
b_0,\eta$ and $\ld$ such that

{\rm(I)}  For any $K\in{\mN}$, $0<T<\infty$,
the entropy $t\mapsto S(F^K(t))$ is absolutely continuous on $[0,T]$ and
\be
\fr{{\rm d}}{{\rm d}t}\big(S(\Og)-S(F^K(t))
\big)\le -D_K(F^K(t))+C_{1} t^{\dt}e^{-t^{\dt}},\qquad {\rm a.e.}\,\,\,t\in[1,\infty).\lb{3.26}\ee

{\rm(II)}
Let $G^K=\fr{F^K}{1+F^K}$. Then for any $K\in {\mN}$
\be {\cal D}_2(G^K(t))
\le C_{2}\Big(\big(t^{\dt(p-1)} K^{-1}D_K(F^K(t)\big)^{\fr{1}{p}}
+\big(t^{\dt}D_K(F^K(t))\big)^{\fr{1}{p}}\Big)\quad \forall\, t\in[1,\infty).
\lb{3.38}\ee

{\rm(III)}  There is
a constant $C_3>0$ such that for the function
$\Psi(y)= C_{3}\big( y^{\fr{1}{10p}}
+y^{\fr{1}{4}}\big), y\in [0,\infty),$
it holds
\be \fr{{\rm d}}{{\rm d}t}\big(S(\Og)-S(F^K(t))
\big)\le  - t^{-\dt}\Psi^{-1}\big(S(\Og)-S(F^K(t))\big)
+C_{4}t^{\dt}e^{-t^{\dt}}\qquad {\rm a.e.}\,\,\,\, t\in [1, T]\lb{3.46}\ee
for all $1<T<\infty$ and all $K\ge  T^{\dt(p-2)}$,
where $\Psi^{-1}(u)$ is the inverse function of $\Psi(y).$
\end{proposition}

\begin{proof}  First we see from the above assumptions that $F^K$ satisfy also (\ref{3.new7}), i.e.
\be \int_{{\bR}}(1,{\bf v}, |{\bf v}|^2/2)F^K({\bf v},t){\rm d}{\bf v}= 4\sqrt{2}(N,0,E)\quad
\forall\, t\ge 0,\,\,\forall\, K\in{\mN}\lb{3.conserv}\ee
and so by Proposition \ref{prop2.1} we have
$ S(\Og)-S(F^K(t))\ge 0$  for all $t\ge 0, K\in {\mN}$.
Next  by (\ref{3.22}), concaveness (\ref{S2}), increase of $t\mapsto S(f^K(t))$,
and (\ref{3.new8})  we have
\beas S(F^k(t))\ge (1-e^{-t^{\dt}})S(f^K(t))+
e^{-t^{\dt}}S({\cal E})\ge (1-e^{-t^{\dt}})S_0+
e^{-t^{\dt}}S({\cal E})\quad \forall\, t\ge 1\eeas
which implies that
\be
\inf_{K\in {\mN},t\ge 1}S(F^k(t))\ge \min\{S_0, S({\cal E})\}>0.\lb{3.lbd}\ee
Also we have
\be \sup_{K\in{\mN}, t\ge 1}\|F^K(t)\|_{L^1_4}\le \max\{C_0, \|{\cal E}\|_{L^1_4}\}<\infty.
\lb{3.ubd}\ee
We note also that, since
 $f^K$ satisfy the integrability assumption (\ref{3.new1}),
there are no problems of integrability in
the following derivation.
To simplify notations we denote for any given $K\in{\mN}$
$$ f({\bf v},t)=f^K({\bf v},t),\quad
F({\bf v},t)=F^K({\bf v},t)
,\quad G({\bf v},t)=\fr{F({\bf v},t)}{1+F({\bf v},t)}.$$
In the following all constants $0<C_i<\infty\,(i=1,2,...,27)$ depend only on $N,E, S_0, C_0,
b_0,\eta, p_1, p_2,$ and $\dt$.

{\rm(I)}: Since
\be 0\le \log\Big(\fr{1}{G({\bf v}, t)}\Big)\le\log\Big( 1+\fr{1}{\alpha_0}e^{t^{\dt}+\beta_0|{\bf v}|}\Big)
\le C_{5} (1+t^{\dt})\la {\bf v}\ra\lb{3.new10}\ee
it follows that for almost every ${\bf v}\in {\bR}$,
$t\mapsto s(F({\bf v},t))$ is absolutely continuous on $[0,T]$
(for all $0<T<\infty$)
and so for all $0\le t_1<t_2<\infty$
\beas&& s(F({\bf v},t_2))- s(F({\bf v},t_1))\\
&& =
\int_{t_1}^{t_2}\log\Big(\fr{1}{G({\bf v}, \tau)}\Big)
\Big(\dt \tau^{\dt-1} e^{-\tau^{\dt}}(f({\bf v},\tau)-{\cal E}({\bf v}))
+(1-e^{-\tau^{\dt}})Q_K(f)({\bf v},\tau)
\Big){\rm d}\tau,\\ \\
&&
S(F(t_2))-S(F(t_1))\\
&&=
\int_{t_1}^{t_2}{\rm d}\tau
\int_{{\bR}}\log\Big(\fr{1}{G({\bf v}, \tau)}\Big)
\Big(\dt \tau^{\dt-1} e^{-\tau^{\dt}}(f({\bf v},\tau)-{\cal E}({\bf v}))
+(1-e^{-\tau^{\dt}})Q_K(f)({\bf v},\tau)
\Big){\rm d}{\bf v}.
\eeas
This implies that for any $0<T<\infty$ the entropy $t\mapsto S(F^K(t))$ is absolutely
continuous on $[0,T]$ and for almost every $t\in [0,\infty)$
\be
\fr{{\rm d}}{{\rm d}t}\big(S(\Og)-S(F(t))
\big)=\dt t^{\dt-1} e^{-t^{\dt}}
\int_{{\bR}}({\cal E}-f)\log\big(\fr{1}{G}\big)
{\rm d}{\bf v}-(1-e^{-t^{\dt}})
\int_{{\bR}}Q_K(f)\log\big(\fr{1}{G}\big)
{\rm d}{\bf v}. \lb{3.28}\ee
With the $\sg$-representation we have
\beas&&
\int_{{\bR}}Q_K(f)({\bf v},t)\log\big(\fr{1}{G({\bf v},t)}\big)
{\rm d}{\bf v}\\
&&=\fr{1}{4}\int_{{\bRRS}}\bar{B}_K({\bf v-v}_*,\sg)
\big(f'f_*'(1+f+f_*)-ff_*(1+f'+f_*')\big)
\log\Big(\fr{G'G_*'}{GG_*}\Big)
{\rm d}{\bf v}{\rm d}{\bf v}_*{\rm d}\sg.\eeas
Now we will prove that with an error $O(t^{\dt}e^{-t^{\dt}})$ the right hand side of (\ref{3.28}) is approximately less than $-D_K(F(t))$.
Recalling (\ref{3.2}),(\ref{3.new3}),(\ref{3.new4}),(\ref{3.new5})
we write (with the $\sg$-representation)
$$ D_K(F)=\fr{1}{4}\int_{{\bRRS}}\bar{B}_K\cdot
\big(F'F_*'(1+F+F_*)-FF_*(1+F'+F_*')\big)\log\Big(\fr{G'G_*'}{GG_*}\Big)
{\rm d}{\bf v}{\rm d}{\bf v}_*{\rm d}\sg.$$
Before going on we note that from (\ref{3.new5}) and (\ref{3.kernel2})  we have
\be\fr{|{\bf {\bf v-v}_*}|}{2(4\pi)^2}\min\Big\{
\Phi(|{\bf v-v'}|, |{\bf {\bf v-v}_*'}|),\, (4\pi)^2K|{\bf v-v'}||{\bf {\bf v-v}_*'}|\Big\}
=\bar{B}_K({\bf {\bf v-v}_*},\sg)\le \fr{|{\bf v-v}_*|}{2(4\pi)^2}
\label{3.new11}\ee
For the convex combination  $F=(1-e^{-t^{\dt}})f+ e^{-t^{\dt}}{\cal E}$ we compute
\beas&& F'F_*'(1+F+F_*)- FF_*(1+F'+F_*')
\\
&&=(1-e^{-t^{\dt}})^3\big(f'f_*'(1+f+f_*)-ff_*(1+f'+f_*')\big)
\\
&&+(1-e^{-t^{\dt}})^2e^{-t^{\dt}}\big(f'f_*'(1+{\cal E}+{\cal E}_*)
-ff_*(1+{\cal E}'+{\cal E}_*')\big)
\\
&&+(1-e^{-t^{\dt}})^2e^{-t^{\dt}}\big((f'{\cal E}_*'+f_*'{\cal E}')
-(f{\cal E}_*+f_*{\cal E})\big)
\\
&&+(1-e^{-t^{\dt}})^2e^{-t^{\dt}}\big((f'{\cal E}_*'+f_*'{\cal E}')(f+f_*)
-(f{\cal E}_*+f_*{\cal E})(f'+f_*')\big)
\\
&&+(1-e^{-t^{\dt}})e^{-2t^{\dt}}\big((f'{\cal E}_*'+f_*'{\cal E}')(1+{\cal E}+{\cal E}_*)
-(f{\cal E}_*+f_*{\cal E})(1+{\cal E}'+{\cal E}_*')\big)
\\
&&+(1-e^{-t^{\dt}})e^{-2t^{\dt}}\big({\cal E}'{\cal E}_*'-{\cal E}{\cal E}_*
\big)\\
&&+(1-e^{-t^{\dt}})e^{-2t^{\dt}}\big({\cal E}'{\cal E}_*'(f+f_*)-{\cal E}{\cal E}_*(f'+f_*')
\big)\\
&&+e^{-3t^{\dt}}\big({\cal E}'{\cal E}_*'(1+{\cal E}+{\cal E}_*)-{\cal E}{\cal E}_*(1+{\cal E}'+{\cal E}_*')\big)
\\
&&:=(1-e^{-t^{\dt}})^3\big(f'f_*'(1+f+f_*)-ff_*(1+f'+f_*')\big)
+\Psi_1+\Psi_2+\cdots+\Psi_7\eeas
\beas\Longrightarrow\quad
D_K(F) &=&(1-e^{-t^{\dt}})^3
\int_{{\bR}}Q_K(f)({\bf v},t)\log\big(\fr{1}{G({\bf v},t)}\big)
{\rm d}{\bf v}\\
&+&\fr{1}{4}\sum_{j=1}^7\int_{{\bRRS}}\bar{B}_K\cdot
\Psi_j\log\Big(\fr{G'G_*'}{GG_*}\Big)
{\rm d}{\bf v}{\rm d}{\bf v}_*{\rm d}\sg.\eeas
Comparing this with (\ref{3.28}) we obtain
\bes&&
\fr{{\rm d}}{{\rm d}t}\big(S(\Og)-S(F(t))
\big)=
\dt t^{\dt-1} e^{-t^{\dt}}
\int_{{\bR}}({\cal E}-f)\log\big(\fr{1}{G}\big)
{\rm d}{\bf v}-\fr{1}{(1-e^{-t^{\dt}})^{2}}D_K(F)\nonumber \\
&&
+\fr{1}{4(1-e^{-t^{\dt}})^{2}}\sum_{j=1}^7\int_{{\bRRS}}\bar{B}_K\cdot
\Psi_j\log\Big(\fr{G'G_*'}{GG_*}\Big)
{\rm d}{\bf v}{\rm d}{\bf v}_*{\rm d}\sg,\quad {\rm a.e.}\,\,\, t\ge 1.\lb{3.29}\ees
By definition of $G$ we have
\be  t^{\dt-1} e^{-t^{\dt}}
\int_{{\bR}}({\cal E}-f)\log(\fr{1}{G})
{\rm d}{\bf v}\le C_{6}t^{2\dt-1} e^{-t^{\dt}}
\int_{{\bR}}{\cal E}({\bf v})(1+|{\bf v}|)
{\rm d}{\bf v}\le C_{7}t^{\dt} e^{-t^{\dt}},\quad t\ge 1.\lb{3.30}\ee
From (\ref{3.new10}) we have
\be \Big|\log\Big(\fr{G({\bf v}',t)G({\bf v}_*',t)}{G({\bf v},t)G({\bf v}_*,t)}\Big)\Big|
\le  C_{8}t^{\dt}\sqrt{1+|{\bf v}|^2+|{\bf v}_*|^2},\quad t\ge 1.\lb{3.31}\ee
Also from $|{\bf v-v}_*|=|{\bf v}'-{\bf v}_*'|, |{\bf v}|^2+|{\bf v}_*|^2
=|{\bf v}'|^2+|{\bf v}_*'|^2$ we have
\be
|{\bf v-v}_*|\sqrt{1+|{\bf v}|^2+|{\bf v}_*|^2}\le
\min\big\{\la {\bf v}\ra^2\la {\bf v}_*\ra^2,
\la {\bf v}'\ra^2\la {\bf v}_*'\ra^2\big\}.\lb{3.32}\ee
From (\ref{3.new11}), (\ref{3.31}),(\ref{3.32}) we obtain
\bes&&\fr{1}{4(1-e^{-t^{\dt}})^{2}}
\sum_{1\le j\le 7, j\neq 3}\int_{{\bRRS}}\bar{B}_K|
\Psi_j|\Big|\log\Big(\fr{G'G_*'}{GG_*}\Big)
\Big|{\rm d}{\bf v}{\rm d}{\bf v}_*{\rm d}\sg \nonumber \\
&&\le C_{9}e^{-t^{\dt}}t^{\dt}\Big((\|f(t)\|_{L^1_2}+\|{\cal E}\|_{L^1_2})^2+
\int_{{\bRRS}}
\wt{{\cal E}}'\wt{{\cal E}}_*'f
{\rm d}{\bf v}{\rm d}{\bf v}_*{\rm d}\sg\Big),\quad t\ge 1
\lb{3.33}\ees
where
$\wt{{\cal E}}({\bf v})=\la {\bf v}\ra^2{\cal E}({\bf v})$.
By definition of ${\cal E}$ we have
$\wt{{\cal E}}({\bf v})\le \alpha_0(1+4/\beta_0)^2 e^{-\fr{\beta_0}{2}|{\bf v}|}$.
Since $|{\bf v}'|+|{\bf v}_*'|
 \ge \sqrt{|{\bf v}'|^2+|{\bf v}_*'|^2}=\sqrt{|{\bf v}|^2+|{\bf v}_*|^2}
\ge |{\bf v}_*|$, this gives
$\wt{{\cal E}}'\wt{{\cal E}}_*'\le C_{10} e^{-\fr{\beta_0}{2}|{\bf v}_*|}$. So
\be\int_{{\bRRS}}
\wt{{\cal E}}'\wt{{\cal E}}_*'f
{\rm d}{\bf v}{\rm d}{\bf v}_*{\rm d}\sg
\le C_{10}4\pi\int_{{\bRR}}e^{-\fr{\beta_0}{2}|{\bf v}_*|}
f({\bf v},t){\rm d}{\bf v}{\rm d}{\bf v}_*
=C_{11}\|f(t)\|_{L^1}.\lb{3.34}\ee
Now we estimate the integral involving $\Psi_3$. We write
\beas&&(f'{\cal E}_*'+f_*'{\cal E}')(f+f_*)-(f{\cal E}_*+f_*{\cal E})(f'+f_*')
\\
&&=ff'({\cal E}_*'-{\cal E}_*)+f_*f'({\cal E}_*'-{\cal E})
+ff_*'({\cal E}'-{\cal E}_*)+f_*f_*'({\cal E}'-{\cal E}).
\eeas
By change of variables $({\bf v}, {\bf v}_*)\to
({\bf v}_*, {\bf v})$ and
$\sg\to -\sg$ and notice that $\bar{B}_K({\bf v-v}_*,\sg)=\bar{B}_K({\bf v}_*-{\bf v},\sg), \bar{B}_K({\bf v-v}_*,-\sg)=\bar{B}_K({\bf v-v}_*,\sg)$
(see (\ref{3.new5}) and recall that $\Phi(r,\rho)=\Phi(\rho,r)$), we have
\beas&& \fr{1}{4(1-e^{-t^{\dt}})^{2}}\int_{{\bRRS}}\bar{B}_K\cdot
\Psi_3\log\Big(\fr{G'G_*'}{GG_*}\Big)
{\rm d}{\bf v}{\rm d}{\bf v}_*{\rm d}\sg\\
&&=e^{-t^{\dt}}
\int_{{\bRRS}}\bar{B}_K\cdot ff'({\cal E}_*'-{\cal E}_*)
\log\Big(\fr{G'G_*'}{GG_*}\Big)
{\rm d}{\bf v}{\rm d}{\bf v}_*{\rm d}\sg
\\
&&\le 2e^{-t^{\dt}}
\int_{{\bRRS}}\bar{B}_K\cdot ff'({\cal E}_*'-{\cal E}_*)_{+}
\log\Big(\fr{G'G_*'}{GG_*}\Big)\Big|_{G'G_*'>GG_*}
{\rm d}{\bf v}{\rm d}{\bf v}_*{\rm d}\sg
\\
&&\le \fr{C_{12}}{(4\pi)^2}t^{\dt}e^{-t^{\dt}}
\int_{{\bRRS}}|{\bf v-v}_*|ff'({\cal E}_*'-{\cal E}_*)_{+}\sqrt{1+|{\bf v}|^2+|{\bf v}_*|^2}\,
{\rm d}{\bf v}{\rm d}{\bf v}_*{\rm d}\sg.\eeas
Observe that ${\cal E}_*'-{\cal E}_*>0 \Longrightarrow |{\bf v}_*'|<|{\bf v}_*| \Longrightarrow $
$${\cal E}_*'-{\cal E}_*\le \alpha_0 e^{-\beta_0|{\bf v}_*'|}\beta_0(|{\bf v}_*|-|{\bf v}_*'|)
\le \alpha_0\beta_0 e^{-\beta_0|{\bf v}_*'|}|{\bf v}'-{\bf v}|.$$
This together with
$\sqrt{1+|{\bf v}|^2+|{\bf v}_*|^2}\le 1+|{\bf v}'|+|{\bf v}_*'|$ and $ 1+\fr{\beta_0}{2}|{\bf v}_*'|\le  e^{\fr{\beta_0}{2}|{\bf v}_*'|} $ gives
$$({\cal E}_*'-{\cal E}_*)_{+}\sqrt{1+|{\bf v}|^2+|{\bf v}_*|^2}
\le C_{13}|{\bf v}'-{\bf v}|\la {\bf v}'\ra e^{-\fr{\beta_0}{2}|{\bf v}_*'|}.$$
Then using Lemma \ref{lemma3.5} we have
\beas&&
\int_{{\bRRS}}|{\bf v-v}_*|ff'({\cal E}_*'-{\cal E}_*)_{+} \sqrt{1+|{\bf v}|^2+|{\bf v}_*|^2}
{\rm d}{\bf v}{\rm d}{\bf v}_*{\rm d}\sg\\
&&
\le C_{13}
\int_{{\bR}}f({\bf v},t)\Big(
\int_{{\bRS}}|{\bf v-v}_*||{\bf v}'-{\bf v}|\la {\bf v}'\ra f({\bf v}', t)
 e^{-\fr{\beta_0}{2}|{\bf v}_*'|}
{\rm d}{\bf v}_*{\rm d}\sg\Big){\rm d}{\bf v}
\\
&&\le C_{14}(\|f(t)\|_{L^1_1})^2\le C_{14}\|f(t)\|_{L^1}
\|f(t)\|_{L^1_2}\eeas
and thus
\be\fr{1}{4(1-e^{-t^{\dt}})^{2}}\int_{{\bRRS}}\bar{B}_K\cdot
\Psi_3\log\Big(\fr{G'G_*'}{GG_*}\Big)
{\rm d}{\bf v}{\rm d}{\bf v}_*{\rm d}\sg
\le  C_{15}t^{\dt}e^{-t^{\dt}}\|f(t)\|_{L^1}
\|f(t)\|_{L^1_2}.\lb{3.35}\ee
Combining all estimates (\ref{3.29}), (\ref{3.30}),(\ref{3.33}), (\ref{3.34})
 (\ref{3.35}), and recalling the conservation of mass and energy
 $\|f(t)\|_{L^1}=4\pi\sqrt{2}N, \|f(t)\|_{L^1_2}=4\pi\sqrt{2}(N+2E)$ we obtain
(\ref {3.26}).
\vskip2mm

(II):
In the following we will use $\sg$-representation  (\ref{3.new5}) of $\bar{B}_K({\bf v-v}_*,\sg)$
and the bounds
\be b_0\min\{1,\,|{\bf {\bf v-v}_*}|^{2\eta}\}\le
\Phi(|{\bf v-v'}|, |{\bf {\bf v-v}_*'}|)\le 1\lb{kernel4}\ee
which comes from the assumption (\ref{1.6}).
Let
$${\cal V}=\Big\{({\bf v},{\bf v}_*,\sg)\in {\bRRS}\,\Big|\, \Phi(|{\bf v-v'}|,
 |{\bf {\bf v-v}_*'}|)\ge (4\pi)^2K|{\bf v-v'}||{\bf {\bf v-v}_*'}|\Big\},$$
${\cal V}^c={\bRRS}\setminus {\cal V}$, $q=p/(p-1)$, and recall
$ \la {\bf v} \ra=(1+|{\bf v}|^2)^{1/2}$. Then
using H\"{o}lder inequality we have
\bes&&{\cal D}_2(G(t))=
\fr{1}{4}\int_{{\bRRS}}\la {\bf v-v}_*\ra^2
\Gm(G'G_*', GG_*)
{\rm d}{\bf v}{\rm d}{\bf v}_*{\rm d}\sg \nonumber
\\
&&\le\fr{1}{4}
\Big(\int_{{\cal V}}
\fr{\la {\bf v-v}_*\ra^{2q}}{(\bar{B}_K)^{{q}/{p}}}\Gm(G'G_*', GG_*)
{\rm d}{\bf v}{\rm d}{\bf v}_*{\rm d}\sg\Big)^{1/{q}} \nonumber
\Big(\int_{{\cal V}}\bar{B}_K\Gm(G'G_*', GG_*)
{\rm d}{\bf v}{\rm d}{\bf v}_*{\rm d}\sg\Big)^{1/{p}} \nonumber
\\
&&+ \fr{1}{4}
\Big(\int_{{\cal V}^c}
\fr{\la {\bf v-v}_*\ra^{2q}}{(\bar{B}_K)^{{q}/{p}}}\Gm(G'G_*', GG_*)
{\rm d}{\bf v}{\rm d}{\bf v}_*{\rm d}\sg\Big)^{1/{q}}\Big(
\int_{{\cal V}^c }
\bar{B}_K \Gm(G'G_*', GG_*)
{\rm d}{\bf v}{\rm d}{\bf v}_*{\rm d}\sg
\Big)^{1/{p}} \nonumber\\
&&\le (I_{q}(t))^{1/{q}}(D_K(F(t))^{1/{p}}+(J_{q}(t))^{1/{q}}(D_K(F(t))^{1/{p}}\lb{DI}\ees
where we have used
$$\fr{1}{4}\int_{{\bRRS}}\bar{B}_K({\bf v-v}_*,\sg)\Gm(G'G_*', GG_*)
{\rm d}{\bf v}{\rm d}{\bf v}_*{\rm d}\sg
\le D_K(F(t)).$$
Let us compute using the equality in (\ref{3.new11}) that
\beas
I_{q}(t) &=& \fr{1}{4}\int_{\fr{1}{(4\pi)^2}\Phi(|{\bf v-v'}|, |{\bf {\bf v-v}_*'}|)\ge K|{\bf v-v'}||{\bf {\bf v-v}_*'}|}
\fr{\la{\bf v-v}_*\ra^{2q}}{(\bar{B}_K({\bf v-v}_*,\sg))^{q/{p}}}\Gm(G'G_*', GG_*)
{\rm d}{\bf v}{\rm d}{\bf v}_*{\rm d}\sg
\\
&=&\fr{1}{2}\int_{{\bRRS}}1_{\{GG_*>G'G_*'\}}
\fr{\la{\bf v-v}_*\ra^{2q}}{[\fr{ K}{2}|{\bf v-v}_*||{\bf v-v'}||{\bf {\bf v-v}_*'}|]^{q/p}}
\Gm(G'G_*', GG_*)
{\rm d}{\bf v}{\rm d}{\bf v}_*{\rm d}\sg.
\eeas
By definition of $\Gm(\cdot,\cdot)$ and using (\ref{3.31}) we have
\be GG_*>G'G_*'\,\Longrightarrow\,
\Gm(GG_*,G'G_*')\le C_{16} t^{\dt}GG_*\sqrt{1+|{\bf v}|^2+|{\bf v}_*|^2}.
\lb{log}\ee
Also we have
$$|{\bf v-v'}||{\bf {\bf v-v}_*'}|=\fr{1}{2}|{\bf v-v}_*|^2\sqrt{1-({\bf n}\cdot\sg)^2},\quad
{\bf n}=\fr{{\bf v-v}_*}{|{\bf v-v}_*|}$$
and $\fr{q}{p}=q-1<1$. So
\bes I_{q}(t) &\le & \fr{1}{2}C_{17}t^{\dt}\big(\fr{4}{K}\big)^{q-1}
\int_{{\bRRS}}
\fr{\la {\bf v-v}_*\ra^{2q}GG_*\sqrt{1+|{\bf v}|^2+|{\bf v}_*|^2}}{[|{\bf v-v}_*|^3
\sqrt{1-({\bf n}\cdot\sg)^2 }]^{q-1}}
{\rm d}{\bf v}{\rm d}{\bf v}_*{\rm d}\sg\nonumber
\\
&=& C_{18}t^{\dt}K^{-(q-1)}
\int_{{\bRR}}
\fr{\la {\bf v-v}_*\ra^{2q}GG_*\sqrt{1+|{\bf v}|^2+|{\bf v}_*|^2}}{|{\bf v-v}_*|^{3(q-1)}}
{\rm d}{\bf v}{\rm d}{\bf v}_*.\lb{3.39}
\ees
Next we have
\be \sqrt{1+|{\bf v}|^2+|{\bf v}_*|^2}
\le 2(\la{\bf v}\ra+ \la {\bf v-v}_*\ra)\le 4\la{\bf v}\ra \la {\bf v-v}_*\ra,\lb{3.40}\ee
\be
\int_{{\bR}}\fr{\la{\bf v-v}_*\ra^{2q+1}}{|{\bf v-v}_*|^{3(q-1)}}
G({\bf v}_*,t){\rm d}{\bf v}_*
\le C_{19}\big(1+\la {\bf v}\ra^{4-q}\|G(t)\|_{L^1_{4-q}}\big)\lb{3.41}
\ee
where we have used $1<q<2$ and $0<G(\cdot)<1$.  From
(\ref{3.39}), (\ref{3.40}), and (\ref{3.41}) we obtain
\be I_{q}(t)\le
C_{20}t^{\dt} K^{-(q-1)}\big(1+\|G(t)\|_{L^1_{4-q}}\big)
\|G(t)\|_{L^1_{5-q}}\le C_{21}t^{\dt}K^{-(q-1)},\quad t\ge 1
\lb{3.44}\ee
where we used $0\le G(\cdot,t)\le F(\cdot,t)\le f(\cdot, t)+{\cal E}$
so that
$$
\|G(t)\|_{L^1_s}\le
\|f(t)\|_{L^1_s}+\|{\cal E}\|_{L^1_s}\le C_{22}\qquad \forall\, t\ge 1,\,\, 0\le s\le 4.$$
Next using (\ref{kernel4}),
(\ref{log}), the first inequality in (\ref{3.40}), and the inequality
$\Phi(|{\bf v-v'}|, |{\bf {\bf v-v}_*'}|)\ge b_0\min\{1, |{\bf v-v}_*|^{2\eta}\}$ we compute as the above that
\bes&& J_{q}(t)=
\fr{1}{4}\int_{\fr{1}{(4\pi)^2}\Phi(|{\bf v-v'}|, |{\bf {\bf v-v}_*'}|)< K|{\bf v-v'}||{\bf {\bf v-v}_*'}|}
\fr{\la {\bf v-v}_*\ra^{2q}}{(\bar{B}_K({\bf v-v}_*,\sg))^{q/p}}
\Gm(G'G_*',GG_*)
{\rm d}{\bf v}{\rm d}{\bf v}_*{\rm d}\sg\nonumber
\\
&&\le C_{23}\int_{{\bRRS}}\fr{\la {\bf v-v}_*\ra^{2q}}{
\big(|{\bf v-v}_*|\min\{1, |{\bf v-v}_*|^{2\eta}\}\big)^{q/p}}
\Gm(G'G_*',GG_*){\rm d}{\bf v}{\rm d}{\bf v}_*{\rm d}\sg\nonumber
\\
&&=2C_{23}\int_{{\bRRS}, GG_*>G'G_*'}\fr{\la {\bf v-v}_*\ra^{2q}}{
\big(|{\bf v-v}_*|\min\{1, |{\bf v-v}_*|^{2\eta}\}\big)^{q/p}}
\Gm(G'G_*',GG_*){\rm d}{\bf v}{\rm d}{\bf v}_*{\rm d}\sg\nonumber
\\
&&\le C_{24}t^{\dt}\int_{{\bRRS}}\fr{\la {\bf v-v}_*\ra^{2q}(\la {\bf v}\ra+
\la {\bf v-v}_*\ra)}{\big(|{\bf v-v}_*|\min\{1, |{\bf v-v}_*|^{2\eta}\}\big)^{q/p}}GG_*
{\rm d}{\bf v}{\rm d}{\bf v}_*{\rm d}\sg\nonumber
\\
&&\le
C_{25}t^{\dt}\big(\|G(t)\|_{L^1_1}+\|G(t)\|_{L^1}\|G(t)\|_{L^1_{2+q}}\big)\le C_{26}t^{\dt},\quad t\ge 1\lb{3.45}\ees
where we have used the condition $(1+2\eta)\fr{q}{p}<3$ and $0<G(\cdot)<1$. Inserting (\ref{3.44}), (\ref{3.45}) into (\ref{DI}) and using
$\fr{q-1}{q}=\fr{1}{p}$ we obtain (\ref{3.38}).
\vskip2mm

(III): Let $B_{\rm min}({\bf v-v}_*,\og)$ be given by (\ref{cutoff}) with
$\Psi_0(r)=\min\{1, r^{1+2\eta}\}$, and let $D_{\rm min}(\cdot)$ be the corresponding
entropy dissipation. By Assumption \ref{assp}  and (\ref{3.new6}),
it is easily seen that $B_{\rm min}({\bf v-v}_*,\og)\le B_{K}({\bf v-v}_*,\og)$ so that
$D_{\rm min}(\cdot)\le D_K(\cdot)$ for all $K\ge 1$.

From (\ref{3.conserv}) and the lower and upper bounds  (\ref{3.lbd}), (\ref{3.ubd})
we see that Proposition \ref{prop2.1} applies to
$F(\cdot,t)=F^K(\cdot,t), G(\cdot,t)=G^K(\cdot,t)$ (for any $t\ge 1, K\in{\mN}$) so that
\be S(\Og)-S(F(t)))\le
C_{27}\Big( \big({\cal D}_2(G(t))\big)^{\fr{1}{10}}+\big(D_{\rm min}(F(t))\big)^{\fr{1}{4}}
\,\Big).\lb{3.47}\ee
For any $1<T<\infty$ and any $K\ge  T^{\dt(p-2)}$ we have from
part {\rm(II)} that for any $t\in [1,T]$
$$({\cal D}_2(G(t)))^{\fr{1}{10}}
\le (C_{2})^{\fr{1}{10}}2\big(t^{\dt}D_K(F(t))\big)^{\fr{1}{10p}}$$
hence from (\ref{3.47}) and $D_{\rm min}(\cdot)\le D_K(\cdot)$ we obtain
$$S(\Og)-S(F(t))\le
C_{3}\Big(\big(t^{\dt}D_K(F(t))\big)^{\fr{1}{10p}}+
\big(t^{\dt}D_K(F(t))\big)^{\fr{1}{4}}\Big)=\Psi\big(t^{\dt}D_K(F(t))\big)$$
that is,
$$t^{-\dt}\Psi^{-1}\big(S(\Og)-S(F(t))\big)\le D_K(F(t)),\quad t\in [1,T]$$
Inserting this inequality into (\ref{3.26}) in part {\rm(I)} gives (\ref{3.46}). The proof is complete.
\end{proof}
\vskip3mm

The next lemma deals with a differential inequality that implies an algebraic rate of decay.

\begin{lemma}\lb{lemma3.2}Let $ a>0, b>0, \alpha>0, 0<\beta\le \gm< 1, \dt<1,
k\ge 0$ and $u_0\ge 0$ be finite constants,
$ \Psi(y)=a( y^{\beta}+y^{\gm}), y\in[0,\infty)$, and
let $\Psi^{-1}(u), u\in [0,\infty), $ be the inverse function of $ \Psi(y)$.
Given any $T>1$. Assume that $t\mapsto u(t)\in [0,\infty)$ is absolutely continuous on $[1, T]$ satisfying
$u(1)\le u_0$ and
$$\fr{{\rm d}}{{\rm d}t}u(t)\le - t^{-\dt}\Psi^{-1}(u(t))+b t^{k} e^{-t^{\alpha}}\quad {\rm a.e.}
\quad t\in [1,T].$$
Then
$$u(t)\le C t^{-\ld}\qquad \forall\, t\in [1,T]$$
where $\ld=\fr{\beta }{1-\beta}(1-\dt)$ and
$0<C<\infty$ depends only on $a,b,\alpha, \beta, \gm, \dt, k $ and $u_0$.
In particular $C$ does not depend on $T$.
\end{lemma}
\begin{proof}  Let
$$m^*({\beta})=\max_{t\ge 1}t^{\beta(k+\dt)+\ld} e^{-\beta t^{\alpha}},\quad
m^*({\gm})=\max_{t\ge 1}t^{\gm(k+\dt)+\ld} e^{-\gm t^{\alpha}}$$
and choose a least constant $C>0$ such that $C\ge u_0$ and
\be\fr{{\ld}^{\beta}}{C^{1-\beta}}+
\fr{{\ld}^{\gm}}{C^{1-\gm}}+ \fr{{b}^{\beta}m^*({\beta})}{C}
+ \fr{{b}^{\gm}m^*(\gm)}{C}
\le \fr{1}{a}.\lb{3.50}\ee
Let
$U(t)=C t^{-\ld}.$
It is easily checked that (\ref{3.50}) and $\ld=\fr{\beta}{1-\beta}(1-\dt)>0$
imply
\beas\fr{{\rm d}}{{\rm d}t}U(t)+t^{-\dt}\Psi^{-1}(U(t))-bt^{k} e^{-t^{\alpha}}\ge 0\qquad \forall\, t\ge 1.
\eeas
From this and $u(1)\le U(1)$ and that $u(t)$ is absolutely continuous on $[1, T]$, we have
\beas&& \big(u(t)-U(t)\big)_{+}
=\int_{1}^{t}\Big(\fr{{\rm d}}{{\rm d}\tau}u(\tau)-
\fr{{\rm d}}{{\rm d}\tau}U(\tau)\Big)1_{\{u(\tau)>U(\tau)\}}{\rm d}\tau
\\
&&\le\int_{1}^{t}\Big(- \tau^{-\dt}\Psi^{-1}(u(\tau))+b \tau^{k} e^{-\tau^{\alpha}}
-\fr{{\rm d}}{{\rm d}\tau}U(\tau)\Big)1_{\{u(\tau)>U(\tau)\}}{\rm d}\tau
\\
&&\le\int_{1}^{t}\Big(- \tau^{-\dt}\Psi^{-1}(U(\tau))+b\tau^{k} e^{-\tau^{\alpha}}
-\fr{{\rm d}}{{\rm d}\tau}U(\tau)\Big)1_{\{u(\tau)>U(\tau)\}}{\rm d}\tau
\le 0\qquad \forall\, t\in[1,T].\eeas
Here we used the increase of $u\mapsto \Psi^{-1}(u) $ on $[0,\infty)$. Thus $u(t)\le U(t)$ for all $t\in [1,T].$
\end{proof}
\vskip3mm

Finally we can state and prove the main result of this section:

\begin{theorem}\label{theorem3.2} Let $B({\bf {\bf v-v}_*},\omega)$ satisfy Assumption \ref{assp}, let $F_0\in {\mathcal B}_1^{+}({\mathbb R}_{\ge 0})$ satisfy
$N(F_0)>0, E(F_0)>0$, and let $F_{\rm be}$ be
the unique Bose-Einstein distribution with the same mass $N=N(F_0)$ and energy $E=E(F_0)$.
Let $\ld$ satisfy $\fr{1}{10\max\{2, \fr{4+2\eta}{3}\}}<\ld<\fr{1}{10\max\{2, \fr{4+2\eta}{3}\}-1}$.
Let $F_t$ be a conservative measure-valued isotropic solution $F_t$ obtained in
Proposition \ref{prop3.2} with the initial datum $F_0$.
Then
$S(F_t)\ge S(F_0)$, $S(F_t)>0 $ for all $t>0$, and 
\be 0\le S(F_{\rm be})-S(F_t)\le C(1+t)^{-\ld}\quad \forall\, t\ge 0\lb{e-conv}\ee
where  $C\in(0,\infty)$ depends only on $N,E, b_0, \eta$
and $\ld$.
\end{theorem}

\begin{proof}  
As mentioned in the theorem, $F_t$
is a weak limit of a subsequence $f^{K_n}(\cdot,t)$ of  $f^K(\cdot,t)$ which are the
isotropic approximate  solutions of  Eq.(\ref{Equation1}) on
${\bR}\times [0,\infty)$  obtained in
Proposition \ref{prop3.2} with the initial data $f^K_0$
which together with $f^K$
have all properties in  part (I) and part (II) of Proposition \ref{prop3.2}.
Since $t_0$ in (\ref{entropy3.6}) is arbitrary, this implies that $S(F_t)>0$ for all $t>0$.
We also recall part {\rm(I)} of Lemma \ref{lemma3.1} that ensures
$S(F_t)\le S(F_{{\rm be}})$ for all $t\in[0,\infty)$.
From (\ref{mprod}) and (\ref{entropy3.5}) and taking $t_0=1$ we have
$$\inf\limits_{K\in {\mN}, t\ge 1}S(f^K(t)\ge S_0:=S_*(1)>0,\quad
\sup\limits_{K\in{\mN}, t\ge 1}\|f^K(t)\|_{L^1_{4}}\le C_0<\infty.$$
Note that since these constants $S_0, C_0$ depend only on $N, E,b_0$ and $\eta$, it follows that
$f^K_0, f^K$ satisfy all assumptions in Proposition \ref{prop4.2}
and, for the present case, the constants $C_1,C_2,...,C_{27}$ in
the proof of Proposition \ref{prop4.2} and thus the constants $C_{28},C_{29}, C_{30}\in (0,\infty)$ appeared below, all depend only on $N,E,
b_0,\eta $ and $\ld$.
Let $F^K({\bf v},t)=(1-e^{-t^{\dt}})f^K(|{\bf v}|^2/2,t)+
e^{-t^{\dt}}{\cal E}({\bf v})$ be defined in Proposition \ref{prop4.2}. Then
with the notations in Proposition \ref{prop4.2} and applying
Lemma \ref{lemma3.2} to part {\rm(III)} of Proposition \ref{prop4.2}  with functions $0\le u(t)=S(\Og)-S(F^K(t))\le S(\Og)$ and $\Psi(y)=C_{3}\big( y^{\fr{1}{10p}}
+y^{\fr{1}{4}}\big)$  we have
$$S(\Og)-S(F^K(t))\le  C_{28}t^{-\ld},\qquad \forall\, t\in [1, T].$$
On the other hand, from (\ref{S3}), (\ref{S1}) we have
\beas S(F^K(t))\le S\big((1-e^{-t^{\dt}})f^K(t)\big)+S\big(e^{-t^{\dt}}{\cal E}\big)\le
S(f^K(t))+C_{29} t^{\dt} e^{-t^{\dt}}\le \quad \forall\, t\ge 1.\eeas
Thus, for any $T>1$,
\be S(\Og)-S(f^K(t))
\le C_{28}t^{-\ld}+C_{29} t^{\dt} e^{-t^{\dt}}\le
C_{30} t^{-\ld}
\quad \forall\,t\in [1, T],\,\,\forall\, K\ge T^{\dt(p-2)}.\lb{3.SKT}\ee
Also, from the convergence (\ref{entropy3.7}), (\ref{3.weak2}),
the non-decrease of $t\mapsto S(f^K(t))$ (see (\ref{eelity})), the definition of entropy $S(\bar{F}_t)$, and
 $S(\bar{F}_t)=S(F_t)$, we have
$$S(F_0)=\lim_{n\to\infty}S(f_0^{K_n})=\limsup_{n\to\infty}S(f^{K_n}(0))\le \limsup\limits_{n\to\infty}S(f^{K_n}(t))\le S(\bar{F}_t)=S(F_t)$$
for all $t\in[0,\infty)$. Thus for any $T>1$, applying (\ref{3.SKT}) to $f^{K_n}(\cdot,t)$ with $K_n> T^{\dt(p-2)}$, we obtain (because $S(F_{\rm be})=S(\Og)$)
\beas&&
S(F_{{\rm be}})-S(F_t)\le \limsup_{n\to\infty}\big(
S(\Og)-S(f^{K_n}(t))\big)\le C_{30} t^{-\ld}
\qquad \forall\,t\in [1, T].\eeas
Since $C_{30}$ is independent of $T$ and $T$ can be  arbitrarily large,  we conclude
$$
0\le S(F_{{\rm be}})-S(F_t)\le C_{30} t^{-\ld}
\qquad \forall\,t\in [1, \infty).$$
This gives (\ref{e-conv}) and finishes the proof of the theorem.
\end{proof}

\begin{center}\section{Rate of Convergence to BEC}\end{center}

In this section we prove the second part of the main result Theorem \ref{theorem1.2}: an algebraic rate of
convergence of $F_t(\{0\})$ to the Bose-Einstein condensation
$F_{\rm be}(\{0\})=(1-(\overline{T}/\overline{T}_c)^{3/5})N$.
The completion of the proof of Theorem \ref{theorem1.2}
is given at the end of this section.
We first prove some properties of the collision integrals
that have been proven to hold for the hard sphere model,
then we present our new progress on proving lower bounds of condensation $F_t(\{0\})$
without any additional assumption on the initial data $F_0$.

The monotone assumption (\ref{Phi*}) is now used to prove the following

\begin{proposition}
\lb{convex-positivity}({\bf Convex-Positivity}). Let $W(x,y,z), {\cal K}[\vp](x,y,z)$ be defined in
(\ref{JKW})-(\ref{Y}) with $\Phi$ satisfying (\ref{Phi}),(\ref{Phi*}). Then for any $F\in {\cal B}_{1}^{+}({\mathbb R}_{\ge 0})$
and any convex function $\vp \in C^{1,1}_b({\mathbb R}_{\ge 0})$ we have
\be
\int_{{\mathbb R}_{\ge 0}^3}{\cal K}[\vp]{\rm d}^3F
\ge \int_{{\mathbb R}_{\ge 0}^3}{\cal K}_1[\vp]{\rm d}^3F+\int_{{\mathbb R}_{>0}^3}{\cal K}_2[\vp]{\rm d}^3F\ge 0 \lb{4.4}
\ee
where
\bes&&{\cal K}_1[\vp](x,y,z)={\bf 1}_{\{0\le x<y\le z\}}
\chi_{y,z}W(x,y,z)\Dt_{\rm sym}\vp(x,y,z)\ge 0, \lb{4.5}
\lb{}\\
&&{\cal K}_2[\vp](x,y,z)={\bf 1}_{\{0<y\le z<
x<y+z\}}
\chi_{y,z}W(x,y,z)\Dt \vp(x,y,z)\ge 0,\dnumber \lb{4.3}\\
&& \chi_{y,z} =\left\{\begin{array}{ll} 2
 \,\,\,\,\qquad\quad {\rm if} \quad y< z,\\
1\,\,\,\,\qquad\quad {\rm if} \quad y=z
\end{array}\right.\dnumber \lb{chi}\ees
for all $(x,y,z)\in{\mR}_{\ge 0}^3$, $ \Dt \vp(x,y,z)$ is
given in (\ref{diff}), and
\bes&& \Dt_{\rm sym}\vp(x,y,z)= \vp(z+y-x)+\vp(z+x-y)-2\vp(z)\nonumber\\
&&=(y-x)^2\int_{0}^1\!\!\!\int_{0}^1\vp''
(z+(s-t)(y-x))
{\rm d}s{\rm d}t,\quad 0\le x, y\le z. \lb{4.7}\ees
\end{proposition}

\begin{proof} The positivity in (\ref{4.5}),(\ref{4.3})
follows from the convexity of $\vp$ and (\ref{4.7}), (\ref{diff}).
To prove the inequality (\ref{4.4}), we first use the
symmetry
$W(x,y,z)= W(x,z,y)$ on ${\mR}_{\ge 0}^3$
(see Remark \ref{remark6.1} in Appendix) and (\ref{diff}) to get
\be {\cal K}[\vp](x,y,z)={\cal K}[\vp](x,z,y),\qquad
 {\cal K}[\vp](x,y,z)|_{x=y}={\cal K}[\vp](x,y,z)|_{x=z}=0\lb{sym}\ee
and then we make a decomposition according to
$x$ lies in the left side, between, and right side of $y$ and $z$ respectively (using (\ref{sym}) and recalling that
${\rm d}^3F={\rm d}F(x){\rm d}F(y){\rm d}F(z)$):
\beas \int_{{\mathbb R}_{\ge 0}^3}{\cal K}[\vp]{\rm d}^3F
&=&\left(2\int_{0\le x<y<z}+2\int_{0\le y<x<z}+
\int_{0\le x<y=z}+\int_{0\le y, z<x}
\right)W(x,y,z)\Dt\vp(x,y,z){\rm d}^3F
\\
&:=& I_1+I_2+I_3+I_4.\eeas
After exchanging notations $x\leftrightarrow y$ for the integrand in $I_2$
we have
\bes
I_1+I_2 &=&
2\int_{0\le x<y<z}W(x,y,z)\Dt_{{\rm sym}}\vp(x,y,z){\rm d}^3F
\nonumber\\
&+&
2\int_{0\le x<y<z}\big(W(y,x,z)-W(x,y,z)\big)\Dt\vp(y,x,z){\rm d}^3F
\lb{IIW}\ees
where we used $\Dt\vp(x,y,z)+\Dt\vp(y,x,z)=\Dt_{{\rm sym}}\vp(x,y,z)$
for $0\le x,y\le z$. We need to prove
\be \big(W(y,x,z)-W(x,y,z)\big)\Dt\vp(y,x,z)\ge 0\qquad \forall\,0\le x<y<z.\lb{WP}\ee
By convexity of $\vp$ and (\ref{diff}) we have
\beas \Dt\vp(y,x,z)=\vp(y)+\vp(x+z-y)-\vp(x)-\vp(z)\le 0
\qquad \forall\, 0\le x<y\le z.\eeas
Therefore to prove (\ref{WP}) we need only prove
\be W(y,x,z)\le W(x,y,z)\qquad \forall\,0\le x<y\le z.\lb{WL}\ee
For any $0< x<y\le z$ we have
\be|\sqrt{x}-\sqrt{y}|\vee |\sqrt{x_*}-\sqrt{z}|
=|\sqrt{x}-\sqrt{y}|,\quad (\sqrt{x}+\sqrt{y})\wedge(\sqrt{x_*}+\sqrt{z})=\sqrt{x}+\sqrt{y}
\lb{minmax}\ee
hence
\beas&&
 W(x,y,z)=\fr{1}{4\pi\sqrt{xyz}}
\int_{|\sqrt{x}-\sqrt{y}|}
^{\sqrt{x}+\sqrt{y}}{\rm d}s
\int_{0}^{2\pi}\Phi(\sqrt{2}s, \sqrt{2} Y_*){\rm d}\theta,
\\
&&
 W(y,x,z)=\fr{1}{4\pi\sqrt{xyz}}
\int_{|\sqrt{y}-\sqrt{x}|}
^{\sqrt{y}+\sqrt{x}}{\rm d}s
\int_{0}^{2\pi}\Phi(\sqrt{2}s, \sqrt{2} Y_*^{\sharp}){\rm d}\theta\eeas
where $Y_*=Y_*(x,y,z,s,\theta)$ is defined in (\ref{Y}) and
$Y_*^{\sharp}=Y_*(y,x,z,s,\theta)$.
By calculation (with the relation (\ref{5.pp}) in Appendix) it is easily checked that
$$ Y_*^{\sharp}\le Y_*\qquad \forall\, s\in [\sqrt{y}-\sqrt{x},
\sqrt{y}+\sqrt{x}]\qquad (0<x<y\le z).$$
Since $\rho\to \Phi(r,\rho)$ is non-decreasing, it follows that
$W(y,x,z)\le W(x,y,z)$.
For the case $0=x<y\le z$, we have
$\Phi(\sqrt{2y}, \sqrt{2(z-y)}\,)\le \Phi(\sqrt{2y}, \sqrt{2z}\,)$ which
means (by definition of $W$) that the inequality $W(y,0,z)\le W(0,y,z)$
holds also true.  This proves (\ref{WL}).

Now from (\ref{IIW}),(\ref{WP}) and $\Dt\vp(x,y,z)|_{y=z}=\Dt_{{\rm sym}}\vp(x,y,z)|_{y=z}$, and
recalling definition of $\chi_{y,z}$  we obtain
\beas&&
I_1+I_2+I_3  \ge
2\int_{0\le x<y<z}W(x,y,z)\Dt_{{\rm sym}}\vp(x,y,z){\rm d}^3F+
I_3\\
&&=\int_{0\le x<y\le z}\chi_{y,z}W(x,y,z)\Dt_{\rm sym}\vp(x,y,z){\rm d}^3F
=\int_{{\mathbb R}_{\ge 0}^3}{\cal K}_1[\vp]
{\rm d}^3F.\eeas
For the last term $I_4$, observe  that
$0\le y, z<x$ and $W(x,y,z)>0$ imply $y>0, z>0$ and $x<y+z$.
Thus
$$I_4=\int_{0<y, z<x} W(x,y,z)
\Dt\vp(x,y,z){\rm d}^3F
=\int_{{\mathbb R}_{> 0}^3}{\cal K}_2[\vp]
{\rm d}^3F.$$
\end{proof}

{\bf Notation}. To study local behavior of a measure $F\in {\cal B}^{+}({\mathbb R}_{\ge 0})$
near the origin, we introduce the following integrals. For
$p>0, \varepsilon>0, \alpha\ge 0$,   define
\beas&& N_{0,p}(F,\varepsilon)=\int_{{\mathbb R}_{\ge 0}}\big[\big(1-\frac{x}{\varepsilon}\big)_{+}\big]^p{\rm d}F(x)=\int_{[0, \varepsilon]}\big(1-\frac{x}{\varepsilon}\big)^p{\rm d}F(x),\\
&&
N_{\alpha, p}(F,\varepsilon)=\frac{1}{\varepsilon^{\alpha}}N_{0,p}(F,\varepsilon),\quad
\,\quad \underline{N}_{\alpha, p}(F,\varepsilon)=\inf_{0<\delta\le \varepsilon}N_{\alpha,p}(F,\delta),\\
&&
A_{0, p}(F,\varepsilon)=\int_{[0,\varepsilon]}\big(\frac{x}{\varepsilon}\big)^{p}
{\rm d}F(x),\quad A_{\alpha, p}(F,\varepsilon)=\fr{1}{\vep^{\alpha}}\int_{[0,\varepsilon]}
\big(\frac{x}{\varepsilon}\big)^{p}
{\rm d}F(x).\eeas

\begin{lemma}\lb{lemma4.2}Let $B({\bf {\bf v-v}_*},\omega)$ be given by (\ref{kernel}),
 (\ref{Phi}), (\ref{Phi*}) where
$\Phi$  also satisfies $0\le \Phi\le 1$ on
${\mR}_{\ge 0}^2$.
Let $F_t\in {\cal B}_{1}^{+}({\mR}_{\ge 0})$ be  a conservative  measure-valued isotropic solution of Eq.(\ref{Equation1}) on
$[0, \infty)$ with the initial datum $F_0$
satisfying $N=N(F_0)>0, E=E(F_0)>0$.
Then with
$c=\sqrt{NE}$ we have:

{\rm (I)} For any convex function $0\le \vp\in C^{1,1}_b({\mR}_{\ge 0})$
\be e^{ct}\int_{{\mR}_{\ge 0}}\vp{\rm d}F_t-e^{cs}
\int_{{\mR}_{\ge 0}}\vp{\rm d}F_s\ge \int_{s}^{t}e^{c\tau}{\rm d}\tau \int_{{\mR}_{\ge 0}^3}
{\cal K}[\vp]{\rm d}^3F_{\tau}\ge 0\quad \forall\, 0\le s<t.\lb{4.12}\ee

{\rm (II)}  For any convex function $0\le \vp\in C_b({\mR}_{\ge 0})$, the function
$t\mapsto e^{ct}\int_{{\mR}_{\ge 0}}\vp{\rm d}F_t$ is non-decreasing on $[0,\infty)$, and thus
 for any $p\ge 1, \vep>0, \alpha>0$,  the functions
$t\mapsto e^{ct}N_{0,p}(F_t,\vep), \\
t\mapsto e^{ct}\underline{N}_{\alpha,p}(F_t,\vep),$ and
$t\mapsto e^{ct}F_t(\{0\})$ are all non-decreasing on $[0,\infty)$.
\end{lemma}

\begin{proof}  {\rm(I)}: We first prove that
\be {\cal J}[\vp](y,z)\ge -\fr{1}{2}\vp(y)\sqrt{z}-\fr{1}{2}\vp(z)\sqrt{y}\lb{4.10}\ee
Since $\vp\in C^{1,1}_b({\mR}_{\ge 0})$ is convex, it is easily seen that $\vp$ is non-increasing on ${\mathbb R}_{\ge 0}$.
Given any $y,z\ge 0$. By symmetry ${\cal J}[\vp](y,z)={\cal J}[\vp](z,y)$ we may assume that $y\le z$.
By (\ref{diff}) we have $\Dt\vp(x,y,z)\ge 0$ for all $x\in [0,y]\cup [z,y+z]$ and so
  ${\cal J}[\vp]\ge \fr{1}{2}\int_y^z \sqrt{x}W(x,y,z)\Dt\vp(x,y,z){\rm d}x$.
  To prove (\ref{4.10}) we can assume that $y<z$. By the assumption on $\Phi$, the definition of $W(x,y,z)$, and that $\vp$ is non-negative and non-increasing, we deduce $0\le W(x,y,z)\le  1/{\sqrt{xz}}$ and $\Dt\vp(x,y,z)\ge-\vp(y)$ for all $x\in (y,z)$. Thus
  $${\cal J}[\vp]\ge -\fr{1}{2\sqrt{z}}\int_y^z \Dt\vp(x,y,z)dx\ge -\fr{1}{2\sqrt{z}}\vp(y)(z-y)\ge -\fr{1}{2}\vp(y)\sqrt{z}$$
 and so (\ref{4.10}) holds true.

 From (\ref{4.10}) we have
 \beas
\int_{{\mathbb R}_{\ge 0}^2}{\cal J}[\vp]{\rm d}^2F_t\ge-M_{1/2}(F_t)\int_{{\mathbb R}_{\ge 0}}\vp{\rm d}F_t
\ge -c \int_{{\mathbb R}_{\ge 0}}\vp{\rm d}F_t
\eeas
 where we used Cauchy-Schwarz inequality and the conservation of mass and energy
 to get $M_{1/2}(F_t)\le \sqrt{N(F_t)E(F_t)}=\sqrt{NE}=c$. So by definition of  measure-valued isotropic solutions we obtain
$$\frac{{\rm d}}{{\rm d}t}\int_{{\mathbb R}_{\ge 0}}\varphi{\rm
d}F_t\ge-c \int_{{\mathbb R}_{\ge 0}}\vp{\rm d}F_t+ \int_{{\mathbb R}_{\ge 0}^3}{\cal K}[\varphi]{\rm d}^3F_t\qquad \forall\,t\in[0, \infty)$$
i.e.
$$\frac{{\rm d}}{{\rm d}t}\Big(e^{ct}\int_{{\mathbb R}_{\ge 0}}\varphi{\rm
d}F_t\Big)\ge e^{ct}\int_{{\mathbb R}_{\ge 0}^3}{\cal K}[\varphi]{\rm d}^3F_t\ge 0\quad \forall\,
t\in[0,\infty).$$
This implies (\ref{4.12}), here  `` $\ge 0$" is due to the convex-positivity
(Proposition \ref{convex-positivity}).

(II): From {\rm(I)} we see that for any convex function $0\le \vp\in C_b^{1,1}({\mR}_{\ge 0})$,
the function $t\mapsto e^{ct}\int_{{\mR}_{\ge 0}}\vp{\rm d}F_t$ is non-decreasing on $[0,\infty)$.
Using approximation as did in \cite{Lu2016},
this monotone property holds also for any convex function $0\le \vp\in C_b({\mR}_{\ge 0})$.
Applying this to the convex function $\vp(x)=[(1-x/\vep)_{+}]^p\,(\vep>0,p\ge 1)$ we deduce that $e^{ct}N_{0,p}(F_t,\vep), e^{ct}N_{\alpha,p}(F_t,\vep)$, and thus $e^{ct}\underline{N}_{\alpha,p}(F_t,\vep)$
are all non-decreasing on $t\in [0,\infty)$. Since
 $F_t(\{0\})=\lim\limits_{\vep\to 0+}N_{0,p}(F_t,\vep)$
 for all $t\ge 0$, it follows that $t\mapsto e^{ct}F_t(\{0\})
$ is also non-decreasing on $[0,\infty)$.
 \end{proof}
\vskip2mm

\begin{lemma} \lb{lemma4.3} Let ${\cal J}[\vp], {\cal K}[\vp]$\
be defined in (\ref{JKW})-(\ref{Y}) where $\Phi(r,\rho)$ satisfies
 Assumption \ref{assp} with $0\le\eta<1$.
Let $F\in{\cal B}_1^{+}({\mathbb R}_{\ge 0}) $, $\vp_{\vep}(x)=[(1-x/\vep)_{+}]^2, 0<\vep\le 1$, $0\le \alpha<1-\eta$.
Then
\bes &&
\int_{{\mathbb R}_{\ge 0}^2}{\cal J}[\vp_{\vep}]{\rm d}^2F\ge
\frac{b_0}{34}\vep^{3/2}
\Big(\int_{[\fr{1}{2}\vep,1]}y^{-\fr{1-\eta}{2}}{\rm d}F(y)\Big)^2
- 2M_{1/2}(F) N_{0,2}(F,\vep),\lb{JJ}\\
&& \int_{{\mathbb R}_{\ge 0}^3}{\cal K}[\vp_{\vep}]{\rm d}^3F
\ge \fr{b_0}{2}\underline{N}_{\alpha,2}(F,\vep)\big(A_{\beta,p}(F,\vep)\big)^2
\dnumber\lb{KK}\ees
where $p=\fr{3}{2}+\alpha, \beta=\fr{1-\alpha-\eta}{2}$.
Also if $0<\vep\le 2/3$ and $0<\gm\le (24)^{-1/2}$, then
\be
 \int_{{\mathbb R}_{\ge 0}^3}{\cal K}[\vp_{\vep}]{\rm d}^3F
 \ge \fr{b_0}{4}\fr{\gm^{2}}{\vep^{1-\eta}}
N_{0,2}(F,\gm \vep)\big(F\big(\big[\gamma \vep,\fr{3}{2}\vep\big]\big)\big)^2.\lb{KK2}
\ee
\end {lemma}

\begin{proof}  We first prove that
\be W(x,y,z)\ge \fr{b_0}{2}\fr{z^{\eta}}{\sqrt{yz}}
\qquad \forall\, 0\le x<y\le z\le 1.\lb{W}\ee
First of all from the assumption on $\Phi$  we have (see (\ref{1.6}))
\be\Phi(r,\rho)
\ge b_0\min\{1,\,(r^2+\rho^2)^{\eta}\}.\lb{Phi2}\ee
Take any $0\le x< y\le z\le 1$.

\noindent If $x=0$, then (recall (\ref{W2})) $W(0,y,z)=\fr{1}{\sqrt{yz}}
\Phi(\sqrt{2y}, \sqrt{2z}\,)\ge b_0\fr{z^{\eta}}{\sqrt{yz}}$.

\noindent Suppose $x>0$. By (\ref{minmax}) we see that if
 $s\in[\sqrt{y}-\sqrt{x}, \sqrt{x}+\sqrt{y}]$, then using definition (\ref{Y}) of $Y_*$ and the property (\ref{5.pp}) in Appendix we have for all $\theta\in [0,2\pi]$
$$Y_*\ge \Big|\sqrt{z-\fr{(x-y+s^2)^2}{4s^2}
}
-\sqrt{x-\fr{(x-y+s^2)^2}{4s^2}
}\,\Big|
\ge  \fr{z-x}{ \sqrt{z} +\sqrt{x}}
=\sqrt{z}-\sqrt{x}
$$
from which we see that if $s\in [\sqrt{y}, \sqrt{x}+\sqrt{y}]$ then $
s+Y_*\ge \sqrt{z}$ and so using (\ref{Phi2}) gives
$$\Phi(\sqrt{2}\,s, \sqrt{2}\,Y_*)
\ge b_0\min\{1,\, (s+Y_*)^{2\eta}\}\ge b_0z^{\eta}\qquad \forall\, \theta\in[0,2\pi]$$
and thus
$$W(x,y,z)
=\fr{1}{4\pi\sqrt{xyz}}
\int_{\sqrt{y}-\sqrt{x}}
^{\sqrt{x}+\sqrt{y}}{\rm d}s
\int_{0}^{2\pi}\Phi(\sqrt{2}s, \sqrt{2} Y_*){\rm d}\theta
\ge
\fr{b_0}{2} \fr{z^{\eta}}{\sqrt{yz}}.$$
This proves (\ref{W}).

According to (\ref{JKW}),(\ref{diff}) we have a
decomposition
${\cal J}[\vp_{\vep}]={\cal J}^{+}[\vp_{\vep}]-{\cal J}^{-}[\vp_{\vep}]$
where
\beas &&{\cal J}^{+}[\vp_{\vep}](y,z)=\frac{1}{2}
\int_{0}^{y+z}W(x,y,z)(\vp_{\vep}(x)+\vp_{\vep}(y+z-x))
\sqrt{x}{\rm d}x,\\
&&{\cal J}^{-}[\vp_{\vep}](y,z)=\frac{1}{2}
\int_{0}^{y+z}W(x,y,z)
\sqrt{x}{\rm d}x (\varphi(y)+\varphi(z)).\eeas
Using the symmetry ${\cal J}^{+}[\vp_{\vep}](y,z)={\cal J}^{+}[\vp_{\vep}](z,y)$
and (\ref{W}) and then omitting $\vp_{\vep}(y+z-x)$
we compute
 \beas&&\int_{{\mathbb R}_{\ge 0}^2}{\cal J}^{+}[\vp_{\vep}](y,z){\rm d}F(y){\rm d}F(z)
 \ge \int_{ \fr{1}{2}\vep\le y\le z\le 1}\chi_{y,z}{{\cal J}^+}[\vp_{\vep}](y,z){\rm d}F(y){\rm d}F(z)\\
 &&\ge \fr{b_0}{4}\int_{ \fr{1}{2}\vep\le y\le z\le 1}\chi_{y,z}\fr{z^{\eta}}{\sqrt{yz}}\Big(\int_{0}^{\fr{1}{2}\vep}\vp_{\vep}(x)\sqrt{x}\,{\rm d}x \Big){\rm d}F(y){\rm d}F(z)\\
&&\ge \fr{b_0}{4}\vep^{3/2}\int_{ \fr{1}{2}\vep\le y\le z\le 1}\chi_{y,z}\fr{y^{\fr{\eta}{2}}z^{\fr{\eta}{2}}}{\sqrt{yz}}\Big(\int_{0}^{\fr{1}{2}}
(1-u)^2\sqrt{u}\,{\rm d}u \Big){\rm d}F(y){\rm d}F(z)\\
&&\ge \fr{b_0}{34}\vep^{3/2}\Big(\int_{[\fr{1}{2}\vep, 1]}
y^{-\fr{1-\eta}{2}}{\rm d}F(y)\Big)^2.\eeas
Next from $\Phi(r,\rho)\le 1$ we have
\beas
{{\cal J}^-}[\vp_{\vep}](y,z) &\le& \fr{\vp_{\vep}(y)+\vp_{\vep}(z)}{2}\Big(\fr{1}{3}\cdot\fr{y\wedge
z}{\sqrt{y\vee z}} {\bf 1}_{\{y\vee z>0\}}+ \sqrt{y\vee z}\Big)
\\
&\le &\fr{\vp_{\vep}(y)+\vp_{\vep}(z)}{2}(\sqrt{y}+\sqrt{z})
\le \vp_{\vep}(y)\sqrt{z}+\vp_{\vep}(z)\sqrt{y}\eeas
where we used the inequality
$\varphi_\varepsilon(y)\sqrt{y}+\varphi_\varepsilon(z)\sqrt{z}\le \varphi_\varepsilon(y)\sqrt{z}+\varphi_\varepsilon(z)\sqrt{y}$
which is because $(\varphi_\varepsilon(y)-\varphi_\varepsilon(z))(\sqrt{y}-\sqrt{z})\le 0$
since the function $x\mapsto \varphi_{\vep}(x) $ is non-increasing.
Thus we obtain
\beas \int_{{\mathbb R}_{\ge 0}^2}{{\cal J}^-}[\vp_{\vep}](y,z){\rm d}F(y){\rm d}F(z)
&\le&
\int_{{\mathbb R}_{\ge 0}^2}\big(
\vp_{\vep}(y)\sqrt{z}+\vp_{\vep}(z)\sqrt{y}\big){\rm d}F(y){\rm d}F(z)
 \\
 &=& 2M_{1/2}(F)N_{0,2}(F,\varepsilon).\eeas
This together with the above estimate proves the first inequality (\ref{JJ}).

Now we prove the second inequality (\ref{KK2}). It is easily seen that the function $x\mapsto \vp_{\vep}(x)$ is convex, belongs to
$C_b^{1,1}({\mathbb R}_{\ge 0})$, and holds the inequality
\be\Dt_{\rm sym}\vp_{\vep}(x,y,z)\ge \fr{(y-x)^2}{\vep^2} \quad {\rm for\,\,all}\,\,\, 0\le x\le y\le z\le \vep.\lb{vp}\ee
Let $0<\vep\le 1$. By Proposition \ref{convex-positivity} and (\ref{W}), (\ref{vp}) we have
\beas
\int_{{\mathbb R}_{\ge 0}^3}{\cal K}[\vp_{\vep}]{\rm d}^3F
&\ge &
\int_{0\le x\le y\le z\le \vep,\, y>0}\chi_{y,z}
\fr{b_0}{2}\fr{z^{\eta}}{\sqrt{yz}}
\fr{(y-x)^2}{\vep^2}{\rm d}F(x){\rm d}F(y){\rm d}F(z)
\\
&=&\fr{b_0}{2\vep^2}\int_{0<y\le z\le \vep}\chi_{y,z}
\fr{y^{3/2} z^{\eta}}{\sqrt{z}}\Big(\int_{[0,y]}(1-x/y)^2{\rm d}F(x)\Big)
{\rm d}F(y){\rm d}F(z)
\\
&=&\fr{b_0}{2\vep^2}\int_{0<y\le z\le \vep}\chi_{y,z}
\fr{y^{\fr{3}{2}+\alpha}z^{\fr{3}{2}+\alpha}}{z^{2-\eta+\alpha}}
N_{\alpha,2}(F,y)
{\rm d}F(y){\rm d}F(z)
\\
&\ge &\fr{b_0}{2}\underline{N}_{\alpha,2}(F,\vep)\fr{1}{
\vep^{4+\alpha-\eta}}\int_{0<y\le z\le \vep}\chi_{y,z}
y^{\fr{3}{2}+\alpha}z^{\fr{3}{2}+\alpha}
{\rm d}F(y){\rm d}F(z)
\\
&= &\fr{b_0}{2}\underline{N}_{\alpha,2}(F,\vep)\fr{1}{\vep^{1-\alpha-\eta}}
\big(A_{0,p}(F,\vep)\big)^2=
\fr{b_0}{2}\underline{N}_{\alpha,2}(F,\vep)\big(A_{\beta,p}(F,\vep)\big)^2.
\eeas
Finally, assuming $0<\vep\le 2/3, 0<\gm\le (24)^{-1/2}$, we compute
\beas&&  \int_{{\mR}^3_{\ge 0}}  {\cal K}[\varphi_{\vep}]\,{\rm d}^3F\ge \int_{{\mR}^3_{\ge 0}}  {\cal K}_1[\varphi_{\vep}]\,{\rm d}^3F \\
&&\ge \int_{0\le x<y\le z\le \vep}  \chi_{y,z}\fr{b_0}{2}\fr{z^{\eta}}{\sqrt{yz}}\fr{(y-x)^{2}}{{\vep}^{2}}\,{\rm d}^3F+\int_{0\le x\le \fr{\vep}{4},\vep< y\le z \le \fr{3}{2}\vep}  \chi_{y,z}
 \fr{b_0}{2}\fr{z^{\eta}}{\sqrt{yz}}\Big(1-\frac{z+x-y}{\vep}\Big)_{+}^2\,{\rm d}^3F
\quad \\&&
\ge \fr{b_0}{2\vep^2}\int_{\gamma \vep\le y\le z\le \vep} \chi_{y,z} \frac {y^{3/2}z^{1/2}
N_{0,2}(F,y)}{z^{1-\eta}}\,{\rm d}^2F+\fr{b_0}{2}\frac{1}{16}F([0,\frac{\vep}{4}])\int_{\vep < y\le z\le \fr{3}{2}\vep} \chi_{y,z} \fr{1}{z^{1-\eta}}\sqrt{\fr{z}{y}}{\rm d}^2 F \\
&&\ge \fr{b_0}{2\vep^2} N_{0,2}(F,\gm \vep)\fr{(\gm\vep)^{2}}{\vep^{1-\eta}}
\int_{\gm \vep\le y\le z\le  \vep}  \chi_{y,z}{\rm d}^2F +\fr{b_0}{2}\fr{1}{16}N_{0,2}(F,\gm \vep)
\fr{1}{(\fr{3}{2}\vep)^{1-\eta} }\int_{\vep < y\le z\le \fr{3}{2}\vep}  \chi_{y,z}{\rm d}^2 F \\
&&=\fr{b_0}{2} N_{0,2}(F,\gm \vep)\fr{\gm^{2}}{\vep^{1-\eta}}
\int_{\gm \vep\le y\le z\le  \vep}  \chi_{y,z}{\rm d}^2F +\fr{b_0}{2}\fr{1}{16}N_{0,2}(F,\gm \vep)
\big(\fr{2}{3}\big)^{1-\eta}\fr{1}{\vep^{1-\eta}}\int_{\vep < y\le z\le \fr{3}{2}\vep}  \chi_{y,z}{\rm d}^2 F \\
&&\ge \fr{b_0}{2}
N_{0,2}(F,\gm \vep)\gm^2\fr{1}{\vep^{1-\eta}}
\Big(\big(F([\gamma \vep,\vep])\big)^2+\big(F((\vep,\fr{3}{2}\vep])\big)^2\Big) \\
&&\ge
\fr{b_0}{4}\fr{\gm^{2}}{\vep^{1-\eta}}
N_{0,2}(F,\gm \vep)
\big(F([\gamma \vep, \fr{3}{2}\vep])\big)^2.\eeas
Here we have used $\fr{1}{16}\big(\fr{2}{3}\big)^{1-\eta}\ge \fr{1}{24}\ge \gm ^{2} $.
This completes the proof.
\end{proof}
\vskip3mm

In the following we will use a convention: in the set ${\mR}_{\ge 0}$ we define the arithmetic operation
\be \fr{b}{a}= \infty \,\,(=+\infty)\quad {\rm if}\quad a=0<b.\lb{convention}\ee

\begin{lemma}\lb{lemma4.4} Let $B({\bf {\bf v-v}_*},\omega)$ satisfy Assumption \ref{assp} with $0\le\eta<1$.
Let
$F_t\in {\cal B}_{1}^{+}({\mR}_{\ge 0})$ be  a conservative  measure-valued isotropic solution of Eq.(\ref{Equation1}) on $[0, \infty)$ with the initial datum $F_0$
satisfying $N=N(F_0)>0, E=E(F_0)>0$. Let
$c=\sqrt{NE}, 0<\vep\le 1$. Then

{\rm (I)} For any $\tau\ge 0, T>0$
\be N_{0,2}(F_{\tau+t},\vep)\ge
\frac{b_0}{34}{\vep}^{3/2}\inf_{s\in [0, T]}
\Big(\int_{[\fr{1}{2}\vep,1]}y^{-\fr{1-\eta}{2}}{\rm d}F_{\tau+s}(y)\Big)^2 e^{-ct}t\quad \forall\, t\in [0,T].
\lb{inf2}\ee
Moreover if $0<\vep\le 2/3, 0<\gm\le (24)^{-1/2}$, then
\be N_{0,2}(F_{\tau+t},\vep)\ge
\fr{b_0}{4}\fr{\gm^{2}}{\vep^{1-\eta}}\inf_{s\in[0, T]}
\Big(N_{0,2}(F_{\tau+s},\gm \vep)\big(F_{\tau+s}\big(\big[\gamma \vep,\fr{3}{2}\vep\big]\big)\big)^2\Big) e^{-ct}t \quad \forall\, t\in[0,T].
\lb{inf3}\ee

{\rm (II)} Let $0\le\alpha<1-\eta, p=\fr{3}{2}+\alpha,\beta=\fr{1-\alpha-\eta}{2},$ and
$h>0$. Then
\be e^{2ch}F_{t+h}(\{0\})
\ge N_{0,p}(F_{t},\vep)-\Big(\fr{2e^{3ch} N}{hb_0\underline{N}_{\alpha, 2}(F_t,\vep)}\Big)^{1/2}
\Big(\fr{p}{\beta}\Big)^p\vep^{\beta}\quad \forall\, t\ge 0.\lb{4.15}\ee
In particular for $\alpha=0$ we have
\be e^{2ch}F_{t+h}(\{0\})\ge N_{0,3/2}(F_{t},\vep)
-\Big(\fr{2e^{3c h} N}{hb_0F_t(\{0\})}\Big)^{1/2}
\Big(\fr{3}{1-\eta}\Big)^{3/2}\vep^{\fr{1-\eta}{2}}\quad \forall\, t\ge 0.\lb{4.16}\ee
\end{lemma}

\begin{proof}  {\rm(I)}:
Since $\vp_{\vep}(x)=[(1-x/\vep)_+]^2$ is convex and belongs to
$C^{1,1}_b({\mR}_{\ge 0})$, we deduce from Proposition \ref{convex-positivity},
Lemma \ref{lemma4.3}, and $M_{1/2}(F_{\tau+t})\le \sqrt{NE}=c$
that
\beas&&\frac{{\rm d}}{{\rm d}t}N_{0,2}(F_{t},\vep)=
 \int_{{\mathbb R}_{\ge 0}^2}{\cal J}[\vp_{\vep}]{\rm d}^2F_{t}+ \int_{{\mathbb R}_{\ge 0}^3}{\cal K}[\vp_{\vep}]{\rm d}^3F_{t}
 \\
 &&\ge
\frac{b_0}{34}\vep^{3/2}
\Big(\int_{[\fr{1}{2}\vep,1]}y^{-\fr{1-\eta}{2}}{\rm d}F_{t}(y)\Big)^2- 2M_{1/2}(F_{t}) N_{0,2}(F_{t},\vep)
 \\
&&\ge
\frac{b_0}{34}\vep^{3/2}\inf_{s\in [\tau, \tau+T]}
\Big(\int_{[\fr{1}{2}\vep,1]}y^{-\fr{1-\eta}{2}}{\rm d}F_{s}(y)\Big)^2- 2c N_{0,2}(F_{t},\vep)
,\quad t\in [\tau, \tau+T]\eeas
and so
\beas&&
N_{0,2}(F_{\tau+t},\vep)
\ge
N_{0,2}(F_{\tau},\vep)e^{-2ct}+
\frac{b_0}{34}\vep^{3/2}\inf_{s\in [\tau, \tau+T]}
\Big(\int_{[\fr{1}{2}\vep,1]}y^{-\fr{1-\eta}{2}}{\rm d}F_{s}(y)\Big)^2\fr{1-e^{-2ct}}{2c}
\\
&&\ge
\frac{b_0}{34}\vep^{3/2}\inf_{s\in [0, T]}
\Big(\int_{[\fr{1}{2}\vep,1]}y^{-\fr{1-\eta}{2}}{\rm d}F_{\tau+s}(y)\Big)^2e^{-ct}
t,\quad t\in [0, T]
\eeas
where we used
$1-e^{-x}\ge x e^{-x/2}, x\ge 0.$ This proves the first inequality.

Next suppose that
$0<\vep\le 2/3, 0<\gm\le (24)^{-1/2}$. Then using Lemma \ref{lemma4.2} and Lemmas \ref{lemma4.3} we have
$$\frac{{\rm d}}{{\rm d}t}N_{0,2}(F_{t},\vep)
\ge -2cN_{0,2}(F_{t},\vep)+
\fr{b_0}{4}\fr{\gm^{2}}{\vep^{1-\eta}}\inf_{s\in[\tau, \tau+T]}\Big(
N_{0,2}(F_{s},\gm \vep)
\big(F_{s}\big(\big[\gamma \vep, \fr{3}{2}\vep\big]\big)\big)^2\Big)
$$ for all $t\in[\tau, \tau+T].$
As shown above we also obtain
$$ N_{0,2}(F_{\tau+t},\vep)\ge
\fr{b_0}{4}\fr{\gm^{2}}{\vep^{1-\eta}}
\inf_{s\in[0, T]}
\Big(N_{0,2}(F_{\tau+s},\gm \vep)\big(F_{\tau+s}\big(\big[\gamma \vep,\fr{3}{2}\vep\big]\big)\big)^2\Big) e^{-ct}t,\quad t\in[0, T].
$$

(II): By Lemma \ref{lemma4.2}, (\ref{KK}) and that $t\mapsto c^{ct}\underline{N}_{\alpha,2}(F_t,\vep)$ is non-decreasing, we have
\beas e^{ct}N_{0,2}(F_t,\vep)- e^{cs}N_{0,2}(F_s,\vep)\ge \fr{b_0}{2}e^{cs}\underline{N}_{\alpha, 2}(s,\vep)\int_{s}^{t}\big(A_{\beta,p}(F_{\tau},\vep)\big)^2 {\rm d}\tau,
\quad 0\le s<t.\eeas
Now for any $h>0$, letting $s$ and $t$ be replaced with $t$ and $t+h$ respectively and
noting that $ N_{0,2}(F_{t+h},\vep)\le N$, we obtain
$$e^{ch}N\ge \fr{b_0}{2}\underline{N}_{\alpha,2}(F_t,\vep)\int_{t}^{t+h}
\big(A_{\beta,p}(F_{\tau},\vep)\big)^2 {\rm d}\tau.$$
Then, using Cauchy-Schwarz inequality,
$$\fr{1}{h}\int_{t}^{t+h}A_{\beta, p}(F_\tau,\vep) {\rm d}\tau\le \Big(\fr{1}{h}\int_{t}^{t+h}
\big(A_{\beta, p}(F_\tau,\vep)\big)^2 {\rm d}\tau\Big)^{1/2}
\le\Big(\fr{2e^{c h} N}{hb_0\underline{N}_{\alpha, 2}(F_t,\vep)}\Big)^{1/2} .$$
Note that according to the convention (\ref{convention}), this inequality
still holds when $\underline{N}_{\alpha, 2}(F_t,\vep)=0$.
On the other hand, applying Lemma 2.3 in \cite{Lu2016} to the measure $F=F_{\tau}$ we have
$$ N_{0,p}(F_{\tau},\vep)\le F_{\tau}(\{0\})+\Big(\fr{p}{\beta}\vep^{\fr{\beta}{p}}\Big)^{p-1}
\int_{0}^{\vep}\vep_1^{-1+\fr{\beta}{p}} A_{\beta,p}(F_{\tau},\vep_1){\rm d}\vep_1
,\quad \vep>0.$$
Taking integration and using the fact that $0<\vep_1\le \vep\, \Longrightarrow\,\underline{N}_{\alpha,2}(F_t,\vep_1)\ge \underline{N}_{\alpha,2}(F_t,\vep)
$ we have
\beas&&
\fr{1}{h}\int_{t}^{t+h}N_{0,p}(F_{\tau},\vep){\rm d}\tau\le \fr{1}{h}\int_{t}^{t+h}F_{\tau}(\{0\}){\rm d}\tau+\Big(\fr{p}{\beta}\vep^{\fr{\beta}{p}}\Big)^{p-1}
\int_{0}^{\vep}\vep_1^{-1+\fr{\beta}{p}} \fr{1}{h}\int_{t}^{t+h}A_{\beta,p}(F_{\tau},\vep_1){\rm d}\tau{\rm d}\vep_1
\\
&&\le \fr{1}{h}\int_{t}^{t+h}F_{\tau}(\{0\}){\rm d}\tau+\Big(\fr{p}{\beta}\vep^{\fr{\beta}{p}}\Big)^{p-1}
\int_{0}^{\vep}\vep_1^{-1+\fr{\beta}{p}}\Big(\fr{2e^{c h} N}{hb_0\underline{N}_{\alpha, 2}(F_t,\vep_1)}\Big)^{1/2}{\rm d}\vep_1
\\
&&\le\fr{1}{h}
\int_{t}^{t+h}F_{\tau}(\{0\}){\rm d}\tau+
\Big(\fr{2e^{c h} N}{hb_0\underline{N}_{\alpha, 2}(F_t,\vep)}\Big)^{1/2}
\Big(\fr{p}{\beta}\Big)^p\vep^{\beta}.
\eeas
Since
$t\mapsto e^{ct}N_{0,p}(F_t,\vep), t\mapsto e^{ct}F_t(\{0\})$ are  non-decreasing on $[0,\infty)$,
it follows that
$ e^{-ch}N_{0,p}(F_{t},\vep)\le N_{0,p}(F_{\tau},\vep),
F_{\tau}(\{0\})\le e^{ch} F_{t+h}(\{0\})$ for all $\tau\in [t,t+h]$
and so
$$ e^{-ch} N_{0,p}(F_{t},\vep)
\le e^{ch}F_{t+h}(\{0\})+
\Big(\fr{2e^{c h} N}{hb_0\underline{N}_{\alpha, 2}(F_t,\vep)}\Big)^{1/2}
\Big(\fr{p}{\beta}\Big)^p\vep^{\beta}.$$
This gives (\ref{4.15})

Finally for the case $\alpha=0$, i.e. $\beta=(1-\eta)/2, p=3/2$,  we have $\underline{N}_{0, 2}(F_t,\vep)
=F_t(\{0\})$ and thus (\ref{4.16}) holds true.
\end{proof}
\\

The following lemma and proposition are key steps for obtaining lower bounds of
$F_t(\{0\})$ and for the convergence of $F_t(\{0\})$ to BEC without additional condition on the initial data.

\begin{lemma}\lb{lemma4.5} Let $B({\bf {\bf v-v}_*},\omega)$ satisfy Assumption \ref{assp} with $0\le\eta<1/4$.
Given any $N>0, E>0$.  Let $\alpha=\fr{1}{10}(1-4\eta)$,
\be A_{\alpha}^*=\Big(
\fr{b_0^2}{136}\big(\fr{2}{3}\big)^{1-\eta}\big(1-\big(\fr{2}{3}\big)^{\alpha/4}\big)^{8}
\Big)^{-1/3},\quad \vep_{\alpha}^*=\min\Big\{
\Big(\fr{\alpha \log(\fr{3}{2})}{16\sqrt{NE}}\Big)^{1/\alpha},\, \fr{2}{3}
\Big\}\lb{4.20}\ee
and let $F_t\in {\cal B}_{1}^{+}({\mR}_{\ge 0})$ be
a conservative measure-valued isotropic solution of Eq.(\ref{Equation1}) on $[0,\infty)$
with initial datum $F_0$ satisfying $N(F_0)=N, E(F_0)=E$. 
If $0<\vep\le \vep_{\alpha}^*, \tau\ge 0$ satisfy
\be N_{0,2}\big(F_{\tau},\fr{3}{2}\vep\big)\ge 
A_{\alpha}^*\vep^{\alpha} \lb{BEC4}\ee
then 
\be N_{0,2}(F_{\tau+2\vep^{\alpha}},\vep)\ge \big(\fr{2}{3}\big)^{\alpha}
N_{0,2}\big(F_{\tau},\fr{3}{2}\vep\big).\lb{4.21}\ee
\end{lemma}

\begin{proof}  Let $c=\sqrt{NE}$.
Since $t\mapsto e^{ct}N_{0,2}(F_t,\vep)$ is non-decreasing on $[0,\infty)$,
to prove (\ref{4.21}) it suffices to prove that
\be \exists\, s \in [\tau, \tau+2\vep^{\alpha}]\,\,\,
{\rm s.t.}\,\,\,
N_{0,2}(F_{s},\vep)\ge  e^{c2\vep^{\alpha}}\big(\fr{2}{3}\big)^{\alpha} N_{0,2}\big(F_{\tau},\fr{3}{2}\vep\big).
\lb{4.22}\ee
In fact if (\ref{4.22}) holds, then 
\beas&& e^{c(\tau+2\vep^{\alpha})}
N_{0,2}(F_{\tau+2\vep^{\alpha}},\vep)
\ge e^{c s}
N_{0,2}(F_{s},\vep)\ge  e^{c s}
e^{c2\vep^{\alpha}}\big(\fr{2}{3}\big)^{\alpha} N_{0,2}\big(F_{\tau},\fr{3}{2}\vep\big)
\\
&&= e^{c(s+2\vep^{\alpha})}\big(\fr{2}{3}\big)^{\alpha} N_{0,2}\big(F_{\tau},\fr{3}{2}\vep\big)
\ge e^{c(\tau+2\vep^{\alpha})}\big(\fr{2}{3}\big)^{\alpha} N_{0,2}\big(F_{\tau},\fr{3}{2}\vep\big)
\eeas
and so
$
N_{0,2}(F_{\tau+\vep^{\alpha}},\vep)\ge \big(\fr{2}{3}\big)^{\alpha} N_{0,2}\big(F_{\tau},\fr{3}{2}\vep\big).$

We use contradiction argument. Suppose that (\ref{4.22}) does not hold. Then we have
\be N_{0,2}(F_{\tau+t}, \vep)< e^{c2\vep^{\alpha}}\big(\fr{2}{3}\big)^{\alpha} N_{0,2}\big(F_{\tau},\fr{3}{2}\vep\big)
\qquad \forall\, t   \in [0,2{\vep}^{\alpha}].\lb{contradict}\ee
For convenience we denote
$$ M = N_{0,2}\big(F_{\tau},\fr{3}{2}\vep\big),\quad \gm =1-  \big(\fr{2}{3}\big)^{\alpha/4}.$$
We have $\gm<1-\big(\fr{2}{3}\big)^{1/8}<(24)^{-1/2}$  (because $\alpha<1/2$),  and
\be N_{0,2}(F_{\tau+t},\vep)\ge  \int_{[0,\gamma\vep]}\Big(1-\fr{x}{\vep}\Big)^2{\rm d}F_{\tau+t}(x)
\ge (1-\gamma)^{2} F_{\tau+t}([0,\gamma \vep]).\lb{4.25}\ee
Noting that
\bes&& 0\le t\le  2\vep^{\alpha}\le \frac {\alpha \log(\fr{3}{2})}{8c}\quad \Longrightarrow\quad
e^{c t}\ge  e^{c 2\vep^{\alpha}}\le
\big(\fr{3}{2}\big)^{\fr{\alpha}{8}}=(1-\gm)^{-\fr{1}{2}}
\nonumber \\
&&  \Longrightarrow\quad
e^{-c t}\ge  e^{-c 2\vep^{\alpha}}\ge
\big(\fr{2}{3}\big)^{\fr{\alpha}{8}}=(1-\gm)^{\fr{1}{2}} \lb{4.ee}\ees
and using (\ref{4.25}) and (\ref{contradict}) we have
\bes&& F_{\tau+t}\big(\big[\gamma \vep,\fr{3}{2}\vep\big]\big)=
F_{\tau+t}\big(\big[0,\fr{3}{2}\vep\big]\big)
  -F_{\tau+t}\big([0, \gamma \vep)\big)\nonumber \\
  &&\ge N_{0,2}\big(F_{\tau+t},\fr{3}{2}\vep\big)
  -\frac {1}{(1-\gamma)^{2}}N_{0,2}(F_{\tau+t},\vep)
\nonumber \\
 && \ge e^{-ct}N_{0,2}(F_{\tau},\fr{3}{2}\vep)
  -\frac {1}{(1-\gamma)^{2}}e^{c2\vep^{\alpha}}\big(\fr{2}{3}\big)^{\alpha} N_{0,2}\big(F_{\tau},\fr{3}{2}\vep\big)
  \nonumber \\
  &&
 \ge (1-\gm)^{\fr{1}{2}}M
  -\frac {1}{(1-\gamma)^{2}}(1-\gm)^{-\fr{1}{2}}(1-\gm)^4M=(1-\gm)^{\fr{1}{2}}\gm M
\quad \forall\, t\in[0,2\vep^{\alpha}].\lb{4.26}\ees
Combining this with the assumption $\fr{3}{2}\vep\le 1$ gives
\beas\Big(\int_{[\fr{1}{2}\gm\vep, 1]}y^{-\fr{1-\eta}{2}}{\rm d}F_{\tau+t}(y)\Big)^2
\ge\Big(\fr{3}{2}\vep\Big)^{-(1-\eta)}\Big(F_{\tau+t}
\big(\big[\fr{1}{2}\gamma\vep, \fr{3}{2}\vep\big]\big)\Big)^2 \ge
\Big(\frac{2}{3\vep }\Big)^{1-\eta}[(1-\gm)^{\fr{1}{2}}\gm M]^{2}\eeas
for all $t\in [0,2\vep^{\alpha}].$
Inserting this into (\ref{inf2}) in Lemma \ref{lemma4.4}  we obtain
\be N_{0,2}(F_{\tau+t},\gm\vep)
 \ge \fr{b_0}{34}\big(\fr{2}{3}\big)^{1-\eta}
(1-\gm){\gamma}^{\fr{3}{2}+2}M^{2}\vep^{\fr{1}{2}+\eta} te^{-ct}\qquad \forall\, t\in [0, 2\vep^{\alpha}]. \lb{4.40}\ee
Using (\ref{4.26}),(\ref{4.40}) and (\ref{4.ee})
to $t=\vep^{\alpha}+s, s\in [0,\vep^{\alpha}]$, we have
\beas&&
N_{0,2}(F_{\tau+\vep^{\alpha}+s},\gm \vep)
\Big(F_{\tau+\vep^{\alpha}+s}\big(\big[\gamma \vep,\fr{3}{2}\vep\big])\Big)^2
\\
&&\ge\fr{b_0}{34}\big(\fr{2}{3}\big)^{1-\eta}
(1-\gm){\gamma}^{\fr{3}{2}+2}M^{2}\vep^{\fr{1}{2}+\eta}(\vep^{\alpha}+ s)
e^{-c(\vep^{\alpha}+ s)}\big((1-\gm)^{\fr{1}{2}}\gm M\big)^2
\\
&&\ge \fr{b_0}{34}\big(\fr{2}{3}\big)^{1-\eta}
(1-\gm)^{\fr{5}{2}}{\gamma}^{\fr{3}{2}+4}M^{4}\vep^{\fr{1}{2}+\eta+\alpha}\qquad \forall\, s\in [0,\vep^{\alpha}]
\quad ({\rm because}\,e^{-c 2\vep^{\alpha}}\ge
(1-\gm)^{\fr{1}{2}} ).\eeas
Then using (\ref{inf3}) in Lemma \ref{lemma4.4} where $\tau$ is replaced by $\tau+\vep^{\alpha}$
and taking $t=T=\vep^{\alpha}$, we compute
\beas&& N_{0,2}(F_{\tau+2\vep^{\alpha}},\vep)\ge
\fr{b_0}{4}\fr{\gm^{2}}{\vep^{1-\eta}} e^{-c\vep^{\alpha}}\vep^{\alpha} \inf_{s\in[0,\vep^{\alpha}]}
\Big(N_{0,2}(F_{\tau+\vep^{\alpha}+s},\gm \vep)\Big(F_{\tau+\vep^{\alpha}+s}
\big(\big[\fr{1}{2}\gamma\vep, \fr{3}{2}\vep\big]\big)\Big)^2\Big)\\
&&\ge
\fr{b_0}{4}\fr{\gm^{2}}{\vep^{1-\eta}} e^{-c\vep^{\alpha}}\vep^{\alpha}
\fr{b_0}{34}\big(\fr{2}{3}\big)^{1-\eta}
(1-\gm)^{\fr{5}{2}}{\gamma}^{\fr{3}{2}+4}M^{4}\vep^{\fr{1}{2}+\eta+\alpha}
\\
&&
=
\fr{b_0}{4}\fr{\gm^{2}}{\vep^{1-\eta}} e^{-3c\vep^{\alpha}}\vep^{\alpha}
\fr{b_0}{34}\big(\fr{2}{3}\big)^{1-\eta}
(1-\gm)^{\fr{5}{2}}{\gamma}^{\fr{3}{2}+4}M^{4}\vep^{\fr{1}{2}+\eta+\alpha}e^{2c\vep^{\alpha}}
\\
&&
\ge
\fr{b_0}{4}\fr{\gm^{2}}{\vep^{1-\eta}} (1-\gm)^{\fr{3}{4}}\vep^{\alpha}
\fr{b_0}{34}\big(\fr{2}{3}\big)^{1-\eta}
(1-\gm)^{\fr{5}{2}}{\gamma}^{\fr{3}{2}+4}M^{4}\vep^{\fr{1}{2}+\eta+\alpha}e^{2c\vep^{\alpha}}
\\
&&=\fr{b_0^2}{136}\big(\fr{2}{3}\big)^{1-\eta}(1-\gm)^{2.75}{\gamma}^{7.5}
\fr{1}{\vep^{\fr{1}{2}-2\eta-2\alpha}}M^{4}e^{2c\vep^{\alpha}}\\
&&=\fr{b_0^2}{136}\big(\fr{2}{3}\big)^{1-\eta-\alpha }(1-\gm)^{3.25}{\gamma}^{7.5}
\Big(\fr{M}{\vep^\alpha}\Big)^3 \big(\fr{2}{3}\big)^{\alpha} e^{2c\vep^{\alpha}}M
\quad ({\rm because}, \, \fr{1}{2}-2\eta-2\alpha=3\alpha)\\
&&\ge \fr{b_0^2}{136}\big(\fr{2}{3}\big)^{1-\eta-\alpha }(1-\gm)^{3.25}{\gamma}^{7.5}
(A_{\alpha}^*)^3 \big(\fr{2}{3}\big)^{\alpha} e^{2c\vep^{\alpha}}M
\\
&&>\fr{b_0^2}{136}\big(\fr{2}{3}\big)^{1-\eta-\alpha }(1-\gm)^{4}{\gamma}^{8}
(A_{\alpha}^*)^3 \big(\fr{2}{3}\big)^{\alpha} e^{2c\vep^{\alpha}}M
\\
&&=\fr{b_0^2}{136}\big(\fr{2}{3}\big)^{1-\eta-\alpha }\big(\fr{2}{3}\big)^{\alpha}\Big(1-\big(\fr{2}{3}\big)^{\fr{\alpha}{4}}\Big)^{8}
(A_{\alpha}^*)^3 \big(\fr{2}{3}\big)^{\alpha} e^{2c\vep^{\alpha}}M
\\
&&=\fr{b_0^2}{136}\big(\fr{2}{3}\big)^{1-\eta}\Big(1-\big(\fr{2}{3}\big)^{\fr{\alpha}{4}}\Big)^{8}
(A_{\alpha}^*)^3 \big(\fr{2}{3}\big)^{\alpha} e^{2c\vep^{\alpha}}M
=\big(\fr{2}{3}\big)^{\alpha} e^{2c\vep^{\alpha}}M\eeas
This contradicts (\ref{contradict}).
 Thus (\ref{4.21}) holds true.
\end{proof}
\vskip2mm

\begin{proposition}\lb{proposition4.6} Let $B({\bf {\bf v-v}_*},\omega)$ satisfy the Assumption \ref{assp} with $0\le\eta<1/4$, and let $\alpha=\fr{1}{10}(1-4\eta)$, $A_{\alpha}^*$ be given in (\ref{4.20}). 
Given any $N>0, E>0$. Let
\be \vep_{\alpha}^{\sharp}= \min\Big\{\Big(\Big(1-\big(\fr{2}{3}\big)^{\alpha}\Big)\fr{\alpha \log(\fr{3}{2})}{8\sqrt{NE}}\Big)^{1/\alpha},\, 
\Big(\fr{\sqrt{b_0}(A_{\alpha}^*)^{3/2}}{87 N^{3/4}E^{1/4}}
\Big)^{\fr{2}{1-3\alpha-\eta}},\,\fr{2}{3}\Big\}
\lb{ABE}\ee
and 
let $F_t\in {\cal B}_{1}^{+}({\mR}_{\ge 0})$ be
a conservative measure-valued isotropic solution of Eq.(\ref{Equation1}) on $[0,\infty)$ with initial datum $F_0$ satisfying $N(F_0)=N, E(F_0)=E$. If $\tau\ge 0$ and $0<\vep\le \vep_{\alpha}^{\sharp}$ satisfy
\be
N_{0,2}\big(F_{\tau}, \fr{3}{2}\vep\big)\ge A_{\alpha}^* \vep^{\alpha}\lb{BEC1}\ee
then, at the time $t_{\vep}=\tau+2 \big(1-\big(\fr{2}{3}\big)^{\alpha}\big)^{-1}{\vep}^{\alpha}+\fr{1}{3\sqrt{NE}}$, it holds
\be F_{t_{\vep}}(\{0\})>\fr{1}{5}A_{\alpha}^*\vep^{\alpha}\lb{lbdBEC} \ee
hence
\be F_{t}(\{0\})\ge e^{-c(t-t_{\vep})}F_{t_{\vep}}(\{0\})>0\qquad \forall\, t\ge t_{\vep}\lb{BEC6}\ee
where $c=\sqrt{NE}$.
\end{proposition}

 \begin{proof}  First of all we have $\vep_{\alpha}^{\sharp}\le \vep_{\alpha}^*$ where 
 $\vep_{\alpha}^*$ is given in (\ref{4.20}).
 
 {\bf Step1.} Let $h_{\vep}^*=2\big(1-\big(\fr{2}{3}\big)^{\alpha}\big)^{-1}{\vep}^{\alpha}$. We prove that
\be\underline N_{\alpha,2}(F_{\tau+h_{\vep}^*},\vep)\ge \fr{0.9}{{\vep}^{\alpha}}
N_{0,2}\big(F_{\tau},\fr{3}{2}\vep\big).\lb{BEC3}\ee
We will use Lemma \ref{lemma4.5} with an iteration argument. Let
$$ \vep_k= \big(\fr{2}{3}\big)^{k}\vep,\quad
 h_n=2\sum\limits_{k = 0}^n (\vep_k)^{\alpha}=2\sum\limits_{k = 0}^n \big(\fr{2}{3}\big)^{k\alpha}{\vep}^{\alpha},\quad k,n=0,1,2,....\,.$$
and let $M = N_{0,2}\big(F_{\tau},\fr{3}{2}\vep\big)$.
We first prove that for all $n=0,1,2,...$,
\be N_{0,2}(F_{\tau+h_n}, \vep_n)\ge \big(\fr{2}{3}\big)^{\alpha(n+1)}M.
\lb{4.31}\ee
From (\ref{4.21}) in Lemma \ref{lemma4.5} we see that (\ref{4.31}) holds for $n=0$.
Suppose that (\ref{4.31}) holds for some $n\in{\mN}\cup\{0\}$.
Then (using $\fr{3}{2}\vep_{n+1}=\vep_n$)
\beas&& N_{0,2}\big(F_{\tau+h_n}, \fr{3}{2}\vep_{n+1}\big)=N_{0,2}(F_{\tau+h_n},\vep_{n})\ge
\big(\fr{2}{3}\big)^{\alpha(n+1)}  M
\ge \big(\fr{2}{3}\big)^{\alpha(n+1)}  A_{\alpha}^* \vep^{\alpha}
\\
&&= A_{\alpha}^*\big(\big(\fr{2}{3}\big)^{\alpha(n+1)}\vep^{\alpha}
\big)=A_{\alpha}^*\big(\big(\fr{2}{3}\big)^{n+1}\vep\big)^{\alpha}
=A_{\alpha}^*(\vep_{n+1})^{\alpha}
\\
&&
\mbox{and}\quad \vep_{n+1}=\big(\fr{2}{3}\big)^{n+1}\vep<\vep\le \vep_{\alpha}^{\sharp}\le \vep_{\alpha}^*.\eeas
Then using Lemma \ref{lemma4.5}
and noting that $h_{n+1}=h_n+2(\vep_{n+1})^{\alpha}$ we have (with the inductive hypotheses) 
\beas&& N_{0,2}(F_{\tau+h_{n+1}}, \vep_{n+1})=N_{0,2}(F_{\tau+h_n+2(\vep_{n+1})^{\alpha}}, \vep_{n+1})\\
&&\ge \big(\fr{2}{3}\big)^{\alpha}N_{0,2}\big(F_{\tau+h_n}, \fr{3}{2}\vep_{n+1}\big)={\big(\fr{2}{3}\big)^{\alpha}}N_{0,2}(F_{\tau+h_n},\vep_{n})
\\
&&\ge  \big(\fr{2}{3}\big)^{\alpha}\big(\fr{2}{3}\big)^{\alpha(n+1)}{M}
=\big(\fr{2}{3}\big)^{\alpha(n+2)}{M}.\eeas
Therefore (\ref{4.31}) holds also for $n+1$ and thus, by induction,
(\ref{4.31}) holds for all $n\in{\mN}\cup\{0\}$.

By definition of $\vep_{\alpha}^{\sharp}$ in (\ref{ABE}) we have
 $$h_n\le h_{\vep}^*=2\Big(1-\big(\fr{2}{3}\big)^{\alpha}\Big)^{-1}{\vep}^{\alpha}\le \fr{\alpha \log(\fr{3}{2})}{4c},\quad  e^{-ch_{\vep}^*}\ge \big(\fr{2}{3}\big)^{\alpha/4}.$$Using the non-decrease of
$t\mapsto e^{ct}N_{0,2}(F_t,\vep)$ and (\ref{4.31}) we have
\beas&& N_{0,2}(F_{\tau+h_{\vep}^*}, \big(\fr{2}{3}\big)^{n}\vep)\ge
e^{-c(h_{\vep}^*-h_n)}N_{0,2}\big(F_{\tau+h_n}, \big(\fr{2}{3}\big)^{n}\vep\big)\nonumber\\
&&
\ge e^{-c(h_{\vep}^*-h_n)}{\big(\fr{2}{3}\big)^{\alpha(n+1)}}{M}\ge e^{-ch_{\vep}^*}{\big(\fr{2}{3}\big)^{\alpha(n+1)}}{M}
,\quad n=0,1,2,...\,.\eeas
Now for any $0<\dt\le \vep$, there is  $n\in {\mN}$
 such that $\big(\fr{2}{3}\big)^{n}\vep<\dt \le \big(\fr{2}{3}\big)^{n-1}\vep$,  so we have
\beas&& \fr{1}{{\dt}^{\alpha}}N_{0,2}(F_{\tau+h_{\vep}^*},\dt)\ge \fr{1}{{\dt}^{\alpha}} N_{0,2}\big(F_{\tau+h_{\vep}^*},\big(\fr{2}{3}\big)^{n}\vep\big)\\
&&  \ge
\fr{1}{\big(\big(\fr{2}{3}\big)^{n-1}\vep\big)^{\alpha}}
 e^{-ch_{\vep}^*} \big(\fr{2}{3}\big)^{\alpha(n+1)}M\ge \big(\fr{2}{3}\big)^{(2+\fr{1}{4})\alpha}\fr{M}
 {{\vep}^{\alpha}}>
 0.9\fr {M}{{\vep}^{\alpha}}
\eeas
where the last inequality is due to $0<\alpha\le \fr{1}{10}$.
Thus
$$\underline N_{\alpha, 2}(F_{\tau+h_{\vep}^*},\vep)=
\inf_{0<\dt\le \vep}\fr {N_{0,2}(F_{\tau+h_{\vep}^*},\dt)}{{\dt}^{\alpha}}
\ge  0.9\fr {M}{{\vep}^{\alpha}}$$
i.e. (\ref{BEC3}) holds true.

{\bf Step2.}
Let $\tau_{\vep}=\tau+h_{\vep}^*.$
Using Lemma \ref{lemma4.4} (recall there $p=\fr{3}{2}+\alpha,\beta=\fr{1-\alpha-\eta}{2}$)
with $h=\fr{1}{3c}=\fr{1}{3\sqrt{NE}}, t=\tau_{\vep}$, and using $p<2$ and the inequality (\ref{BEC3}) to deduce
$$\fr{1}{\vep^{\alpha}}N_{0,p}(F_{\tau_{\vep}},\vep)
\ge \fr{1}{\vep^{\alpha}}N_{0,2}(F_{\tau_{\vep}},\vep)
\ge \underline N_{\alpha,2}(F_{\tau_{\vep}},\vep)\ge
\fr{0.9}{\vep^{\alpha}}N_{0,2}(F_{\tau}, \fr{3}{2}\vep)\ge 0.9 A_{\alpha}^*
$$
we obtain
\beas&& e^{2/3}F_{\tau_{\vep}+h}(\{0\})
\ge N_{0,p}(F_{\tau_{\vep}},\vep)
 -\Big(\fr{2e^{3c h} N}{hb_0\underline{N}_{\alpha, 2}(F_{\tau_{\vep}},\vep)
}\Big)^{1/2}\Big(\fr{p}{\beta}\Big)^{p}\vep^{\beta}
\\
&&\ge  0.9 A_{\alpha}^*\vep^{\alpha}
-\Big(\fr{6e N^{3/2}E^{1/2}}{b_00.9
A_{\alpha}^*}\Big)^{1/2}
\Big(\fr{3+2\alpha}{1-\alpha-\eta}\Big)^{\fr{3}{2}+\alpha}
\vep^{\beta}\ge \fr{0.9}{2} A_{\alpha}^*\vep^{\alpha}\eeas
where for the last inequality we used $\beta-\alpha=\fr{1-3\alpha-\eta}{2}>0, \vep^{\beta-\alpha}\le 
(\vep_{\alpha}^{\sharp})^{\beta-\alpha}\le \fr{\sqrt{b_0}(A_{\alpha}^*)^{3/2}}{87 N^{3/4}E^{1/4}}$, and
$$\Big(\fr{6e}{0.9}\Big)^{1/2}
\Big(\fr{3+2\alpha}{1-\alpha-\eta}\Big)^{\fr{3}{2}+\alpha}
\fr{1}{87}\le\Big(\fr{6e}{0.9}\Big)^{1/2}
\fr{4^{8/5}}{87}<\fr{0.9}{2}.$$
Thus, at the time $t_{\vep}:=\tau_{\vep}+h=\tau+2 \big(1-\big(\fr{2}{3}\big)^{\alpha}\big)^{-1}{\vep}^{\alpha}+\fr{1}{3\sqrt{NE}}$,  we obtain
$$F_{t_{\vep}}(\{0\})=F_{\tau_{\vep}+h}(\{0\})\ge \fr{0.9}{2 e^{2/3}} A_{\alpha}^*\vep^{\alpha}
 >\fr{1}{5}A_{\alpha}^*\vep^{\alpha}.$$
The inequality (\ref{BEC6}) follows from this and the non-decrease of
 $t\mapsto F_{t}(\{0\}) e^{ct}$ on $[0,\infty)$. This completes the proof.
 \end{proof}
\vskip2mm

\begin{remark}\lb{remark4.7}{\rm In comparison with the previous results in \cite{EV2}, \cite{Lu2013} on
the occurrence of condensation (i.e. $F_t(\{0\})>0$) in finite time,
Proposition \ref{proposition4.6} not only provides a simple condition (\ref{BEC1}) with $\tau=0$ for the initial data $F_0$, but
also gives an explicit and  useful lower bound (\ref{lbdBEC}) for condensation.}
\end{remark}

\begin{remark}\lb{remark4.7*}{\rm 
Since  $N_{0,2}\big(F_{\tau},\fr{3}{2}\vep)\le N$, a necessary for $\vep$ satisfying 
(\ref{BEC4}) is that $\vep\le 
\big(\fr{N}{A_{\alpha}^*}\big)^{1/\alpha}$. In applications of Proposition \ref{proposition4.6} the number $\vep$ will be chosen much less than
$\big(\fr{N}{A_{\alpha}^*}\big)^{1/\alpha}$. See the following example and the proof of Theorem \ref{theorem4.8}.}
\end{remark}
\vskip2mm

\noindent {\bf Example of bounded initial data.} Here we show that for any $N>0, E>0$,
there are many bounded and smooth initial data that satisfy the condition in
Proposition \ref{proposition4.6} for $\tau=0$.  Given any $N>0, E>0$. Let
$\alpha, A_{\alpha}^*, \vep_{\alpha}^{\sharp}$ be given in
Proposition \ref{proposition4.6}, let
\beas&& 0<\vep\le \left\{\vep_{\alpha}^{\sharp},\,
\fr{E}{2N},\,\Big(\fr{N}{27 A_{\alpha}^*} \Big)^{1/\alpha}
\right\},\\
&&\dt=3\fr{E}{N},\quad
a=\fr{2E}{\vep(\dt-\vep)},\quad b=\fr{2N(\fr{2E}{N}-\vep)}{\dt(\dt-\vep)},\\
&& g_0(x)=a{\bf 1}_{[\fr{1}{4}\vep,\fr{3}{4}\vep]}(x)
  +b{\bf 1}_{[\fr{1}{4}\dt,\fr{3}{4}\dt]}(x),\quad {\rm d}G_0(x)=g_0(x){\rm d}x.\eeas
We compute (notice that $\vep\le \fr{1}{6}\dt$)
\beas&& N(G_0)=\fr{1}{2}(a\vep+b\dt)=N,\quad  E(G_0)=\fr{1}{4}(a{\vep}^{2}+b\dt^{2})=E,\\
&&
G_0([0,\vep])=\fr{1}{2}a\vep=\fr{E}{3\fr{E}{N}-\vep}>\fr{N}{3}\ge 9 A_{\alpha}^*\vep^{\alpha}. \eeas
Let
$J(x) = c_1e^{-\fr{1}{1-x^2}}{\bf 1}_{(-1,1)}(x)$ where $c_1>0$ is
such that $\int_{{\mR}}J(x){\rm d}x=1$. Let $J_{\ld}(x)=\fr{1}{\ld}J(\fr{x}{\ld})\,(\ld>0)$,
choose $\ld=\fr{1}{8}\vep$, and consider $f_0(x)=\fr{1}{\sqrt{x}}(J_{\ld}*g_0)(x)$ (convolution).
It is easily seen that $0\le f_0\in C^{\infty}_c({\mR})$ and ${\rm supp}f_0\subset [\fr{1}{4}\vep-\ld, \fr{3}{4}\dt+\ld]=[\fr{1}{8}\vep, \fr{3}{4}\dt+\fr{1}{8}\vep]$.
Let $F_0\in {\cal B}_{1}^{+}({\mR}_{\ge 0})$ be
defined by ${\rm d}F_0(x)=f_0(x)\sqrt{x}\,{\rm d}x$.
By simple calculation (using $\dt\ge 6\vep$) we have
$$N(F_0)=N(G_0)=N,\quad E(F_0)=E(G_0)=E,\quad F_0([0,\vep])=G_0([0,\vep])\ge 9 A_{\alpha}^*\vep^{\alpha}.$$
Thus (since $(1-\fr{x}{\fr{3}{2}\vep})^2\ge \fr{1}{9} $ for all $x\in [0,\vep]$ )
$$ N_{0,2}(F_0,\fr{3}{2}\vep)\ge \fr{1}{9}F_0([0,\vep])\ge
A_{\alpha}^*\vep^{\alpha}.$$
So $F_0$ (with the number $\vep$) satisfies the condition (\ref{BEC1}) in Proposition \ref{proposition4.6} for $\tau=0$.
$\hfill\Box$
\vskip2mm

\begin{theorem}\lb{theorem4.8}
Let $B({\bf {\bf v-v}_*},\omega)$ satisfy Assumption \ref{assp} with $0\le\eta<1/4$, and
let $F_0\in {\mathcal B}_1^{+}({\mathbb R}_{\ge 0})$ with
$N:=N(F_0)>0, E:=E(F_0)>0$ satisfy the low temperature condition
$\overline{T}/\overline{T}_c=2.2720\frac{E}{N^{5/3}}<1$. 
Let $\fr{1}{20}<\ld<\fr{1}{19}$ and let $F_t\in {\cal B}_{1}^{+}({\mR}_{\ge 0})$
with the initial datum $F_0$  be a conservative  measure-valued
isotropic solution of Eq.(\ref{Equation1}) on $[0, \infty)$ obtained in Theorem \ref{theorem3.2}.
Then, for any  we have
\beas \big|F_t(\{0\})-(1-(\overline{T}/\overline{T}_c)^{3/5})N\big|\le  C(1+t)^{-\fr{1-\eta}{2(4-\eta)}\ld}\qquad \forall\, t\ge 0\eeas
where the constant $C>0$ depends only on $N, E, b_0,\eta$ and $\ld$.
 \end{theorem}

\begin{proof}  Let $F_{\rm be}$ be
the unique Bose-Einstein distribution having the mass and energy $N, E$, let
$N_0=F_{\rm be}(\{0\})=(1-(\overline{T}/\overline{T}_c)^{3/5})N$.
From Lemma \ref{lemma3.1}, Theorem\ref{theorem3.2} we have
\be R(t):=\sup_{\tau\ge t}\|F_{\tau}-F_{\rm be}\|_{1}^{\circ}\le C_1\sup_{\tau\ge t}\sqrt{S(F_{\rm be})-S(F_{\tau})}\le
C_2(1+t)^{-\ld/2}\qquad \forall\, t\ge 0\lb{RRR}\ee
where here and below $C>0, C_i>0\,(i=1,2,...,8), \tau_0>0$ and $t_0>1$ are finite constants that depend only on  $N, E, b_0,\eta$ and $\ld$.

{\bf Step1.} Let  $\alpha, A_{\alpha}^*,\vep_{\alpha}^{\sharp}$ be given in
Proposition \ref{proposition4.6}, let
\beas&& \vep_0=\min\left\{\vep_{\alpha}^{\sharp},\,
\Big(\fr{N_0}{2A_{\alpha}^*}\Big)^{1/\alpha}
\right\},\quad k_0=2 \Big(1-\big(\fr{2}{3}\big)^{\alpha}\Big)^{-1}{\vep_0}^{\alpha}+\fr{1}{3\sqrt{NE}}.\eeas
We prove that there is  $\tau_0>0$
such that
\be F_{t}(\{0\})\ge \fr{1}{5}A_{\alpha}^*\vep_0^{\alpha}\qquad \forall\, t\ge \tau_0+k_0.
\lb{lbdF}\ee
For any $\vep>0, p\ge 1$, using the conservation of mass and $0\le 1- [(1-x/\vep)_{+}]^p
\le \fr{p}{\vep} x$ we have
\be|N_{0,p}(F_t,\vep)-N_{0,p}(F_{\rm be},\vep)|
=\Big|\int_{{\mR}_{\ge 0}}(1-[(1-x/\vep)_{+}]^p){\rm d}(F_t-F_{\rm be})(x)
\Big|\le \fr{p}{\vep}\|F_t-F_{\rm be}\|_{1}^{\circ}.
\lb{4.49}\ee
From this and $N_{0,p}(F_{\rm be},\vep)\ge F_{\rm be}(\{0\})=N_0$ we obtain
\be N_{0,p}(F_t,\vep)\ge
N_{0,p}(F_{\rm be},\vep)-\fr{p}{\vep}\|F_t-F_{\rm be}\|_{1}^{\circ}\ge N_0-\fr{p}{\vep}\|F_t-F_{\rm be}\|_{1}^{\circ}.\lb{NpN}\ee
Since $\|F_t-F_{\rm be}\|_{1}^{\circ}\le C_2(1+t)^{-\ld/2}$,
there is $\tau_0> 0$   such that
$$\fr{2}{\vep_0}\|F_t-F_{\rm be}\|_{1}^{\circ}<\fr{N_0}{2}\quad \forall\, t\ge \tau_0
\quad {\rm hence}\quad
 N_{0,p}(F_t,\vep_0)\ge \fr{1}{2}N_0
\quad \forall\, t\ge \tau_0,\,\,\forall\, 1\le p\le 2. $$
Note that $\vep\mapsto N_{0,2}(F_t,\vep)$
is non-decreasing. We then deduce from the choice of $\vep_0$ that
\beas&& N_{0,2}(F_t,\fr{3}{2}\vep_0)
\ge N_{0,2}(F_t,\vep_0)
\ge \fr{1}{2}N_0
\ge A_{\alpha}^*\vep_0^{\alpha}\qquad \forall\, t\ge \tau_0.\eeas
Thus by Proposition \ref{proposition4.6} we conclude
$F_{t+k_0}(\{0\})\ge \fr{1}{5}A_{\alpha}^*\vep_0^{\alpha}$. Since this holds for all $t\ge \tau_0$,
we obtain (\ref{lbdF}).

{\bf Step2.} According to (\ref{RRR}),
we need only prove that there is $t_0>1$ such that
\be |F_{t}(\{0\})-N_0|\le C\big(R(t-1)\big)^{\fr{1-\eta}{4-\eta}}\qquad \forall\, t\ge t_0.
\lb{4.38}\ee
Using (\ref{4.16}) (see Lemma \ref{lemma4.4}), (\ref{NpN}) for $p=3/2$, the lower bound (\ref{lbdF}), and denoting $c=\sqrt{NE}$,
$C_{3}=\fr{1}{5}A_{\alpha}^*\vep_0^{\alpha}$, we have for all $0<h\le 1, 0<\vep\le 1$
\bes F_{t}(\{0\})&\ge& e^{-2ch}N_{0,3/2}(F_{t-h},\vep)
-e^{-2ch}\Big(\fr{2e^{3c h} N}{hb_0F_{t-h}(\{0\})}\Big)^{1/2}
\Big(\fr{3}{1-\eta}\Big)^{3/2}\vep^{\fr{1-\eta}{2}}\nonumber\\
&\ge& (1-2ch)
N_0- \fr{3}{2\vep}\|F_{t-h}-F_{\rm be}\|_{1}^{\circ}
-\Big(\fr{2N}{hb_0C_{3} }\Big)^{1/2}
\Big(\fr{3}{1-\eta}\Big)^{3/2}\vep^{\fr{1-\eta}{2}}\nonumber\\
&\ge & (1-2ch)N_0 - \fr{3}{2\vep}R(t-h) -C_{4} h^{-\fr{1}{2}}\vep^{\fr{1-\eta}{2}}
\quad \forall\, t\ge \tau_0+k_0+h.
\lb{4.54}\ees
Here we used
$e^{-2ch}\ge 1-2ch$, $ e^{-2ch} e^{3ch/2}<1$. Now we consider
\beas 0<\vep\le C_5:=\min\Big\{1,\, \Big(\fr{2cN_0}{C_4 }\Big)^{\fr{2}{1-\eta}}\Big\},\quad
h=h_{\vep}=\Big(\fr{C_4 }{2cN_0}\Big)^{2/3}
\vep^{\fr{1-\eta}{3}}.\eeas
For all $0<\vep\le C_5$ we have $0<h_{\vep}\le 1$ and from (\ref{4.54}) we obtain
\be
F_{t}(\{0\})\ge
N_0-\fr{3}{2\vep}R(t-1)-C_6
\vep^{\fr{1-\eta}{3}}\qquad \forall\,t\ge \tau_0+k_0+1.\lb{FR}\ee
Since $R(t-1)\le C_2 t^{-\ld/2}$, we choose $t_0\ge \tau_0+k_0+1$  large enough
such that
$\big(R(t-1)\big)^{\fr{3}{4-\eta}}<C_5$ for all $t\ge t_0$.
Then for every $t\ge t_0$, taking $\vep=\big(R(t-1)\big)^{\fr{3}{4-\eta}}$ we obtain from (\ref{FR}) that
\be
F_{t}(\{0\})\ge N_0-C_7\big(R(t-1)\big)^{\fr{1-\eta}{4-\eta}}\qquad \forall\, t\ge t_0.\lb{4.41}\ee
On the other hand using the inequality (\ref{4.49}) for $p=1$ we have
\bes&&
F_{t}(\{0\})\le N_{0,1}(F_t,\vep)
\le N_{0,1}(F_{{\rm be}},\vep)+\fr{1}{\vep}\|F_t-F_{{\rm be}}\|_1^{\circ}
\nonumber
\\
&&=F_{\rm be}(\{0\})+
\int_{(0,\vep]}\Big(1-\fr{x}{\vep}\Big)
\fr{1}{e^{x/\kappa}-1}\sqrt{x}\,{\rm d}x +\fr{1}{\vep}\|F_t-F_{{\rm be}}\|_1^{\circ}
\nonumber \\
&&\le N_0+2\kappa \sqrt{\vep}+
\fr{1}{\vep}R(t)
\qquad \forall\,t\ge 0,\quad \forall\, \vep>0\lb{4.59}
.\ees
Minimizing the right hand side of (\ref{4.59}) with respect to $\vep\in(0,\infty)$
gives
\be F_{t}(\{0\})\le N_0+C_8\big(R(t)\big)^{1/3}\le N_0+C_8\big(R(t-1)\big)^{\fr{1-\eta}{4-\eta}}\qquad
\forall\, t\ge t_0.\lb{4.42}\ee
Combining (\ref{4.41}) and (\ref{4.42})
we obtain (\ref{4.38}) with the constant $C=\max\{C_7, C_8\}$.
\end{proof}
\vskip2mm

Finally at the end of this section we finish the proof of Theorem \ref{theorem1.2}:

{\bf Proof of Theorem \ref{theorem1.2}.}  We need only prove the
algebraic decay rate of $\|F-F_{\rm be}\|_{1}$  since the algebraic decay rate of
$S(F_{\rm be})-S(F_t)$ has been proven by Theorem \ref{theorem3.2}.
Let
$C_i\,(i=1,2,...,8)$ denote finite positive constants that depend only on $N,E,b_0,\eta$ and $\ld$.   Since
$F_t$ conserves the mass and energy, it follows from
Lemma \ref{lemma3.1} and Theorem \ref{theorem3.2} that
\be \|F-F_{\rm be}\|_{1}^{\circ}\le C_1\big(S(F_{\rm be})-S(F_t)\big)^{1/2}\le C_2(1+t)^{-\ld/2}\qquad \forall\, t\ge 0.\lb{4.Final}\ee
Now if $\overline{T}/\overline{T}_c< 1$, then
by Lemma \ref{lemma3.1}, Theorem \ref{theorem4.8}, and (\ref{4.Final}) we have
\beas&&
\|F_t-F_{\rm be}\|_1\le  2|F_t(\{0\})-F_{\rm be}(\{0\})|
+C_{3}(\|F_t-F_{\rm be}\|_{1}^{\circ})^{1/3}\\
&&\le C_4(1+t)^{-\fr{1-\eta}{2(4-\eta)}\ld}+
C_5(1+t)^{-\ld/6}\le C_6(1+t)^{-\fr{1-\eta}{2(4-\eta)}\ld}\qquad \forall\, t\ge 0\eeas
while if
$\overline{T}/\overline{T}_c\ge 1$, then we have from Lemma \ref{lemma3.1} and  (\ref{4.Final})
that
$$\|F_t-F_{\rm be}\|_1\le C_7(\|F-F_{\rm be}\|_{1}^{\circ})^{1/3}\le C_8(1+t)^{-\ld/6}\le C_8(1+t)^{-\fr{1-\eta}{2(4-\eta)}\ld}\qquad \forall\, t\ge 0.$$
This completes the proof of Theorem \ref{theorem1.2}.
$\hfill\Box$

\begin{center}\section{Appendix}\end{center}

Here we prove some properties that have been used in the previous sections.

\noindent{\bf 6.1. Some integral equalities.}  We will use the  following integral formula:

(1) (Carleman Representation).  Let $\Psi$ be a Borel measurable function on ${\bRR}$ and it is nonnegative or satisfies some integrability such that the following integrals make sense)
\be \int_{{\mR}^3\times{\mS}^2}\Psi({\bf v}',{\bf v}_*'){\rm d}\og {\rm d}{\bf v}_*=2\int_{{\mR}^3}
\fr{{\rm d}{\bf x}}{|{\bf x}|^2}\int_{{\mR}^2
({\bf x})}\Psi({\bf {\bf v}-x},{\bf {\bf v}-y}){\rm d}{\bf y}\lb{5.1}\ee
for almost all ${\bf v}\in{\mR}^3$, where $({\bf v}',{\bf v}_*')$ is given in the
$\og$-representation (\ref{colli}) and ${\rm d}{\bf y}$ in $\int_{{\mR}^2
({\bf x})}\{\cdots\} {\rm d}{\bf y}$
is the Lebesgue measure element on the plan
${\mR}^2({\bf x})=\{{\bf y}\in {\mR}^3\,|\, {\bf y}\bot {\bf x}\}$.

(2) (see e.g.\cite{Lu2000}). Let $\Psi$ be continuous on $({\mR}^3\setminus\{0\})\times
({\mR}^3\setminus\{0\})$ and suppose that
$\Psi$ is nonnegative or generally such that the following
integrals makes sense. Then we have
\be\int_{{\mathbb R}^3}\fr{{\rm d}{\bf x}}{|{\bf x}|}
\int_{{\mathbb R}^2({\bf x})}\Psi({\bf x,y})
{\rm d}{\bf y}
=\int_{{\mathbb R}^3}\fr{{\rm d}{\bf y}}{|{\bf y}|}
\int_{{\mathbb R}^2({\bf y})}\Psi({\bf x,y})
{\rm d}{\bf x}.\lb{5.2}\ee

(3)  Let $\Psi$
be continuous on ${\mR}^2_{>0}$ and suppose that $\Psi$ is nonnegative or  generally
such that the following integrals make sense. Then
for any ${\bf x}\in {\mathbb R}^3\setminus\{0\}$ we have
\be\int_{{\mathbb R}^2({\bf x})}\Psi(|{\bf y}|, |{\bf v}-{\bf y}|)
{\rm d}{\bf y}=\int_{|{\bf v}_{\bf x}|}^{\infty}r_*'{\rm d}r_*'
\int_{0}^{2\pi}\Psi\Big(\Big|\sqrt{{r_*'}^2-|{\bf v}_{\bf x}|^2}+e^{{\rm i}\theta}
\sqrt{|{\bf v}|^2-|{\bf v}_{\bf x}|^2}\,\Big|,
r_*'\Big){\rm d}\theta\lb{5.3}\ee
where
${\bf v}_{\bf x}=({\bf v}\cdot\fr{{\bf x}}{|{\bf x}|}) \fr{{\bf x}}{|{\bf x}|}, {\rm i}=\sqrt{-1}.$
Here for real numbers $a,b,\theta$ we just use
$|a+ e^{{\rm i}\theta} b|=(a^2+b^2+2ab\cos(\theta))^{1/2}$ to shorten notation.
\vskip2mm

By identity $|{\bf v}_*|^2=|{\bf v}'|^2+|{\bf v}_*'|^2-|{\bf v}|^2$,
any function of $(...,|{\bf v}|^2/2,
|{\bf v}_*|^2/2,|{\bf v}'|^2/2, |{\bf v}_*'|^2/2\big)$ will be automatically written as a
 function of $(..., |{\bf v}|^2/2,|{\bf v}'|^2/2, |{\bf v}_*'|^2/2\big)$.

 In the following, for any $(x,y,z,s,\theta)\in {\mR}_{\ge 0}^4\times [0,2\pi]$, we denote as the above that  $x_*=(y+z-x)_{+}$ and let $Y_*=Y_*(x,y,z,s,\theta)$ be given by (\ref{Y}) and  $
\wt{Y_*}$ be defined by $Y_*(\cdot )$ with exchanging $y\leftrightarrow z$
i.e.
$$\wt{Y}_*=Y_*(x,z,y,s,\theta).$$

\begin{lemma}\lb{lemma5.1} Let $\Psi\in C({\mR}_{\ge 0}^5)$ be such that
the following integrals make sense ( for instance $\Psi\ge 0$ on ${\mR}_{\ge 0}^5$ or
$\Psi$ is such that integral is absolutely convergent). Then for any ${\bf v}\in {\mR}^3\setminus\{0\}$ we
have with $x=|{\bf v}|^2/2=r^2/2$ that
\bes&&\int_{{\mathbb R}^3\times{\mS}^2}\fr{|({\bf v}-{\bf v}_*)\cdot\omega|}{(4\pi)^2}
\Psi\big(|{\bf v}-{\bf v}'|, |{\bf v}-{\bf v}_*'|, |{\bf v}|^2/2,
 |{\bf v}'|^2/2, |{\bf v}_*'|^2/2\big){\rm d}\og{\rm d}{\bf v}_*\nonumber\\
&&=\fr{1}{4\pi\sqrt{x}}\int_{{\mR}_{\ge 0}^2}1_{\{y+z>x\}}{\rm d}y{\rm d}z
\int_{|\sqrt{x}-\sqrt{y}|\vee |\sqrt{x_*}-\sqrt{z}|}^{(\sqrt{x}+\sqrt{y})\wedge (\sqrt{x_*}+\sqrt{z})}
\int_{0}^{2\pi}\Psi\big(\sqrt{2} s, \sqrt{2} Y_*, x,y,z\big){\rm d}\theta {\rm d}s\nonumber\\
&&=
\fr{1}{4\pi\sqrt{x}}\int_{{\mR}_{\ge 0}^2}1_{\{y+z>x\}}{\rm d}y{\rm d}z
\int_{|\sqrt{x}-\sqrt{z}|\vee |\sqrt{x_*}-\sqrt{y}|}^{(\sqrt{x}+\sqrt{z})\wedge (\sqrt{x_*}
+\sqrt{y})}
\int_{0}^{2\pi}\Psi\big(\sqrt{2}\wt{Y_*}, \sqrt{2}s,x,y,z\big){\rm d}\theta {\rm d}s.
\lb{5.4}\ees
\end{lemma}
\begin{proof}  Fix any ${\bf v}=r\sg=\sqrt{2x}\,\sg\,\,(x>0, \sg\in {\bS}$).
Let $I({\bf v})$ be the left hand side of (\ref{5.4}).
Note that
$|({\bf v}-{\bf v}_*)\cdot\omega|=|{\bf v}-{\bf v}'|$. Then
using (\ref{5.1}),(\ref{5.3}) and simple changes of variables and denoting
$r_*=\sqrt{({r'}^2+{r_*'}^2-r^2)_{+}}$ we compute
\bes&&(4\pi)^2 I({\bf v})
=2\int_{{\mathbb R}^3}\fr{{\rm d}{\bf x}}{|{\bf x}|}\int_{{\mathbb R}^2({\bf x})}
\Psi\big(|{\bf x}|, |{\bf y}|,|{\bf v}|^2/2, |{\bf v}-{\bf x}|^2/2, |{\bf {\bf v}}-{\bf y}|^2/2\big)
{\rm d}{\bf y}\nonumber\\
&&=2\int_{{\mathbb R}^3}\fr{{\rm d}{\bf x}}{|{\bf x}|}
\int_{|{\bf v}_{\bf x}|}^{\infty} r_*'{\rm d}r_*'
\int_{0}^{2\pi}\Psi\Big(|{\bf x}|,
\Big|\sqrt{{r_*'}^2-|{\bf v}_{\bf x}|^2}+e^{{\rm i}\theta}\sqrt{|{\bf v}|^2-|{\bf v}_{\bf x}|^2}\,\Big|,
\fr{r^2}{2},\fr{|{\bf v-x}|^2}{2},
\fr{{r_*'}^2}{2}\Big){\rm d}\theta\nonumber
\\
&&=2\int_{{\mathbb R}^3}\fr{{\rm d}{\bf x}}{|{\bf v-x}|}
\int_{|{\bf v}\cdot \fr{{\bf v-x}}{|{\bf v-x}|}|}^{\infty} r_*'{\rm d}r_*'
\int_{0}^{2\pi}\nonumber\\
&&\times\Psi\Big(|{\bf v-x}|,
\Big|\sqrt{{r_*'}^2-|{\bf v}\cdot\fr{{\bf v-x}}{|{\bf v-x}|}|^2}
+e^{{\rm i}\theta}\sqrt{|{\bf v}|^2-| {\bf v}\cdot \fr{{\bf v-x}}{|{\bf v-x}|}|^2}\,\Big|,\fr{r^2}{2},
\fr{|{\bf x}|^2}{2},
\fr{{r_*'}^2}{2}\Big){\rm d}\theta\nonumber
\\
 &&=\fr{4\pi}{r}\int_{{\mR}_{\ge 0}^2, {r'}^2+{r_*'}^2>r^2} r'r_*'{\rm d}r'{\rm d}r_*'
\int_{|r'-r|\vee |r_*'-r_*|}^{(r'+r)\wedge (r_*+r_*')}{\rm d}s\nonumber
\\
&&\times
\int_{0}^{2\pi}\Psi\Big(s,
\Big|\sqrt{{r_*'}^2-\fr{(r^2-{r'}^2+s^2)^2}{4s^2}
}
+e^{{\rm i}\theta}\sqrt{r^2-\fr{(r^2-{r'}^2+s^2)^2}{4s^2}
}\,\Big|, \fr{r^2}{2},\fr{{r'}^2}{2}, \fr{{r_*'}^2}{2}\Big){\rm d}\theta\nonumber
\\
&&=
\fr{4\pi}{\sqrt{x}}\int_{{\mR}_{\ge 0}^2, y+z>x}
{\rm d}y{\rm d}z
\int_{|\sqrt{x}-\sqrt{y}|\vee |\sqrt{x_*}-\sqrt{z}|}^{(\sqrt{x}+\sqrt{y})\wedge (\sqrt{x_*}+\sqrt{z})}
\int_{0}^{2\pi}\Psi\big(\sqrt{2}s, \sqrt{2}Y_*,x,y,z\big){\rm d}\theta{\rm d}s\lb{6.new0}
\ees
where in the above calculation we have used the following properties: if
$r>0, r'>0, r_*'>0, r_*>0, -1<t<1, s>0$, then
\beas r_*'>\fr{|r^2-rr't|}{\sqrt{r^2+{r'}^2-2rr't}}
\Longrightarrow\, {r'}^2+{r_*'}^2>r^2;\quad  r_*'>\fr{|r^2-{r'}^2+s^2|}{2s}\,\,
\Longleftrightarrow\,\, |r_*'-r_*|<s<r_*'+r_*.\eeas
By the way, we also have the following relation
\bes&& |\sqrt{x}-\sqrt{y}|\vee |\sqrt{x_*}-\sqrt{z}| \le s\le (\sqrt{x}+\sqrt{y})\wedge (\sqrt{x_*}+\sqrt{z})\,\,\,{\rm and}\,\,\,s>0 \nonumber\\
&&\Longrightarrow\quad z-\fr{(x-y+s^2)^2}{4s^2}\ge 0\quad {\rm and}\quad
x-\fr{(x-y+s^2)^2}{4s^2}\ge 0\lb{5.pp}\ees
 which is useful when dealing with
$Y_*=Y_*(x,y,z,s,\theta)$.

Next using the formula (\ref{5.2}) and (\ref{6.new0}) with exchanges
$ {\bf x}\leftrightarrow {\bf y},\, y\leftrightarrow z$
we also have
\beas (4\pi)^2I({\bf v}) &=&2\int_{{\mathbb R}^3}\fr{{\rm d}{\bf y}}{|{\bf y}|}\int_{{\mathbb R}^2({\bf x})}
\Psi\big(|{\bf x}|, |{\bf y}|, |{\bf v}|^2/2, |{\bf v}-{\bf x}|^2/2, |{\bf {\bf v}}-{\bf y}|^2/2\big)
{\rm d}{\bf x}
\\
&=&2\int_{{\mathbb R}^3}\fr{{\rm d}{\bf x}}{|{\bf x}|}\int_{{\mathbb R}^2({\bf y})}
\Psi\big(|{\bf y}|, |{\bf x}|, |{\bf v}|^2/2, |{\bf v}-{\bf y}|^2/2, |{\bf {\bf v}}-{\bf x}|^2/2\big)
{\rm d}{\bf y}\\
&=&
\fr{4\pi}{\sqrt{x}}\int_{{\mR}_{\ge 0}^2, y+z>x}
{\rm d}y{\rm d}z
\int_{|\sqrt{x}-\sqrt{y}|\vee |\sqrt{x_*}-\sqrt{z}|}^{(\sqrt{x}+\sqrt{y})\wedge (\sqrt{x_*}+\sqrt{z})}
\int_{0}^{2\pi}\Psi\big(\sqrt{2}Y_*, \sqrt{2}s,x,z,y\big){\rm d}\theta{\rm d}s
\\
&=&
\fr{4\pi}{\sqrt{x}}\int_{{\mR}_{\ge 0}^2, y+z>x}
{\rm d}y{\rm d}z
\int_{|\sqrt{x}-\sqrt{z}|\vee |\sqrt{x_*}-\sqrt{y}|}^{(\sqrt{x}+\sqrt{z})\wedge (\sqrt{x_*}+\sqrt{y})}
\int_{0}^{2\pi}\Psi\big(\sqrt{2}\wt{Y_*}, \sqrt{2}s,x, y,z\big){\rm d}\theta{\rm d}s. \eeas
\end{proof}

\vskip2mm

\begin{lemma}\lb{lemma5.2} Let $\Psi\in C({\mR}_{\ge 0}^5)$. Then for any $y>0, z>0$
\bes&&
\int_{{\mS}^2\times{\mS}^2}{\rm d}\sg{\rm d}\sg_*
\int_{{\mS}^2}\fr{|({\bf v}-{\bf v}_*)\cdot\og|}{(4\pi)^2}\Psi\big(|{\bf v}-{\bf v}'|, |{\bf v}-{\bf v}_*'|, |{\bf v}|^2/2,
|{\bf v}'|^2/2, |{\bf v}_*'|^2/2\big)
{\rm d}\og\nonumber\\
&&=\fr{1}{\sqrt{2yz}}
\int_{0}^{y+z}{\rm d}x
\int_{|\sqrt{y}-\sqrt{x}|\vee |\sqrt{x_*}-\sqrt{z}|}^{(\sqrt{y}+\sqrt{x})\wedge (\sqrt{z}+
\sqrt{x_*})}
\int_{0}^{2\pi}\Psi\big(\sqrt{2} s, \sqrt{2} Y_*, x,y,z\big)
{\rm d}\theta {\rm d}s\lb{5.5}\ees
where in the left integral ${\bf v}=\sqrt{2y}\,\sg,\, {\bf v}_*=\sqrt{2z}\,\sg_*$.
\end{lemma}

\begin{proof}  It is easily seen that
both sides of (\ref{5.5}) are continuous in $(y,z)\in {\mR}_{>0}^2$.
Take any $\psi\in C_c({\mR}_{\ge 0}^2)$. We compute with change of variables and using Lemma \ref{lemma5.1}
\beas&&
\int_{{\mR}_{\ge 0}^2}\psi(y,z)\sqrt{y}\sqrt{z}\Big(
\int_{{\mS}^2\times{\mS}^2}{\rm d}\sg{\rm d}\sg_*
\int_{{\mS}^2}\fr{|({\bf v}-{\bf v}_*)\cdot\og|}{(4\pi)^2}\\
&&\times \Psi(|{\bf v}-{\bf v}'|, |{\bf v}-{\bf v}_*'|,|{\bf v}|^2/2,
|{\bf v}'|^2/2, |{\bf v}_*'|^2/2 )\Big|_{ {\bf v}=\sqrt{2y}\,\sg,\, {\bf v}_*=\sqrt{2z}\,\sg_*}
{\rm d}\og \Big){\rm d}y{\rm d}z\\
&&=\fr{1}{2}
\int_{{\mR}^3\times {\mR}^3\times{\mS}^2}\psi(|{\bf v}|^2/2, |{\bf v}_*|^2/2)
\fr{|({\bf v}-{\bf v}_*)\cdot\og|}{(4\pi)^2}\\
&&\times \Psi(|{\bf v}-{\bf v}'|, |{\bf v}-{\bf v}_*'|,|{\bf v}|^2/2,
|{\bf v}'|^2/2, |{\bf v}_*'|^2/2 )
{\rm d}\og {\rm d}{\bf v}{\rm d}{\bf v}_*
\\
&&=\fr{1}{2}
\int_{{\mR}^3\times {\mR}^3\times{\mS}^2}\psi(
|{\bf v}'|^2/2, |{\bf v}_*'|^2/2)\fr{|({\bf v}-{\bf v}_*)\cdot\og |}{(4\pi)^2}\\
&&\times \Psi(|{\bf v}-{\bf v}'|, |{\bf v}-{\bf v}_*'|,
|{\bf v}|^2/2,  |{\bf v}'|^2/2,|{\bf v}_*'|^2/2)
{\rm d}\og {\rm d}{\bf v}{\rm d}{\bf v}_*
\\
&&=
\int_{{\mR}_{\ge 0}^2}  \psi(y, z)\sqrt{yz}
\Big(\fr{1}{\sqrt{2yz}}\int_{0}^{y+z}{\rm d}x
\int_{|\sqrt{x}-\sqrt{y}|\vee |\sqrt{x_*}-\sqrt{z}|}^{(\sqrt{x}+\sqrt{y})\wedge (\sqrt{x_*}+\sqrt{z})}
\int_{0}^{2\pi}\Psi\big(\sqrt{2} s, \sqrt{2} Y_*, x,y,z\big){\rm d}\theta {\rm d}s
\Big){\rm d}y{\rm d}z.\eeas
Since $\psi $  is arbitrary, this implies (\ref{5.5}) by continuity.
\end{proof}
\vskip2mm

\begin{lemma}\lb{lemma5.3} Let $\Phi\in C({\mR}_{\ge 0}^2)$.
Then for any $x,y, z\ge 0$ satisfying $x_*:=y+z-x\ge 0$  we have
\be
\int_{|\sqrt{x}-\sqrt{z}|\vee |\sqrt{x_*}-\sqrt{y}|}^{(\sqrt{x}+\sqrt{z})\wedge (\sqrt{x_*}+\sqrt{y})}
{\rm d}s\int_{0}^{2\pi}\Phi(\sqrt{2}\wt{Y_*}, \sqrt{2}s)
{\rm d}\theta=\int_{|\sqrt{x}-\sqrt{y}|\vee |\sqrt{x_*}-\sqrt{z}|}^{(\sqrt{x}+\sqrt{y})\wedge (\sqrt{x_*}+\sqrt{z})} {\rm d}s
\int_{0}^{2\pi}\Phi(\sqrt{2}s, \sqrt{2}Y_*)
{\rm d}\theta\lb{5.7}\ee
and
\bes&&\int_{|\sqrt{x}-\sqrt{y}|\vee |\sqrt{x_*}-\sqrt{z}|}^{(\sqrt{x}+\sqrt{y})\vee (\sqrt{x_*}+\sqrt{z})}
 {\rm d}s\int_{0}^{2\pi}\fr{\sqrt{z} s}{\sqrt{y}Y_*+\sqrt{z}s}\Phi(\sqrt{2}s, \sqrt{2}Y_*)
{\rm d}\theta\nonumber\\
&&+
\int_{|\sqrt{x}-\sqrt{z}|\vee |\sqrt{x_*}-\sqrt{y}|}^{(\sqrt{x}+\sqrt{z})\wedge (\sqrt{x_*}+\sqrt{y})}
{\rm d}s\int_{0}^{2\pi}\fr{\sqrt{y} s}{\sqrt{z}\wt{Y_*}+\sqrt{y}s}\Phi(\sqrt{2}\wt{Y_*}, \sqrt{2}s)
{\rm d}\theta\nonumber\\
&&
=\int_{|\sqrt{x}-\sqrt{y}|\vee |\sqrt{x_*}-\sqrt{z}|}^{(\sqrt{x}+\sqrt{y})\vee (\sqrt{x_*}+\sqrt{z})}
{\rm d}s\int_{0}^{2\pi}\Phi(\sqrt{2}s, \sqrt{2}Y_*)
{\rm d}\theta. \lb{5.8}\ees
\end{lemma}
\begin{proof}  Recall $(\sqrt{x}+\sqrt{y})\wedge (\sqrt{x_*}+\sqrt{z})
-|\sqrt{x}-\sqrt{y}|\vee |\sqrt{x_*}-\sqrt{z}|
=2\min\{\sqrt{x},\sqrt{x_*},\sqrt{y},\sqrt{z}\,\}$. If $ \min\{\sqrt{x},\sqrt{x_*}, \sqrt{y}, \sqrt{z}\}=0$, then the above
 integrals are all zero.

In the following we suppose  $ \min\{\sqrt{x},\sqrt{x_*}, \sqrt{y}, \sqrt{z}\}>0$.
Take any $\psi\in C_c({\mR}_{\ge 0}^2)$. Applying
 the second equality in (\ref{5.4}) of
 Lemma \ref{lemma5.1} to the function
$\Phi\big(|{\bf v}-{\bf v}'|, |{\bf v}-{\bf v}_*'|\big)
\psi\big(|{\bf v}'|^2/2, |{\bf v}_*'|^2/2\big)$ we have
\beas&&\int_{{\mR}_{\ge 0}^2, y+z>x}\psi(y,z)\Big(
\int_{|\sqrt{x}-\sqrt{y}|\vee |\sqrt{x_*}-\sqrt{z}|}^{(\sqrt{x}+\sqrt{y})\wedge (\sqrt{x_*}+\sqrt{z})}
{\rm d}s\int_{0}^{2\pi}\Phi\big(\sqrt{2} s, \sqrt{2} Y_*\big){\rm d}\theta \Big){\rm d}y{\rm d}z \\
&&=
\int_{{\mR}_{\ge 0}^2, y+z>x}\psi(y,z)\Big(
\int_{|\sqrt{x}-\sqrt{z}|\vee |\sqrt{x_*}-\sqrt{y}|}^{(\sqrt{x}+\sqrt{z})\wedge (\sqrt{x_*}
+\sqrt{y})}{\rm d}s
\int_{0}^{2\pi}\Phi\big(\sqrt{2}\wt{Y_*}, \sqrt{2}s\big){\rm d}\theta\Big){\rm d}y{\rm d}z.
\eeas
Since $\psi\in C_c({\mR}_{\ge 0}^2)$ is arbitrary and both sides of (\ref{5.7}) are
continuous in $(x,y,z)$, this implies that (\ref{5.7}) holds true.

To prove (\ref{5.8}), we fix $x,y,z$ mentioned in the lemma and apply Lemma \ref{lemma5.3} to continuous functions
$$(r, \rho)\mapsto \Phi_n(r,\rho)=\fr{\sqrt{2z}\,r}{\sqrt{2y}\,\rho+\sqrt{2z}\,r+1/n} \Phi(r, \rho),\quad
(r,\rho)\in{\mR}_{\ge 0}^2$$
to get the revelent equality and then letting $n\to \infty$ and using Lebesgue dominated convergence
we obtain the equality
\beas&&
\int_{|\sqrt{x}-\sqrt{y}|\vee |\sqrt{x_*}-\sqrt{z}|}^{(\sqrt{x}+\sqrt{y})\vee (\sqrt{x_*}+\sqrt{z})}
{\rm d}s\int_{0}^{2\pi}\fr{\sqrt{z}s}{\sqrt{y}\,Y_*+\sqrt{z}\,s}\Phi(\sqrt{2}s, \sqrt{2}Y_*)
{\rm d}\theta \\
&&=
\int_{|\sqrt{x}-\sqrt{z}|\vee |\sqrt{x_*}-\sqrt{y}|}^{(\sqrt{x}+\sqrt{z})\vee (\sqrt{x_*}+\sqrt{y})}
{\rm d}s\int_{0}^{2\pi}
\fr{\sqrt{z}\wt{Y_*}}
{\sqrt{y}s+\sqrt{z}\wt{Y_*}}\Phi(\sqrt{2}\wt{Y_*}, \sqrt{2}s)
{\rm d}\theta.
\eeas
From this and (\ref{5.7}) it is easily deduced (\ref{5.8}).
\end{proof}

\begin{remark}\lb{remark6.1}{\rm Applying (\ref{5.7}) to the function $\Phi(r,\rho)$ given in (\ref{Phi}) one sees that the function $W(x,y,z)$ (defined in (\ref{W1}),(\ref{W2}))
is symmetric in $y,z$, i.e.
$W(x,y,z)\equiv W(x,z,y)$ on ${\mR}_{\ge 0}^3$.}
\end{remark}
\vskip2mm

The following lemma deals with equalities for total integration where the integrand $\Psi$ can be
arbitrary nonnegative Lebesgue measurable function.

\begin{lemma}\lb{lemma5.new1} Let $0\le \Phi\in C_b({\mR}_{\ge 0}^2)$,
$\Psi:{\mR}_{\ge 0}^3\to {\mR}_{\ge 0}$ be Lebesgue measurable.
Then
\bes&&\int_{{\bRRS}}\fr{|({\bf v}-{\bf v}_*)\cdot\omega|}{(4\pi)^2}
\Phi\big(|{\bf v}-{\bf v}'|, |{\bf v}-{\bf v}_*'|)\Psi(|{\bf v}|^2/2,
 |{\bf v}'|^2/2, |{\bf v}_*'|^2/2\big){\rm d}\og{\rm d}{\bf v}_*{\rm d}{\bf v}\nonumber\\
&&=\sqrt{2}\int_{{\mR}_{\ge 0}^3}1_{\{y+z>x\}}\Psi(x,y,z)
\left(
\int_{|\sqrt{x}-\sqrt{y}|\vee |\sqrt{x_*}-\sqrt{z}|}^{(\sqrt{x}+\sqrt{y})\wedge (\sqrt{x_*}+\sqrt{z})}
\int_{0}^{2\pi}\Phi\big(\sqrt{2} s, \sqrt{2} Y_*\big){\rm d}\theta {\rm d}s\right){\rm d}x{\rm d}y{\rm d}z\nonumber\\
&&=
\sqrt{2}\int_{{\mR}_{\ge 0}^3}1_{\{y+z>x\}}\Psi(x,y,z)
\left(
\int_{|\sqrt{x}-\sqrt{z}|\vee |\sqrt{x_*}-\sqrt{y}|}^{(\sqrt{x}+\sqrt{z})\wedge (\sqrt{x_*}
+\sqrt{y})}
\int_{0}^{2\pi}\Phi\big(\sqrt{2}\wt{Y_*}, \sqrt{2}s\big){\rm d}\theta {\rm d}s\right){\rm d}x{\rm d}y{\rm d}z.\qquad
\lb{5.4new}\ees
\end{lemma}

\begin{proof}  We need only prove the first equality sign in (\ref{5.4new}) since, by Lemma \ref{lemma5.3}, the second equality  holds true. [In fact, using the same proof below one sees that
the first line in (\ref{5.4new}) also equals to the third line.]
For notation convenience we define the measure $\mu$ on ${\bRS}$ by
\beas
{\rm d}\mu({\bf v}_*,\og)= \fr{|({\bf v}-{\bf v}_*)\cdot\omega|}{(4\pi)^2}
\Phi\big(|{\bf v}-{\bf v}'|, |{\bf v}-{\bf v}_*'|){\rm d}\og{\rm d}{\bf v}_*\eeas
and let
\beas\varrho(x,y,z)=\sqrt{2}\cdot1_{\{y+z>x\}}
\int_{|\sqrt{x}-\sqrt{y}|\vee |\sqrt{x_*}-\sqrt{z}|}^{(\sqrt{x}+\sqrt{y})\wedge (\sqrt{x_*}+\sqrt{z})}
\int_{0}^{2\pi}\Phi\big(\sqrt{2} s, \sqrt{2} Y_*\big){\rm d}\theta {\rm d}s,\quad (x,y,z)\in{\mR}_{\ge 0}^3.\eeas
Then the first equality in (\ref{5.4new}) is written
\be\int_{{\bRRS}}\Psi(|{\bf v}|^2/2,
 |{\bf v}'|^2/2, |{\bf v}_*'|^2/2\big){\rm d}\mu({\bf v}_*,\og){\rm d}{\bf v}
 = \int_{{\mR}_{\ge 0}^3}\Psi(x,y,z)\varrho(x,y,z){\rm d}x{\rm d}y{\rm d}z.\lb{6.new2}
\ee
From (\ref{1.difference}) we have
\be 0\le \varrho(x,y,z)\le\sqrt{2}\cdot 4\pi\|\Phi\|_{\infty} 1_{\{y+z>x\}}\min\{\sqrt{x}, \sqrt{x_*}, \sqrt{y},\sqrt{z}\,\}\quad \forall\, (x,y,z)\in{\mR}_{\ge 0}^3.\lb{6.new1}\ee
Then it is easily proved that $\varrho$ is continuous on
${\mR}_{\ge 0}^3$. But in the present proof, we need only the inequality (\ref{6.new1}) and the fact that
 $\varrho$ is Lebesgue measurable on
${\mR}_{\ge 0}^3$.
Now let us prove (\ref{6.new2}).

 {\bf Step1.} We prove that if $0\le \Psi\in C({\mR}_{\ge 0}^3)$, then (\ref{6.new2})
 holds true. In fact, starting from the left hand side of (\ref{6.new2}), using Fubini theorem and changing variable ${\bf v}=\sqrt{2x}\,\wt{\og}, \wt{\og}\in{\bS}$,
 and then applying Lemma \ref{lemma5.1} to the nonnegative
 continuous function $(x_1,x_2,x_3,x_4,x_5)\mapsto
 \Phi(x_1,x_2)\Psi(x_3,x_4,x_5)$, one obtains the equality
(\ref{6.new2}).

{\bf Step2.} Suppose that $\Psi$ is Lebesgue measurable on ${\mR}_{\ge 0}^3$ satisfying
$$0\le \Psi(x,y,z)<M\quad {\rm and}\quad 0\le\Psi(x,y,z)\le M 1_{[0,R]^3}(x,y,z)\qquad \forall\, (x,y,z)\in {\mR}_{\ge 0}^3$$
for some constants $0<M,R<\infty$. In this case we prove that (\ref{6.new2}) holds true.

In fact for any $n\in {\mN}$, applying Lusin's theorem (see e.g. a  proof in
\cite{Rudin})
there exists a function $\Psi_n\in C({\bR})$ satisfying
$$0\le \Psi_n(x,y,z)<M\quad \forall\,(x,y,z)\in {\bR};\quad
{\rm supp} \Psi_n\subset [-1/n, R+1/n]^3$$
such that
$$mes(\{(x,y,z)\in [0,R]^3\,\,|\,\,\Psi(x,y,z)\neq \Psi_n(x,y,z)\})<\fr{1}{n}.$$
Then it is easily seen that
$\Psi_n\to \Psi \,(n\to\infty)$ in $L^1({\mR}_{\ge 0}^3)$.
So there is a subsequence
$\{\Psi_{n_k}\}_{k=1}^{\infty}$ and a null set $Z\subset {\mR}_{\ge 0}^3$ such that
$\Psi_{n_k}(x,y,z)\to \Psi(x,y,z)\,(k\to\infty)$ for all $(x,y,z)\in{\mR}_{\ge 0}
\setminus Z$. Note that from the bound in (\ref{6.new1}) we also have the weighted strong convergence:
$\Psi_n\to \Psi \,(n\to\infty)$ in $L^1({\mR}_{\ge 0}^3, \varrho(x,y,z){\rm d}x{\rm d}y{\rm d}z)$.
It is easily proved that the set
$$\wt{Z}=\{({\bf v},{\bf v}_*,\og)\in{\bRRS}\,|\, (|{\bf v}|^2/2,|{\bf v}'|^2/2,
|{\bf v}_*'|^2/2)\in Z\}$$
has measure zero with the measure ${\rm d}\mu({\bf v}_*,\og){\rm d}{\bf v}$ (here as usual we
may assume that the product measure ${\rm d}\mu({\bf v}_*,\og){\rm d}{\bf v}$ is a complete measure).
Thus  $\Psi_{n_k}(|{\bf v}|^2/2,
 |{\bf v}'|^2/2, |{\bf v}_*'|^2/2\big)\\
\to \Psi(|{\bf v}|^2/2,
 |{\bf v}'|^2/2, |{\bf v}_*'|^2/2\big) \,(k\to\infty)$ for all
 $({\bf v},{\bf v}_*,\og)\in{\bRRS}\setminus \wt{Z}$.  Since, by {\bf Step1}, $\Psi_{n_k}$ satisfy (\ref{6.new2}), it follows from Fatou's lemma that
 \be\int_{{\bRRS}}\Psi(|{\bf v}|^2/2,
 |{\bf v}'|^2/2, |{\bf v}_*'|^2/2\big){\rm d}\mu({\bf v}_*,\og){\rm d}{\bf v}
\le \int_{{\mR}_{\ge 0}^3}\Psi(x,y,z)\varrho(x,y,z){\rm d}x{\rm d}y{\rm d}z.\lb{6.new3}
\ee
Since $|\Psi-\Psi_n|$ have the same properties as
$\Psi$ (just by replacing $R$ with $R+1$), the inequality (\ref{6.new3})
holds also for $|\Psi-\Psi_n|$ and thus $\Psi_n\to \Psi\,(n\to\infty)$ in
$L^1({\bRRS},{\rm d}\mu({\bf v}_*,\og){\rm d}{\bf v})$. We then conclude that $\Psi$ satisfies the equality (\ref{6.new2}).

{\bf Step3.}  Let $\Psi$ be given in the lemma and let
$$\Psi_n(x,y,z)=\Psi(x,y,z)1_{\{\Psi(x,y,z)<n\}}1_{[0, n]^3}(x,y,z),\quad
(x,y,z)\in {\mR}_{\ge 0}^3,\,\,n\in{\mN}.$$
Then $\Psi_n$ satisfy the conditions in {\bf Step2} with $M=R=n$.
So (\ref{6.new2}) holds for all $\Psi_n$. Also we have
$ 0\le \Psi_n\le \Psi_{n+1} $ and
$\lim\limits_{n\to\infty}\Psi_n(x,y,z)=\Psi(x,y,z)$ for all $(x,y,z)\in {\mR}_{\ge 0}^3.$
Thus taking the limit $n\to\infty$ to (\ref{6.new2}) for $\Psi_n$
we conclude from Levi's monotone convergence theorem
that (\ref{6.new2}) holds true for $\Psi$.
\end{proof}
\vskip2mm

\noindent {\bf 6.2. Equivalence of Solutions.}
In \cite{Lu2004} the measure-valued isotropic solution is defined
 for the measure $\bar{F}\in {\cal B}_2^{+}({\mR}_{\ge 0})$ whose special case is ${\rm d}\bar{F}(r)=4\pi r^2f(r^2/2){\rm d}r$, and the corresponding integrands $J_{B}[\vp](r,r_*), K_{B}[\vp](r,r', r_*')$
for quadratic and cubic collision integrals are defined as follows:  denote
$$B(|{\bf v-v}_*|,\cos(\theta))=\fr{|{\bf v-v}_*|\cos(\theta)}{(4\pi)^2}\Phi(|{\bf v- v}_*|\cos(\theta), |{\bf v-v}_*|\sin(\theta))$$
$\theta=\arccos(|({\bf v-v}_*)\cdot\omega|/|{\bf {\bf v-v}_*}|)$, and
for any $\vp\in C_b^2({\mathbb R}_{\ge 0})$, let
\beas J_{B}[\vp](r,r_*) &=& \fr{2}{(4\pi)^2}
\int_{{\mS}^2\times{\mS}^2}{\rm d}\sg{\rm d}\sg_*\int_{0}^{\pi/2}B(|{\bf v-v}_*|,\cos(\theta))
\sin(\theta){\rm d}\theta\\
&\times& \int_{0}^{2\pi}\big(\vp(|{\bf v}'|^2/2)-\vp(|{\bf v}|^2/2)\big){\rm d}\vartheta\Big|_{{\bf v}=r\sg, {\bf v}_*=r_*\sg_*},\eeas
\beas&& K_{B}[\vp](r,r', r_*')=\fr{4}{(4\pi)^3}\int_{{\mS}^2\times{\mS}^2}
{\rm d}\sg{\rm d}\sg'\int_{0}^{2\pi}1_{\{r_*'>r|\sg\cdot\xi|\}}
\fr{\Dt\vp(r^2/2,{r'}^2/2,{r_*'}^2/2)}{X r_*'+Yr'}\Phi(X,Y){\rm d}\vartheta\\
&&{\rm if}\quad (r'-r)(r_*'-r)\neq 0, \\
&& K_{B}[\vp](r,r', r_*')=0\quad {\rm if}\quad (r'-r)(r_*'-r)=0  \eeas
where $({\bf v}',{\bf v}'_*)$ is given by the $\og$-representation (\ref{colli}),
\bes&&
|{\bf v}'|^2=r^2\sin^2(\theta)+r_*^2\cos^2(\theta)
-2rr_*\sin(\theta)\cos(\theta)
\sqrt{1-\la \sg,\sg_*\ra}\,\cos(\vartheta),\nonumber\\
&&|{\bf v}_*'|^2=r^2\cos^2(\theta)+r_*^2\sin^2(\theta)
+2rr_*\sin(\theta)\cos(\theta)
\sqrt{1-\la \sg,\sg_*\ra}\,\cos(\vartheta),\nonumber\\
&&\Dt\vp(r^2/2,{r'}^2/2,{r_*'}^2/2)=\vp(r^2/2)+\vp(r_*^2/2)-\vp({r'}^2/2)-\vp({r_*'}^2/2),
\nonumber\\
&& r_*=\sqrt{({r'}^2+{r_*'}^2-r^2)_{+}},\nonumber\\
&&
X=|r\sg-r'\sg'|,\quad Y=\Big|\sqrt{{r_*'}^2-|\la r\sg,\xi\ra|^2}+e^{{\rm i}\vartheta}
\sqrt{r^2-\la r\sg,\xi\ra^2}\Big|,\quad r_*'>|r\sg\cdot\xi|,\qquad \lb{5.14}\\
&&\xi=\fr{r\sg-r'\sg'}{|r\sg-r'\sg'|}\,\,\,{\rm if}\,\,\, r\sg\neq r'\sg';\quad
\xi=\sg\,\,\,{\rm if}\,\,\, r\sg=r'\sg'.\qquad  \dnumber\lb{5.15}\ees
Note that $r_*'>|r\sg\cdot\xi|$ implies ${r'}^2+{r_*'}^2>r^2$, and
using (\ref{diff}) we have $$|\Dt\vp(r^2/2,{r'}^2/2,{r_*'}^2/2)|\le \fr{1}{4}\|\vp''\|_{\infty}
|({r'}^2-r^2)({r_*'}^2-r^2)|$$ from which
one sees that if $(r'-r)(r_*'-r)\neq 0$ and
$r_*'>r|\sg\cdot\xi|$ then $X\ge |r'-r|>0, Y\ge |r_*'-r|>0$ and so
$K_{B}[\vp](r,r', r_*')$ is well defined on ${\mR}_{\ge 0}^3$. Also it has been proven in \cite{Lu2004} that $J_{B}[\vp], K_{B}[\vp]$ are continuous on ${\mR}_{\ge 0}^2, {\mR}_{\ge 0}^3$ respectively.

\begin{definition}\label{definition5.1}(\cite{Lu2004}) Let $B({\bf {\bf v-v}_*},\omega)$ be given by (\ref{kernel}), (\ref{Phi}). Let $\bar{F}_0\in {\cal B}_{2}^{+}({\mathbb R}_{\ge 0})$. We say that a
family $\{\bar{F}_t\}_{t\ge 0}\subset {\cal B}_{2}^{+}({\mathbb R}_{\ge 0})$, or simply $\bar{F}_t$, is a conservative  measure-valued isotropic solution of Eq.(\ref{Equation1}) on the time-interval $[0, \infty)$ with the initial datum $\bar{F}_t|_{t=0}=\bar{F}_0$ if

{\rm (i)} $\int_{{\mR}_{\ge 0}}(1, r^2){\rm d}\bar{F}_t(r)=
\int_{{\mR}_{\ge 0}}(1, r^2){\rm d}\bar{F}_0(r)$
 for all $t\in [0, \infty)$,

{\rm (ii)} for every $\varphi\in C^{2}_b({\mathbb R}_{\ge 0})$,  $t\mapsto \int_{{\mathbb R}_{\ge 0}}\varphi(r^2/2){\rm d}\bar{F}_t(r)$ belongs to
$C^1([0, \infty))$,

{\rm (iii)} for every $\varphi\in C^{2}_b({\mathbb R}_{\ge 0})$ and $t\in [0,\infty)$,
\beas&&\frac{{\rm d}}{{\rm d}t}\int_{{\mathbb R}_{\ge 0}}\varphi(r^2/2){\rm
d}\bar{F}_t(r)\nonumber\\
&&= \int_{{\mathbb R}_{\ge 0}^2}J_{B}[\varphi](r,r_*){\rm d}\bar{F}_t(r)
{\rm d}\bar{F}_t(r_*)+ \int_{{\mathbb R}_{\ge 0}^3}K_{B}[\varphi](r,r',r_*')
{\rm d}\bar{F}_t(r){\rm d}\bar{F}_t(r'){\rm d}\bar{F}_t(r_*').\eeas
\end{definition}

It has been proven in \cite{Lu2004} that for any $\bar{F}_0\in {\cal B}_{2}^{+}({\mathbb R}_{\ge 0})$,
the Eq.(\ref{Equation1}) with the initial datum $\bar{F}|_{t=0}=\bar{F}_0$ has a conservative  measure-valued isotropic solution on $[0,\infty)$
in the sense of Definition\ref{definition5.1}.

Let $F\in {\cal B}_1^{+}({\mathbb R}_{\ge 0}), \bar{F}\in {\cal B}_2^{+}({\mathbb R}_{\ge 0})$
be defined from each other by
\be F(A)=\fr{1}{4\pi\sqrt{2}}\int_{{\mathbb R}_{\ge 0}}{\bf 1}_{A}(r^2/2){\rm d}\bar{F}(r),\quad
\bar{F}(A)=4\pi\sqrt{2}\int_{{\mathbb R}_{\ge 0}}
{\bf 1}_{A}(\sqrt{2x}){\rm d}F(x)\lb{5.FF}\ee for all Borel sets $A\subset {\mR}_{\ge 0}$.
It is easily seen that if $F\in {\cal B}_1^{+}({\mathbb R}_{\ge 0}), \bar{F}\in {\cal B}_2^{+}({\mathbb R}_{\ge 0})$ are given through one of the equalities in (\ref{5.FF}), then
\be 4\pi\sqrt{2}\int_{{\mathbb R}_{\ge 0}}\vp(x){\rm
d}F(x)=\int_{{\mathbb R}_{\ge 0}}\vp(r^2/2){\rm d}\bar{F}(r)
\lb{5.20}\ee
for all Borel measurable functions $\vp$ on ${\mR}_{\ge 0}$ satisfying
$\sup\limits_{x\ge
0}(1+x)^{-1}|\vp(x)|<\infty$.
In the special case where $F$ and $\bar{F}$ are given by ${\rm d}F(x)= f(x)\sqrt{x}\,{\rm d}x$,
${\rm d}\bar{F}(r)=4\pi f(r^2/2) r^2{\rm d} r$,  the above relation is just the change of variable.

The following lemma and proposition show that the  Definition \ref{definition1.1} and
Definition \ref{definition5.1} are equivalent. This also ensures the existence of solutions in the sense of
Definition \ref{definition1.1}.

\begin{lemma}\lb{lemma5.6}  Let $B({\bf {\bf v-v}_*},\omega)$ be given by (\ref{kernel}), (\ref{Phi}), and let $\bar{F}\in {\cal B}_{2}^{+}({\mathbb R}_{\ge 0}), F\in {\cal B}_{1}^{+}({\mR}_{\ge 0})$
be given through one of the equalities in (\ref{5.FF}). Then for any
$\vp\in C_b^2({\mR}_{\ge 0})$ we have
\be\int_{{\mR}_{\ge 0}^2}J_{B}[\vp](r,r_*){\rm d}\bar{F}(r){\rm d}\bar{F}(r_*)
=4\pi\sqrt{2}\int_{{\mR}_{\ge 0}^2}{\cal J}[\vp](y,z)
{\rm d}F(y){\rm d}F(z), \lb{5.JJ}\ee
\be \int_{{\mR}_{\ge 0}^3}K_{B}[\vp](r,r',r_*'){\rm d}\bar{F}(r)
{\rm d}\bar{F}(r'){\rm d}\bar{F}(r_*')
=4\pi\sqrt{2}
\int_{{\mR}_{\ge 0}^3}{\cal K}[\vp](x,y,z)
{\rm d}F(x){\rm d}F(y){\rm d}F(z).\lb{5.KK}\ee
\end{lemma}

\begin{proposition}\lb{prop.5.7} Let $F_t\in {\cal B}^{+}_1({\mathbb R}_{\ge 0})$,
$\bar{F}_t\in{\cal B}^{+}_2({\mathbb R}_{\ge 0})$ satisfy one of the
equalities in (\ref{5.FF}) for every $t\in[0,\infty)$.
Then $F_t$ is a conservative  measure-valued isotropic solution of Eq.(\ref{Equation1})
 in the sense of Definition \ref{definition1.1}, if and only if
$\bar{F}_t$ is  a conservative  measure-valued isotropic solution of Eq.(\ref{Equation1})
 in the sense of Definition \ref{definition5.1}.
\end{proposition}

{\bf Proof of Lemma \ref{lemma5.6} and Proposition \ref{prop.5.7}}.
From (\ref{5.20}) we see that Proposition \ref{prop.5.7} follows easily from
Lemma \ref{lemma5.6}.  So we need only prove  Lemma \ref{lemma5.6}.

We first prove that for any $\vp\in C_c^2({\mR})$ and any $(r,r_*)\in {\mR}_{\ge 0}^2$
\bes&& \int_{{\mS}^2\times {\mS}^2}{\rm d}\sg{\rm d}\sg_*\int_{{\mS}^2}
\fr{|({\bf v}-{\bf v}_*)\cdot\omega|}{(4\pi)^2}\Phi(|{\bf v}-{\bf v}'|, |{\bf v}-{\bf v}_*'|)\nonumber\\
&&\times \Big(\vp(|{\bf v}'|^2/2)+\vp(|{\bf v}_*'|^2/2)-\vp(|{\bf v}|^2/2)-\vp(|{\bf v}_*|^2/2)\Big)
\Big|_{{\bf v}=r\sg, {\bf v}_*=r_*\sg_*}{\rm d}\og
\nonumber\\
&&=4\pi\sqrt{2}{\cal J}[\vp](r^2/2, r_*^2/2).\lb{5.22}
\ees
It is easily seen that the right hand side of (\ref{5.22})
is also continuous in $(r,r_*)\in{\mR}_{\ge 0}^2$. Thus
we need only prove (\ref{5.22}) for all $(r,r_*)\in {\mR}_{>0}^2$.
Let $(r,r_*)\in {\mR}_{>0}^2$. Using Lemma \ref{5.2} to $ y=r^2/2, z=r_*^2/2$ and recalling
definition of ${\cal J}[\vp]$ we have
\beas&& {\rm the \,\,l.h.s.\,\, of\,\, (\ref{5.22})}=\fr{1}{\sqrt{2yz}}
\int_{0}^{y+z}{\rm d}x
\int_{|\sqrt{y}-\sqrt{x}|\vee |\sqrt{x_*}-\sqrt{z}|}^{(\sqrt{y}+\sqrt{x})\wedge (\sqrt{z}+
\sqrt{x_*})}{\rm d}s
\int_{0}^{2\pi}\Phi(\sqrt{2}s, \sqrt{2}Y_*)
\\
&&\times \big(\vp(x_*)+\vp(x)-\vp(z)-
\vp(y)\big){\rm d}\theta
\Big|_{{\bf v}=\sqrt{2z}\,\sg, {\bf v}_*=\sqrt{2y}\,\sg_*}
\\
&&=4\pi\sqrt{2}{\cal J}[\vp](y,z)=4\pi\sqrt{2}
{\cal J}[\vp](r^2/2, r_*^2/2).\eeas
Now first making exchanges $r\leftrightarrow  r_*,\,\sg \leftrightarrow
\sg_*$ and using
$\int_{0}^{2\pi}g(\cos\vartheta){\rm d}\vartheta=\int_{0}^{2\pi}g(-\cos\vartheta){\rm d}\theta
$ and then using  (\ref{5.22}) we compute
\beas&&\int_{{\mR}_{\ge 0}^2}J_{B}[\vp](r,r_*){\rm d}\bar{F}(r){\rm d}\bar{F}(r_*)
\\
&&=\fr{1}{(4\pi)^2}\int_{{\mR}_{\ge 0}^2}
\int_{{\mS}^2\times{\mS}^2}{\rm d}\sg{\rm d}\sg_*\int_{0}^{\pi/2}
B(|{\bf v}-{\bf v}_*|,\cos(\theta))\sin(\theta){\rm d}\theta\\
&&\times \int_{0}^{2\pi}\Big(\vp(|{\bf v}'|^2/2)+\vp(|{\bf v}_*'|^2/2)-
\vp(|{\bf v}|^2/2)-\vp(|{\bf v}_*|^2/2)\Big){\rm d}\vartheta
\Big|_{{\bf v}=r\sg, {\bf v}_*=r_*\sg_*}{\rm d}\bar{F}(r){\rm d}\bar{F}(r_*)
\\
&&=\fr{1}{(4\pi)^2}\fr{1}{2}\int_{{\mR}_{\ge 0}^2}
\int_{{\mS}^2\times{\mS}^2}{\rm d}\sg{\rm d}\sg_*\int_{{\bS}}
\fr{|({\bf v}-{\bf v}_*)\cdot\og|}{(4\pi)^2}\Phi(|{\bf v}-{\bf v}'|, |{\bf v}-{\bf v}_*'|)
\\
&&\times \Big(\vp(|{\bf v}'|^2/2)+\vp(|{\bf v}_*'|^2/2)-
\vp(|{\bf v}|^2/2-\vp(|{\bf v}_*|^2/2)\Big){\rm d}\og
\Big|_{{\bf v}=r\sg, {\bf v}_*=r_*\sg_*}{\rm d}\bar{F}(r){\rm d}\bar{F}(r_*)
\\
&&=\fr{1}{(4\pi)^2}\fr{1}{2}4\pi\sqrt{2}\int_{{\mR}_{\ge 0}^2}{\cal J}[\vp](r^2/2, r_*^2/2)
{\rm d}\bar{F}(r){\rm d}\bar{F}(r_*)\\
&&
=4\pi\sqrt{2}\int_{{\mR}_{\ge 0}^2}{\cal J}[\vp](y, z)
{\rm d}F(y){\rm d}F(z).
\eeas
Next by definition of $K_{B}[\vp](r,r', r_*')$ and
its properties mentioned above we see that for any $r\ge 0, r'\ge 0, r_*'\ge 0$
satisfying $(r'-r)(r_*'-r)= 0$, or
 ${r'}^2+{r_*'}^2-r^2\le 0$, or $r_*'=0$,
 then $K_{B}[\vp](r,r', r_*')=0$.
So we need only consider the integration domain
$${\cal R}=\{(r,r',r_*')\in{\mR}_{\ge 0}\,\,|\,\,
(r'-r)(r_*'-r)\neq 0\,\,{\rm and}\,\,
{r'}^2+{r_*'}^2-r^2> 0\,\,{\rm and} \,\, r_*'>0\}.$$
Then we have
\bes&&\int_{{\mR}_{\ge 0}^3}K_{B}[\vp](r,r',r_*'){\rm d}\bar{F}(r)
{\rm d}\bar{F}(r'){\rm d}\bar{F}(r_*')
=\int_{{\cal R}}K_{B}[\vp](r,r',r_*'){\rm d}\bar{F}(r)
{\rm d}\bar{F}(r'){\rm d}\bar{F}(r_*')\nonumber
\\
&&=\Big(\int_{{\cal R}_1}+\int_{{\cal R}_2}+\int_{{\cal R}_3}\Big)
K_{B}[\vp](r,r',r_*'){\rm d}\bar{F}(r)
{\rm d}\bar{F}(r'){\rm d}\bar{F}(r_*'):= I_1+I_2+I_3\lb{5.24}\ees
where ${\cal R}_1={\cal R}\cap\{r>0, r'>0\}, {\cal R}_2={\cal R}\cap\{r=0, r'>0\}, {\cal R}_3=
{\cal R}\cap \{r>0,r'=0\}$.

For the integrand $K_{B}[\vp](r,r', r_*')$ in the first term $I_1$ we compute using change of variables
(e.g. $t=\fr{r^2+{r'}^2-s^2}{2rr'}$)  and
then letting $x=r^2/2, y={r'}^2/2, z={r_*'}^2/2$ (for $(r,r', r_*')\in {\cal R}_1$)
that
\beas&&K_{B}[\vp](r,r', r_*')\\
&&=\fr{2}{4\pi}
\int_{-1}^{1}{\rm d}t\int_{0}^{2\pi}1_{\{r_*'>\fr{|r^2-rr't|}
{\sqrt{r^2+{r'}^2-2rr't}}\}}
\fr{\Dt\vp(r^2/2,{r'}^2/2,{r_*'}^2/2)}{\sqrt{r^2+{r'}^2-2rr't}r_*'+Yr'}
\Phi(\sqrt{r^2+{r'}^2-2rr't},
Y){\rm d}\theta
\\
&&=\fr{2}{4\pi}
\int_{|r-r'|}^{r+r'}\fr{s}{rr'}{\rm d}s\int_{0}^{2\pi}1_{\{r_*'>\fr{|r^2-{r'}^2+s^2|}{2s}\}}
\fr{\Dt\vp(r^2/2,{r'}^2/2,{r_*'}^2/2)}{sr_*'+Yr'} \Phi(s, Y){\rm d}\theta\\
&&=\fr{1}{4\pi}\fr{\Dt\vp(x,y,z)}{\sqrt{xyz}}
\int_{|\sqrt{x}-\sqrt{y}|\vee |\sqrt{x_*}-\sqrt{z}|}^{(\sqrt{x}+\sqrt{y})
\wedge (\sqrt{x_*}+\sqrt{z})}{\rm d}s\int_{0}^{2\pi}\fr{s\sqrt{z}}{s\sqrt{z}+Y_*\sqrt{y}}
\Phi(\sqrt{2} s, \sqrt{2}Y_*){\rm d}\theta.
\eeas
Then with exchange $y \leftrightarrow z $ and recalling $\Phi(r,\rho)\equiv \Phi(\rho,r)$,
and using Lemma \ref{lemma5.3} we have
\beas&& I_1=
(4\pi\sqrt{2})^3
\int_{{\mR}_{\ge 0}^3} {\bf 1}_{{\cal R}_1}(2\sqrt{x},\sqrt{2y}, \sqrt{2z})K_{B}[\vp](2\sqrt{x},\sqrt{2y}, \sqrt{2z})
{\rm d}F(x)
{\rm d}F(y){\rm d}F(z)\\
&&=\fr{(4\pi\sqrt{2})^3}{4\pi}
\int_{{\mR}_{\ge 0}^3, (y-x)(z-x)\neq 0, x>0,y>0,z>0,y+z>x}
\fr{\Dt\vp(x,y,z)}{\sqrt{xyz}}\\
&&\qquad \qquad\times \Bigg\{
\int_{|\sqrt{x}-\sqrt{y}|\vee |\sqrt{x_*}-\sqrt{z}|}^{(\sqrt{x}+\sqrt{y})
\wedge (\sqrt{x_*}+\sqrt{z})}{\rm d}s\int_{0}^{2\pi}\fr{s\sqrt{z}}{s\sqrt{z}+Y_*\sqrt{y}}
\Phi(\sqrt{2} s, \sqrt{2}Y_*){\rm d}\theta
\\
&&\qquad \qquad+
\int_{|\sqrt{x}-\sqrt{z}|\vee |\sqrt{x_*}-\sqrt{y}|}^{(\sqrt{x}+\sqrt{z})
\wedge (\sqrt{x_*}+\sqrt{y})}{\rm d}s\int_{0}^{2\pi}\fr{s\sqrt{y}}{s\sqrt{y}+\wt{Y}_*\sqrt{z}}
\Phi(\sqrt{2} s, \sqrt{2}\wt{Y}_*){\rm d}\theta
\Bigg\}{\rm d}F(x){\rm d}F(y){\rm d}F(z)
\\
&&=4\pi\sqrt{2}
\int_{{\mR}_{\ge 0}^3, (y-x)(z-x)\neq 0, x>0,y>0,z>0,y+z>x}
\fr{\Dt\vp(x,y,z)}{4\pi\sqrt{xyz}}\\
&&\qquad \qquad \times
\int_{|\sqrt{x}-\sqrt{y}|\vee |\sqrt{x_*}-\sqrt{z}|}^{(\sqrt{x}+\sqrt{y})
\wedge (\sqrt{x_*}+\sqrt{z})}{\rm d}s\int_{0}^{2\pi}
\Phi(\sqrt{2} s, \sqrt{2}Y_*){\rm d}\theta {\rm d}F(x){\rm d}F(y){\rm d}F(z)
\\
&&=4\pi\sqrt{2}\int_{{\mR}_{\ge 0}^3, x>0,y>0,z>0}
{\cal K}[\vp](x,y,z){\rm d}F(x){\rm d}F(y){\rm d}F(z)
\eeas
where here and below we used the fact that $(y-x)(z-x)=0$ or $y+z\le x\, \Longrightarrow {\cal K}[\vp](x,y,z)=0.$

To compute $I_2$ we recall definition of $X, Y$ in (\ref{5.14}),(\ref{5.15})
to see that $(r,r',r_*')\in {\cal R}_2 \Longrightarrow r=0,
X=r'>0, Y=r_*'>0,$
and so letting $x=0, y={r'}^2/2>0, z={r_*'}^2/2>0$  gives
\beas&& K_{B}[\vp](0,r', r_*')
=\fr{1}{(4\pi)^3}
\int_{-1}^{1}{\rm d}t\int_{0}^{2\pi}1_{\{r_*'>0\}}
\fr{\Dt\vp(0,{r'}^2/2,{r_*'}^2/2)}{r'r_*'}\Phi(r',r_*'){\rm d}\theta
\\
&&
=\fr{2\pi}{(4\pi)^3}W(0,y,z)\Dt\vp(0,y,z)
=\fr{2\pi}{(4\pi)^3}{\cal K}[\vp](0,y,z)\eeas
and so
\beas I_2 &=&
(4\pi\sqrt{2})^3
\int_{{\mR}_{\ge 0}^3}{\bf 1}_{{\cal R}_2}(\sqrt{2x},\sqrt{2y},\sqrt{2z}){\cal K}[\vp](\sqrt{2x},\sqrt{2y},\sqrt{2z})
{\rm d}F(x){\rm d}F(y){\rm d}F(z)\\
&=&4\pi\sqrt{2}
\int_{{\mR}_{\ge 0}^3, x=0, y>0,z>0}{\cal K}(x,y,z){\rm d}F(x)
{\rm d}F(y){\rm d}F(z).\eeas
For the integral $I_3$, we have $(r,r',r_*')\in {\cal R}_3\Longrightarrow r'=0$,
$r_*'>X=r>0, Y=\sqrt{{r_*'}^2-r^2}$ and so
letting $x=r^2/2>0, y=0, z={r_*'}^2/2>0$ gives
\beas&&K_{B}[\vp](r,0, r_*')
=\fr{2}{(4\pi)^3}
\int_{-1}^{1}{\rm d}t\int_{0}^{2\pi}1_{\{r_*'>r\}}
\fr{\Dt\vp(r^2/2,0,{r_*'}^2/2)}{rr_*'}
\Phi\big(r, \sqrt{{r_*'}^2-r^2}\,\big){\rm d}\theta
\\
&&=\fr{2}{(4\pi)^3}4\pi1_{\{z>x\}}
\fr{\Dt\vp(x,0,z)}{2\sqrt{xz}}\Phi(\sqrt{2x}, \sqrt{2(z-x)}\,)
=\fr{2}{(4\pi)^2}1_{\{z>x>0\}}{\cal K}[\vp](x,0,z).
\eeas
Then using the symmetry ${\cal K}[\vp](x,y,z)\equiv {\cal K}[\vp](x,z,y)$ we have
\beas I_3 &=&
(4\pi\sqrt{2})^3
\int_{{\mR}_{\ge 0}^3}{\bf 1}_{{\cal R}_3}(\sqrt{2x},\sqrt{2y},\sqrt{2z}){\cal K}[\vp](\sqrt{2x},\sqrt{2y},\sqrt{2z})
{\rm d}F(x){\rm d}F(y){\rm d}F(z)\\
&=&
4\pi\sqrt{2}\cdot 2
\int_{{\mR}_{\ge 0}^3, y=0, z>x>0 }{\cal K}[\vp](x,y,z)
{\rm d}F(x){\rm d}F(y){\rm d}F(z)\\
&=&
4\pi\sqrt{2}
\int_{{\mR}_{\ge 0}^3, y=0, z>x>0 }{\cal K}[\vp](x,y,z)
{\rm d}F(x){\rm d}F(y){\rm d}F(z)\\
&+&
4\pi\sqrt{2}
\int_{{\mR}_{\ge 0}^3, z=0, y>x>0 }{\cal K}[\vp](x,y,z)
{\rm d}F(x){\rm d}F(y){\rm d}F(z).
\eeas
Taking sum $I_1+I_2+I_3$ and using (\ref{5.24}) we obtain (\ref{5.KK}).$\hfill\Box$
\\

\noindent {\bf 6.3. Existence and positivity of some potential $U(|{\bf x}|)$.}

\begin{lemma}\lb{lemma5.10} Let $V\in C_b({\mR}_{\ge 0})\cap C^2({\mR}_{>0}), V_1(r)=V(r) r, $ and
suppose that  $\fr{{\rm d}^2}{{\rm d}r^2}V_1(r)$ is non-decreasing
in $(0,\infty)$ and there are constants $0<\dt<1, C>0$ such that
\beas \Big|\fr{{\rm d}}{{\rm d}r}V_1(r)\Big|\le \fr{C}{1+r^{2\dt}},\quad r>0;\quad \int_{0}^{\infty}(1+r^{\dt})\Big|\fr{{\rm d}^2}{{\rm d}r^2}V_1(r)\Big|{\rm d}r<\infty.\eeas
Let
\beas U(\rho)
=\fr{1}{2\pi^2\rho^3}\int_{0}^{\infty}\Big(-\,
\fr{{\rm d}^2}{{\rm d}r^2}V_1(r)\Big) \sin(\rho r) {\rm d}r,\quad\rho>0.\eeas
Then $U(\rho)\ge 0$ for all $\rho\in (0,\infty)$ and
$\wh{U}(\xi)=V(|\xi|)$ for all $\xi\in {\bR}$,
where $\wh{U}$ is the Fourier transform of the function
${\bf x}\mapsto U (|{\bf x}|)$ in terms of theory of generalized functions.

\end{lemma}

\begin{proof}  To prove the positivity of $U$ we will use the following property:
if $f\in L^1((0,\infty))$ is non-negative and non-increasing in $(0,\infty)$, then
\be \int_{0}^{\infty} f(r)\sin(\rho r){\rm d}r\ge 0\qquad \forall\, \rho>0.\lb{5.po}\ee
The proof of (\ref{5.po}) is easy: first for any $0<a<b<\infty$, according to the second mean-value
theorem of integration, there is $c\in (a,b)$ such that $\int_{a}^{b} f(r)\sin(\rho r){\rm d}r=f(a)\int_{a}^{c}\sin(\rho r){\rm d}r$, then use the fact that $0\le f\in L^1((0,\infty))$ implies  $\liminf\limits_{a\to 0+} f(a) a^2=0$.
These deduce (\ref{5.po}).

To prove the lemma we use approximation: for any $\vep>0$, let
$$U_{\vep}({\bf x})=
(2\pi)^{-3}\int_{\bR} e^{-\vep |\xi|}V(|\xi|)e^{{\rm i}{\bf x}\cdot\xi}
{\rm d}\xi,\quad {\bf x}\in {\bR}.$$
$U_{\vep}({\bf x})$ is obviously radially symmetric, so if we denote
$U_{\vep}(\rho)=U_\vep({\bf x})|_{|{\bf x}|=\rho}$, then
$$U_{\vep}(\rho)=
\fr{1}{2\pi^2\rho}
  \int_{0}^{\infty} e^{-\vep r}V_1(r) \sin(\rho r){\rm d}r,\quad \rho>0.$$
Next using integration by parts twice we deduce
\beas&&
\int_{0}^{\infty} e^{-\vep r}V_1(r) \sin(\rho r){\rm d}r
=\fr{1}{\vep}
\int_{0}^{\infty} e^{-\vep r}\fr{{\rm d}}{{\rm d}r}V_1(r) \sin(\rho r){\rm d}r\\
&&
+\fr{1}{\vep^2}
\int_{0}^{\infty} e^{-\vep r}\fr{{\rm d}}{{\rm d}r}V_1(r) \cos(\rho r) \rho {\rm d}r
-\fr{\rho^2}{\vep^2}
\int_{0}^{\infty} e^{-\vep r}V_1(r) \sin(\rho r){\rm d}r.\eeas
For the integral that contains  $\cos(\rho r)$, integrating by parts again and reorganizing the results we obtain
\beas U_{\vep}(\rho)
&=&\fr{1}{2\pi^2\rho(\vep^2+\rho^2)}\Big(2\vep
\int_{0}^{\infty} e^{-\vep r}\fr{{\rm d}}{{\rm d}r}V_1(r) \sin(\rho r){\rm d}r
+\int_{0}^{\infty}e^{-\vep r}(-
\fr{{\rm d}^2}{{\rm d}r^2}V_1(r)) \sin(\rho r) {\rm d}r
\Big)\\
&=& U_{1,\vep}(\rho)+U_{2,\vep}(\rho),\quad \rho>0.\eeas
From the assumption we see that $-\fr{{\rm d}^2}{{\rm d}r^2}V_1(r)$ is non-increasing
in $(0,\infty)$ and so by integrability we have
$-\fr{{\rm d}^2}{{\rm d}r^2}V_1(r)\ge \lim\limits_{R\to\infty}(-\fr{{\rm d}^2}{{\rm d}r^2}V_1(R))=0$ for all $r>0.$
From this we see that $\fr{{\rm d}}{{\rm d}r}V_1(r)$ is non-increasing in $(0,\infty)$
and so
$\fr{{\rm d}}{{\rm d}r}V_1(r)\ge \lim\limits_{R\to\infty}\fr{{\rm d}}{{\rm d}r}V_1(R)=0$ for all $r>0$.
 Thus the functions $-\fr{{\rm d}^2}{{\rm d}r^2}V_1(r), e^{-\vep r}\fr{{\rm d}}{{\rm d}r}V_1(r), e^{-\vep r}(-\fr{{\rm d}^2}{{\rm d}r^2}V_1(r))$ are non-negative and non-increasing in $(0,\infty)$
and belong to $L^1((0,\infty))$.
By (\ref{5.po}) we conclude $U(\rho)\ge 0, U_{1,\vep}(\rho)\ge 0, U_{2,\vep}(\rho)\ge 0$ hence $U_{\vep}(\rho)\ge 0$ for all $\rho>0$
($\forall\, \vep>0$).
Next we show that
\be \sup_{0<\vep\le 1}U_{\vep}(\rho)\le C_3
\fr{1}{\rho^{3-\dt}},\quad \lim_{\vep\to 0}U_{\vep}(\rho)=U(\rho)
\quad \forall\, \rho>0\lb{5.41}\ee
Here and below the constants $C_i>0\,(i=1,2,3,4) $ depend only on
$V, C$ and $\dt$.

In fact for all $\vep>0,\rho>0$  we have
\be0\le U_{1,\vep}(\rho)
\le \fr{C}{\pi^2\rho^3}\vep
\int_{0}^{\infty} e^{-\vep r}\fr{(\rho r)^{\dt}}{r^{2\dt}}{\rm d}r
=\fr{C_1}{\rho^{3-\dt}}\vep^{\dt}
\lb{5.42}\ee
and (using
$|\sin(x)|\le x^{\dt}$ for all $x\ge 0$)
\be
\fr{1}{2\pi^2\rho(\vep^2+\rho^2)}e^{-\vep r}\Big|
\fr{{\rm d}^2}{{\rm d}r^2}V_1(r))\Big||\sin(\rho r)|
\le\fr{C_2}{\rho^{3-\dt}}r^{\dt}\Big|
\fr{{\rm d}^2}{{\rm d}r^2}V_1(r))\Big|,\quad r>0.\lb{5.43}\ee
This gives the first inequality in (\ref{5.41}).
From (\ref{5.43}) and Lebesgue dominated convergence we obtain
$$\lim_{\vep\to 0+}U_{2,\vep}(\rho)=\lim_{\vep\to 0+}
\fr{1}{2\pi^2\rho(\vep^2+\rho^2)}
\int_{0}^{\infty}e^{-\vep r}\Big(-\,\fr{{\rm d}^2}{{\rm d}r^2}V_1(r)\Big) \sin(\rho r) {\rm d}r
=U(\rho)$$
for all $\rho>0$.
This together with (\ref{5.42}) gives (\ref{5.41}).

Now let ${\cal S}({\bR})$ be the class of Schwartz functions on ${\bR}$.
For any $\psi\in {\cal S}({\bR})$ we have $\wh{\psi}\in {\cal S}({\bR})$ and it
holds the inverse formula
$$\psi(\xi)=2\pi)^{-3}\int_{{\bR}}
\wh{\psi}({\bf x})e^{{\rm i}{\bf x}\cdot \xi}{\rm d}{\bf x}
\qquad \forall\,\xi\in{\bR}$$
from which and Fubini theorem we obtain
\be \int_{{\bR}}U_{\vep} (|{\bf x}|)
\wh{\psi}({\bf x}){\rm d}{\bf x}=
\int_{{\bR}}e^{-\vep |\xi|} V(|\xi|)\psi(\xi){\rm d}\xi,\quad \vep>0.\lb{5.44}\ee
Since
$$ \sup_{0<\vep\le 1}| U_{\vep} (|{\bf x}|)
 \widehat{\psi}({\bf x})|\le C_3|\widehat{\psi}({\bf x})|\fr{1}{|{\bf x}|^{3-\dt}},\quad
\sup_{\vep>0}|e^{-\vep |\xi|}V(|\xi|)\psi(\xi)|\le C_4|\psi(\xi)|,$$
$$ \lim_{\vep\to 0+}U_{\vep} (|{\bf x}|)
\wh{\psi}({\bf x})=U (|{\bf x}|)\wh{\psi}({\bf x}),\quad
\lim_{\vep\to 0+}
 e^{-\vep |\xi|}V(|\xi|)\psi(\xi)=V(|\xi|)\psi(\xi) $$
 for all ${\bf x}\in {\bR}\setminus\{0\}, \xi\in{\bR}$,
 it follows from (\ref{5.44}) and Lebesgue dominated convergence that
$$\int_{{\bR}}U (|{\bf x}|)
\wh{\psi}({\bf x}){\rm d}{\bf x}=
 \int_{{\bR}}V(|\xi|)\psi(\xi){\rm d}\xi\qquad \forall\, \psi\in {\cal S}({\bR}).$$
Thus according to the definition of the Fourier transform of generalized functions we conclude
$\wh{U}(\xi)=V(|\xi|). $
 \end{proof}
\vskip2mm

If we choose $V(r)=\fr{1}{1+ r^{\eta}}, 0<\eta<1$, then it is easily checked that
$V$ satisfies the conditions in Lemma \ref{lemma5.10}  with $\dt=\eta/2$.
\\

{\bf Acknowledgments}. We are grateful to a referee for helpful suggestions on the
 presentation of the paper. This work was supported by National Natural Science
Foundation of China Grant No.11771236.

{\bf Note:} This is the second renew version of the original paper which has been published in {\it Journal of Statistical Physics}, {\bf 175}(2019), no. 2, 289-350. In this second renew version, the only main change is that, under the extended Assumption \ref{assp} (i.e. the interval 
$0\le \eta<1$ is extended to the whole interval $0\le \eta<\infty$), the statement and proof 
of Proposition \ref{prop4.2} are designed to be as similar as possible to the original
version.  A typo is corrected: in page 330 of the original paper mentioned above, ``applying Lemma 2.3 in \cite{Lu2013}"
is corrected to ``applying Lemma 2.3 in \cite{Lu2016}".
\\

\end{document}